\documentclass[reqno]{amsart}
\usepackage{amsaddr}
\usepackage[english]{babel}
\counterwithin{figure}{section}
\usepackage{csquotes}
\usepackage{bbold}
\usepackage{mathabx}
\usepackage[letterpaper,top=2cm,bottom=2cm,left=3cm,right=3cm,marginparwidth=1.75cm]{geometry}
\usepackage{amsmath, amsfonts}
\usepackage{graphicx}
\usepackage{enumitem}
\usepackage{caption}
\usepackage{subcaption}
\usepackage[colorlinks=true, allcolors=blue]{hyperref}
\usepackage{seqsplit}
\usepackage{mathtools}
\mathtoolsset{showonlyrefs}
\usepackage{booktabs}
\usepackage{amsthm}
\newtheorem{theorem}{Theorem}[section]
\newtheorem{definition}{Definition}[section]
\newtheorem{rhp}{Riemann-Hilbert problem}[section]
\newtheorem{lemma}{Lemma}[section]
\newtheorem{proposition}{Proposition}[section]

\newtheorem{remark}{Remark}[section]
\newcommand{\Cy}{\mathcal{C}}
\newcommand{\re}{\mathrm{Re}}
\newcommand{\im}{\mathrm{Im}}
\newcommand{\ii}{\mathrm{i}}
\newcommand{\ee}{\mathrm{e}}
\newcommand{\C}{\mathbb{C}}
\newcommand{\dd}{\mathrm{d}}
\newcommand{\R}{\mathbb{R}}

\usepackage[giveninits=true,uniquename=false,uniquelist=false,style=alphabetic,sortcites=false,natbib=true,backend=biber,datamodel=mrnumber, maxbibnames=99,maxnames=99,maxalphanames=99]{biblatex}
\DeclareFieldFormat{mrnumber}{%
  MR\addcolon\space
  \ifhyperref
    {\href{http://www.ams.org/mathscinet-getitem?mr=#1}{\nolinkurl{#1}}}
    {\nolinkurl{#1}}}

\renewbibmacro*{doi+eprint+url}{%
  \iftoggle{bbx:doi}
    {\printfield{doi}}
    {}%
  \newunit\newblock
  \printfield{mrnumber}%
  \newunit\newblock
  \iftoggle{bbx:eprint}
    {\usebibmacro{eprint}}
    {}%
  \newunit\newblock
  \iftoggle{bbx:url}
    {\usebibmacro{url+urldate}}
    {}}

\title[Asymptotic properties of special function solutions
of the Painlev\'e III equation]{Asymptotic properties of special function solutions
of the Painlev\'e III equation for fixed parameters.}
\author{Hao Pan}
\address[Hao Pan]{Department of Mathematics, University of Michigan, East Hall, 530 Church St., Ann Arbor, MI 48109, USA.\\ Email: \href{mailto:haopan@umich.edu}{haopan@umich.edu} }
\author{Andrei Prokhorov}
\address[Andrei Prokhorov]{Department of Mathematics, University of Michigan, East Hall, 530 Church St., Ann Arbor, MI 48109, USA, and\\
St. Petersburg University, 7/9 Universitetskaya nab. 199034 St. Petersburg, Russia, and \\Department of Statistics, University of Chicago, 5747 South Ellis Avenue, Chicago, IL 48109, USA \\Email: \href{mailto:andreip@umich.edu}{andreip@uchicago.edu}}
\date{}
\begin{document}
\begin{abstract}
In this paper, we compute the small and large $x$ asymptotics of the special function solutions of the Painlev\'e-III equation in the complex plane. We use the representation in terms of Toeplitz determinants of Bessel functions obtained in \cite{masuda}. Toeplitz determinants are rewritten as multiple contour integrals using Andr\`eief's identity. The small and large $x$ asymptotics are obtained using elementary asymptotic methods applied to the multiple contour integral. The asymptotics is extended to the whole complex plane using analytic continuation formulas for Bessel functions. The claimed result has not appeared in the literature before. We note that the Toeplitz determinant representation is useful for numerical computations of corresponding solutions of the Painlev\'e-III equation.

\noindent\textit{Key words and phrases: Toeplitz determinant, Painlev\'e equations, Bessel function, asymptotic analysis, integrable systems, ordinary differential equations}
\end{abstract}
\maketitle
\vspace{-1cm}
\tableofcontents
\section{Introduction}
Painlev\'e equations are six {nonlinear second-order ordinary differential equations}. They are written in the form of $u''=R(u',u,t)$ with $R$ a rational function. Their solutions have the so-called {Painlev\'e property}. This means that the locations of singularities of branching type in the complex plane do not depend on the initial conditions, but the locations of isolated singularities might depend on the initial conditions. They were discovered at the beginning of the 20th century in the works \cite{Painleve, Gambier}, see also \cite{Ince}. The solutions of the Painlev\'e equations are classified into three groups: rational and algebraic solutions, solutions expressed in terms of classical special functions, and the rest, see \cite{umemura_watanabe_p2}. For generic values of parameters, the solutions are not rational or algebraic and cannot be reduced to classical special functions, so they belong to the third class and are called Painlev\'e transcendents. We are interested in solutions of the Painlev\'e III equation expressed in terms of Bessel functions (see \cite{Umemura_Watanabe, Murata}). For the applications of such special function solutions of the Painlev\'e III equation in random matrix theory, we refer the reader to \cite{FW03},\cite{forrester_witte_2006}, \cite{ZCL}.

We start with presenting the Painlev\'e III equation
\begin{align}
\label{eq:PIII}
u''=\dfrac{\left( u'\right)^2}{u}-\dfrac{u'}{x}  + \dfrac{\alpha u^2 + \beta}{x}+u^3-\frac{1}{u}, \quad \alpha,\beta \in  \mathbb{C}.
\end{align}
Consider the Toeplitz determinant of cylinder functions:
\begin{equation}\label{def:hankel_bessel_determinant}
\Delta_n(x,\alpha)=\det\left(\left\{\Cy_{\frac{\alpha}{2}-j+k}(x)\right\}_{j,k=0}^{n-1}\right), \quad n\in\mathbb{N}
\end{equation}
with
\begin{equation}\label{def:f_nu}
         \Cy_{\nu}(x)=d_1J_\nu(x)+d_2Y_{\nu}(x),\quad d_1, d_2\in\mathbb{C},
     \end{equation}
and $J_{\nu}(x)$, $Y_\nu(x)$ are Bessel functions of first and second kinds. In addition, denote $\Delta_0(x,\alpha)=1$. 
\begin{proposition}\label{prop:q_n_formula}
The expression 
\begin{equation}\label{eq:q_n_formula}
u_n(x,\alpha)=-\frac{\Delta_{n+1}(x,\alpha-2)\Delta_n(x,\alpha)}{\Delta_{n+1}(x,\alpha)\Delta_n(x,\alpha-2)},\quad n\in \mathbb{N}\cup \{0\},\quad \alpha\in\mathbb{C}
\end{equation}
solves the Painlev\'e III equation with shifted parameters
\begin{equation}\label{eq:q_n_painleve_equation}
u_n''=\dfrac{\left( u_n'\right)^2}{u_n}-\dfrac{u_n'}{x}  + \dfrac{(\alpha+2n)u_n^2 + (-\alpha+2+2n)}{x}+u_n^3-\frac{1}{u_n}.
\end{equation}
\end{proposition}
The fact that the Toeplitz determinants of Bessel functions are related to the solutions of the Painlev\'e III equation is well known; see, for example, \cite{C23}, \cite[(3.5)]{okamoto}. The formula \cite[(4.20)]{masuda} is very similar to \eqref{eq:q_n_formula}, but it involves the Wronskian matrix instead of the Toeplitz matrix. It is not difficult to reduce one to another; see \cite{FW03} or Proposition \ref{lem:tau-formula}. The main advantage of \eqref{eq:q_n_formula} compared to the classical formula \cite[(3.5)]{okamoto} is the absence of a derivative operation applied to the corresponding determinants. A similar formula for the case of rational solutions can be found in \cite{clarkson2018constructive}. Moreover, \cite{masuda} contains Wronskian formulas for special function solutions of Painlev\'e-IV, V, and VI without derivatives. For convenience of the reader, we present the proof of Proposition \ref{prop:q_n_formula} in Section \ref{sec:prop_1_proof}.

Now we are ready to present the first result of our asymptotic analysis.
\begin{theorem}\label{thm:hankel_bessel_det_asymptotic}
 The Toeplitz determinant  \eqref{def:hankel_bessel_determinant} admits the following  $x \rightarrow 0$, $-\pi<\arg(x)<\pi$ asymptotics for fixed $d_1,d_2\in \mathbb{C}$, $n\in\mathbb{N}\cup\{0\}$, $\alpha\in\mathbb{C}\setminus (2\mathbb{Z})$.
\begin{enumerate}
    \item If $d_2\neq 0$ and $Re(\alpha)>2n-2 \text{ or }d_1\sin\left(\frac{\pi\alpha}{2}\right)+d_2\cos\left(\frac{\pi\alpha}{2}\right)=0$, then $$\Delta_n(x,\alpha)\sim (-1)^{\frac{n(n+1)}{2}}\left(\frac{d_2}{\pi}\right)^{n}\frac{G(n+1)G(\frac{\alpha}{2}+1)}{G(\frac{\alpha}{2}-n+1)}\left(\frac{x}{2}\right)^{-\frac{n\alpha}{2}}, \text{ as }x\to 0,\quad  -\pi<\arg(x)<\pi.$$
    \item If $d_2\neq 0$, $d_1\sin\left(\frac{\pi\alpha}{2}\right)+d_2\cos\left(\frac{\pi\alpha}{2}\right)\neq 0$,  and $2n-4j-2<Re(\alpha)<2n-4j+2\\\text{ for some }j=1, 2 \ldots,n-1$, then \begin{align*}&\Delta_n(x,\alpha)\sim (-1)^{\frac{n(n+1)}{2}+j(n-1)}\left( \frac{d_2}{\pi}\right)^n\left(\frac{d_1}{d_2}\sin\left(\frac{\pi\alpha}{2}\right)+\cos\left(\frac{\pi\alpha}{2}\right)\right)^j\\&\times\frac{G(n-j+1)G(-\frac{\alpha}{2}+n-j+1)G(j+1)G(j+\frac{\alpha}{2}+1)}{G(-\frac{\alpha}{2}+n-2j+1)G(\frac{\alpha}{2}-n+2j+1)}\left(\frac{x}{2}\right)^{(\alpha-2n+2j)j-\frac{n\alpha}{2}}, \\&\text{ as }x\to 0,\quad -\pi<\arg(x)<\pi.\end{align*}
    \item If $d_1\sin\left(\frac{\pi\alpha}{2}\right)+d_2\cos\left(\frac{\pi\alpha}{2}\right)\neq 0$, and $Re(\alpha)<-2n+2\text{ or }d_2=0$, then \begin{align}
    &\Delta_n(x,\alpha)\sim \frac{(-1)^{\frac{n(n+1)}{2}}}{\pi^n}\left(d_1\sin\left(\frac{\pi\alpha}{2}\right)+d_2\cos\left(\frac{\pi\alpha}{2}\right)\right)^{n}\frac{G(n+1)G(-\frac{\alpha}{2}+1)}{G(-\frac{\alpha}{2}-n+1)}\left(\frac{x}{2}\right)^{\frac{n\alpha}{2}}, \\&\text{ as }x\to 0,\, -\pi<\arg(x)<\pi.\end{align}
\end{enumerate}
\normalsize
where $G(x)$ refers to the Barnes $G$-function.\\
\end{theorem}
For discussion of the necessity of condition $\alpha\in\mathbb{C}\setminus (2\mathbb{Z}+\ii\mathbb{R})$, see Remarks \ref{rem:complex_alpha}, \ref{rem:integer_alpha}. We also address it in Appendix \ref{app:special_asymptotics}. 

The asymptotic formulas are obtained after rewriting the Toeplitz determinant \eqref{def:hankel_bessel_determinant} as multiple contour integral using Andr\`eief's identity and performing an elementary asymptotic analysis. The multiple contour integral formula holds only for $\re(x)>0$. However, we know that the power series for cylindrical functions is valid for $-\pi<\arg(x)<\pi$. Since our computation can be interpreted as a calculation of the leading term of the product of many power series, our asymptotic result also holds for $-\pi<\arg(x)<\pi$. Notice that plugging naively the asymptotic of Bessel functions into \eqref{eq:q_n_formula} and trying to derive Theorem \ref{thm:hankel_bessel_det_asymptotic} is a difficult task. One would have to reproduce different leading behaviors for different $\alpha$ and the mechanism for it is unclear to us.

We should mention that the same strategy was applied to special function solutions of the Painlev\'e-II equation in \cite{D} and to Hankel determinant solutions of Painlev\'e-VI in \cite{Chen_Zhang}.

Combining Theorem \ref{thm:hankel_bessel_det_asymptotic} with Proposition \ref{prop:q_n_formula} we derive the asymptotics of $u_n(x,\alpha)$ as $x\rightarrow 0$.
\begin{theorem}\label{thm:q_n_asymptotic} Solution \eqref{eq:q_n_formula} of the Painlev\'e-III equation \eqref{eq:q_n_painleve_equation}  admits the following  $x \rightarrow 0$, $-\pi<\arg(x)<\pi$ asymptotics for fixed $d_1,d_2\in \mathbb{C}$, $n\in\mathbb{N}\cup\{0\}$, $\alpha\in\mathbb{C}\setminus (2\mathbb{Z})$
\begin{enumerate}
    \item If $d_2\neq 0$ and $Re(\alpha)>2n+2 \text{ or }d_1\sin\left(\frac{\pi\alpha}{2}\right)+d_2\cos\left(\frac{\pi\alpha}{2}\right)=0$,  then $$u_n(x,\alpha)\sim \left(\frac{2}{2n+2-\alpha}\right)\frac{x}{2}, \text{ as }x\to 0,\quad  -\pi<\arg(x)<\pi.$$
    \item If $d_2\neq 0$, $d_1\sin\left(\frac{\pi\alpha}{2}\right)+d_2\cos\left(\frac{\pi\alpha}{2}\right)\neq 0$,  and $2n-4j<Re(\alpha)<2n-4j+2 \text{ for some }\\j=0, 1,\ldots,n$, then \begin{align*}u_n(x,\alpha)\sim &  (-1)^{n}\left(\frac{d_1}{d_2}\sin\left(\frac{\pi\alpha}{2}\right)+\cos\left(\frac{\pi\alpha}{2}\right)\right)\left(\frac{\Gamma(-\frac{\alpha}{2}+n-2j+1)}{\Gamma(\frac{\alpha}{2}-n+2j)}\right)^2\\& \times\frac{\Gamma(j+\frac{\alpha}{2})\Gamma(j+1)}{\Gamma(-\frac{\alpha}{2}+n-j+1)\Gamma(n-j+1)}\left(\frac{x}{2}\right)^{\alpha-2n+4j-1}, \text{ as }x\to 0,\quad -\pi<\arg(x)<\pi.\end{align*}
    \item If $d_2\neq 0$, $d_1\sin\left(\frac{\pi\alpha}{2}\right)+d_2\cos\left(\frac{\pi\alpha}{2}\right)\neq 0$,  and $2n-4j-2<Re(\alpha)<2n-4j \text{ for some }\\j=0, 1,\ldots,n-1$, then \begin{align*}u_n(x,\alpha)\sim & (-1)^{n} \left(\frac{d_1}{d_2}\sin\left(\frac{\pi\alpha}{2}\right)+\cos\left(\frac{\pi\alpha}{2}\right)\right)^{-1}\left(\frac{\Gamma(\frac{\alpha}{2}-n+2j+1)}{\Gamma(-\frac{\alpha}{2}+n-2j)}\right)^2\\&\times\frac{\Gamma(-\frac{\alpha}{2}+n-j+1)\Gamma(n-j)}{\Gamma(j+\frac{\alpha}{2}+1)\Gamma(j+1)}\left(\frac{x}{2}\right)^{-\alpha+2n-4j-1} ,\text{ as }x\to 0,\quad -\pi<\arg(x)<\pi.\end{align*}
    \item If $d_1\sin\left(\frac{\pi\alpha}{2}\right)+d_2\cos\left(\frac{\pi\alpha}{2}\right)\neq 0$,  and $Re(\alpha)<-2n \text{ or }d_2=0$, then $$u_n(x,\alpha)\sim\left(-\frac{\alpha}{2}-n\right)\left(\frac{x}{2}\right)^{-1}, \text{ as }x\to 0,\quad -\pi<\arg(x)<\pi.$$
\end{enumerate}
where $\Gamma(x)$ refers to the Gamma function.
\end{theorem}
\begin{figure*}[ht]
\begin{subfigure}[t]{0.4\textwidth}
\includegraphics[width=\textwidth]{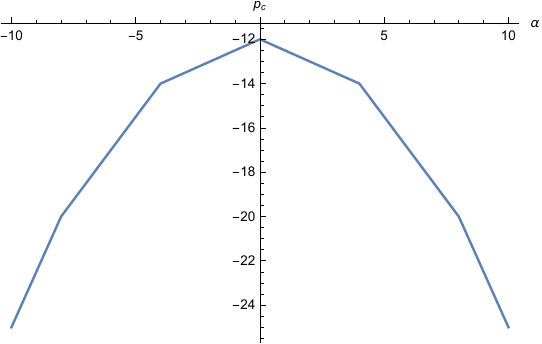}
\caption{Hankel determinant}
        \end{subfigure}
        \begin{subfigure}[t]{0.4\textwidth}
\includegraphics[width=\textwidth]{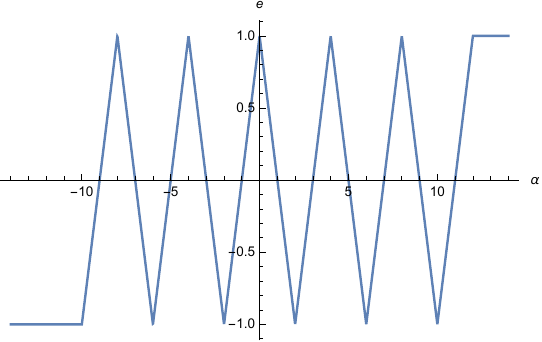}
\caption{Solution of the Painlev\'e-III equation.}
    \end{subfigure}
    \caption{Leading power in the asymptotics for $n=5$ as a function of $\alpha$. The piecewise expressions for $p_c(\alpha,n)$ and $e(\alpha,n)$ can be found in \eqref{eq:pc} and \eqref{eq:e}, respectively.}
        \label{fig:power_illustration}
        \label{fig:continuous_power}\end{figure*}
We can observe that the leading power of asymptotics in Theorems \ref{thm:hankel_bessel_det_asymptotic}, \ref{thm:q_n_asymptotic} is continuous as a function of $\alpha$, see the illustration for $n=5$ in Figure \ref{fig:continuous_power}. That indirectly confirms the validity of our computation. The plot also suggests that the solutions have a qualitatively different behavior for $-2n<\re(\alpha)<2n+2$ and for $\re(\alpha)>2n+2$ or for $\re(\alpha)<-2n$.

Another interesting observation is that in cases $(1)$ and $(4)$ of Theorem \ref{thm:q_n_asymptotic} the leading term of the asymptotic at zero does not fix the solution uniquely. To determine $\frac{d_1}{d_2}$ we would need to compute the subleading term.

We can compare Theorem \ref{thm:q_n_asymptotic} with the small $x$ asymptotics computed based on the monodromy data in \cite[Proposition 1.5]{BLMP}, see also \cite{Jimbo, kitaev87}. More precisely, consider $w_n(x,\alpha)=-\ii u_n(-2\ii x,\alpha)$. It solves the equation
\begin{align}
    \label{eq:v_n_painleve_equation}
w_n''=\dfrac{\left( w_n'\right)^2}{w_n}-\dfrac{w_n'}{x}  + \dfrac{(2\alpha+4n)w_n^2 + (2\alpha-4-4n)}{x}+4w_n^3-\frac{4}{w_n}.
\end{align}
According to \cite{BDM}, solution $w_n(x,\alpha)$ has Riemann--Hilbert representation given by \cite[RHP 4.1]{BLMP} with
\begin{align}\label{eq:monodromy_data}
&\mathbf{C}_{0\infty}^+=\mathbf{C}_{0\infty}^-=\begin{pmatrix}
        1&0\\2b_1&1
    \end{pmatrix},\quad \mathbf{S}_1^{\infty}=\mathbf{S}_1^{0}=\mathbb{1},\quad \Theta_0=\frac{\alpha}{2}+n,\quad \Theta_\infty=n+2-\frac{\alpha}{2},\\
    & \mathbf{S}_2^{\infty}=\begin{pmatrix}
        1&0\\2(b_1-b_2\ee^{\ii\pi\alpha})&1
    \end{pmatrix},\quad \mathbf{S}_2^{0}=\begin{pmatrix}
        1&0\\2\ee^{\ii\pi\alpha}(b_1-b_2)&1
    \end{pmatrix}.\label{eq:monodromy_data2}
\end{align}
where $b_1$ and $b_2$ are given by \eqref{def:b1b2}. The monodromy data corresponding to this solution is given by $e_1^{2}=e_\infty^{-2}=e_0^2=(-1)^n\ee^{\frac{\ii\pi\alpha}{2}}$, $e_2=1$, see \cite[Section 4]{BLMP}. We see that it does not satisfy the conditions \cite[Definition 1.3]{BLMP}. That means that we are filling the gap in the literature regarding the asymptotics of solutions of the Painlev\'e-III equation. For convenience of the reader we provide the derivation of formulas \eqref{eq:monodromy_data}, \eqref{eq:monodromy_data2} in Appendix \ref{app:rhp representation}.

Next, we present our results for large $x$ asymptotics.
\begin{theorem}\label{thm:hankel_bessel_det_asymptotic_at_infinity}
 The Toeplitz determinant \eqref{def:hankel_bessel_determinant} admits the following $x \rightarrow \infty,\,$ asymptotics for fixed $d_1,d_2\in \mathbb{C}$, $n\in\mathbb{N}\cup\{0\}$, $\alpha\in\mathbb{C}$.
\begin{enumerate}
    \item If  $d_1\pm\ii d_2\neq 0$, $n$ is even, and $x>0$, then $$\Delta_n(x,\alpha)\sim \left(\frac{d_1^2+d_2^2}{2\pi}\right)^{\frac{n}{2}}{}{}\left(G\left(\frac{n}{2}+1\right)\right)^2\left(\frac{x}{4}\right)^{-\frac{n^2}{4}}, \text{ as }x\to \infty,\quad x>0.$$
    \item If  $d_1\pm\ii d_2\neq 0$, $n$ is odd, and $x>0$, then \begin{align*}\Delta_n(x,\alpha)\sim & (-1)^{\frac{n-1}{2}}\left(\frac{d_1^2+d_2^2}{2\pi}\right)^{\frac{n}{2}}G\left(\frac{n+1}{2}\right)G\left(\frac{n+3}{2}\right)\sin\left(x-\phi+\frac{\pi}{4}(n-\alpha)\right)\left(\frac{x}{4}\right)^{-\frac{n^2+1}{4}}, \\&\text{as }x\to \infty,\quad x>0.\end{align*}
    \item  \label{thm: toeplitz_infinity_part_3}If  $d_1\pm\ii d_2\neq 0$, and $-\pi<\arg(x)<0$, then  \begin{align}\Delta_n(x,\alpha)\sim&{(d_1-\ii d_2)^{n}}{}\left(\frac{1}{2\pi}\right)^{\frac{n}{2}}\ee^{-\frac{\ii\pi n^2}{4}}G(n+1)\ee^{-\frac{\ii\pi n\alpha }{4}}\ee^{\ii n x}x^{-\frac{n^2}{2}},\quad x\to\infty\end{align}
        \item  If  $d_1\pm\ii d_2\neq 0$, and $0<\arg(x)<\pi$, then  \begin{align}\Delta_n(x,\alpha)\sim&{(d_1+\ii d_2)^{n}}{}\left(\frac{1}{2\pi}\right)^{\frac{n}{2}}\ee^{\frac{\ii\pi n^2}{4}}G(n+1)\ee^{\frac{\ii\pi n\alpha }{4}}\ee^{-\ii n x}x^{-\frac{n^2}{2}},\quad x\to\infty\end{align}
            \item  If  $d_1+\ii d_2=0$, and $-\pi<\arg(x)<\pi$, then  \begin{align}\Delta_n(x,\alpha)\sim&{(d_1-\ii d_2)^{n}}{}\left(\frac{1}{2\pi}\right)^{\frac{n}{2}}\ee^{-\frac{\ii\pi n^2}{4}}G(n+1)\ee^{-\frac{\ii\pi n\alpha }{4}}\ee^{\ii n x}x^{-\frac{n^2}{2}},\quad x\to\infty\end{align}
        \item \label{thm: toeplitz_infinity_part_6} If  $d_1-\ii d_2= 0$, and $-\pi<\arg(x)<\pi$, then  \begin{align}\Delta_n(x,\alpha)\sim&{(d_1+\ii d_2)^{n}}{}\left(\frac{1}{2\pi}\right)^{\frac{n}{2}}\ee^{\frac{\ii\pi n^2}{4}}G(n+1)\ee^{\frac{\ii\pi n\alpha }{4}}\ee^{-\ii n x}x^{-\frac{n^2}{2}},\quad x\to\infty\end{align}
\end{enumerate}
where $G(x)$ refers to the Barnes $G$-function and $\phi=\frac{1}{2\ii}\ln\left({d_1+\ii d_2}\right)-\frac{1}{2\ii}\ln\left({d_1-\ii d_2}\right)$.\\
\end{theorem}
\begin{theorem}\label{thm:q_n_asymptotic_at_infinity} Solution \eqref{eq:q_n_formula} of the Painlev\'e-III equation \eqref{eq:q_n_painleve_equation}  admits the following  $x \rightarrow \infty,\,x>0$ asymptotics for fixed $d_1,d_2\in \mathbb{C}$, $n\in\mathbb{N}\cup\{0\}$, $\alpha\in\mathbb{C}$,
\begin{enumerate}
    \item If $d_1\pm\ii d_2\neq 0$ and $n$ is even, then for some $M>0$ $$u_n(x,\alpha)\sim -\cot\left(x-\phi+\frac{\pi}{4}(n+1-\alpha)\right), \text{ as }x\to \infty,\quad x>0,\mbox{ and } |\cot(x-\phi+\frac{\pi}{4}(n+1-\alpha))|<M.$$
    \item If $d_1\pm\ii d_2\neq 0$ and $n$ is odd, then for some $M>0$ $$u_n(x,\alpha)\sim -\tan\left(x-\phi+\frac{\pi}{4}(n-\alpha)\right), \text{ as }x\to \infty,\quad x>0,\mbox{ and } |\tan(x-\phi+\frac{\pi}{4}(n-\alpha))|<M.$$
    \item If $d_1\pm\ii d_2\neq 0$ and $\im(x)\neq 0$, then we have 
    \begin{align}
  u_n(x,\alpha)-\ii\sim  \left(\frac{d_1-\ii d_2}{d_1+\ii d_2}\right)\frac{2^{2n-1}}{(n-1)!}\ee^{-\frac{\ii\pi}{2}(n+1+\alpha)}x^{n-1}\ee^{2\ii x },\quad x\to \infty,\quad 0<\arg (x)<\pi,\\
        u_n(x,\alpha)+\ii\sim  \left(\frac{d_1+\ii d_2}{d_1-\ii d_2}\right)\frac{2^{2n-1}}{(n-1)!}\ee^{\frac{\ii\pi}{2}(n-1+\alpha)}x^{n-1}\ee^{-2\ii x },\quad x\to \infty,\quad -\pi<\arg (x)<0.
    \end{align}
     \item If $d_1+\ii d_2= 0$ and then the asymptotics holds in larger domain $$u_n(x,\alpha)+\ii\sim \frac{1-\alpha}{2x},\quad x\to\infty ,\quad -\pi<\arg (x)<\pi$$
     Similarly for $d_1-\ii d_2= 0$ we get $$u_n(x,\alpha)-\ii\sim \frac{1-\alpha}{2x},\quad x\to\infty ,\quad -\pi<\arg (x)<\pi.$$
\end{enumerate}

\end{theorem}

Again we refer to the fact that the asymptotics of Hankel functions is valid for $-\pi<\arg(x)<\pi$ to extend our result from domain $\re(x)>0$ with available contour integral representation to the plane with a cut.  To get the asymptotic for $\arg(x)=\pm \pi$ one can use the analytic continuation formulas for cylindrical functions , see Appendix \ref{sec:analytic_continuation}. More precisely, one could replace $d_1, d_2$ with $d_1^\pm$, $d_2^\pm$ given by \eqref{eq:d+}, \eqref{eq:d-} and $x$ with $\ee^{\pm\ii\pi }x$ and directly use the result for $x>0$. We can notice that $d_1^{\pm}$ and $d_2^{\pm}$ start depending on $\alpha$ in that case, but that change does not affect our computation. This strategy also provides an alternative justification for our result in the domain $\re(x)<0$, where contour integral representation does not hold.

We should also mention that for $\alpha-1\in 2\mathbb{Z}$ and $d_1\pm\ii d_2=0$ the special function solutions reduce to the rational solutions considered in \cite{BMS18}. It can be seen from the asymptotic formulas \eqref{eq:hankel1_asym} and \eqref{eq:hankel2_asym} that truncate for described choice of $\alpha$.


The Proposition \ref{prop:q_n_formula} is useful for the numerical calculation of the solution $u_n(x,\alpha)$ through direct evaluation of the determinants. We also present a color plot for the argument of $u_n(x,\alpha)$ in the complex plane for various choices of $\alpha$ and $n$ in Figure \ref{fig:complex_illustration}.
We present the result for fixed $\alpha$ and large $n$ in Figure \ref{illustration}. We observe that the pole structure is similar to the pole structure of rational solutions of the Painlev\'e III equation observed in \cite{BMS18}, but the poles now also lie in the regions extending to infinity. The special case with $d_2=0$ can be found in Figure \ref{illustration2}, and the case with $d_1+\ii d_2=0$ can be found in Figure \ref{illustrationm1}. The other case of large $\alpha$ and large $n$ can be seen in Figure \ref{illustration3}. We see that the pole structure looks different in this case. Finally, in Figure \ref{illustration4}, we can see what happens when we take $\alpha$ much larger than $n$. The analysis of such pictures would require tools like the nonlinear steepest descent method for Riemann-Hilbert problems. The Jupyter notebook with presented plots can be found at \url{https://github.com/andrei-prokhorov/special-function-solutions-of-PIII.git}.
\begin{figure*}[ht]
\begin{subfigure}[t]{0.49\textwidth}
\centering\includegraphics[width=\textwidth]{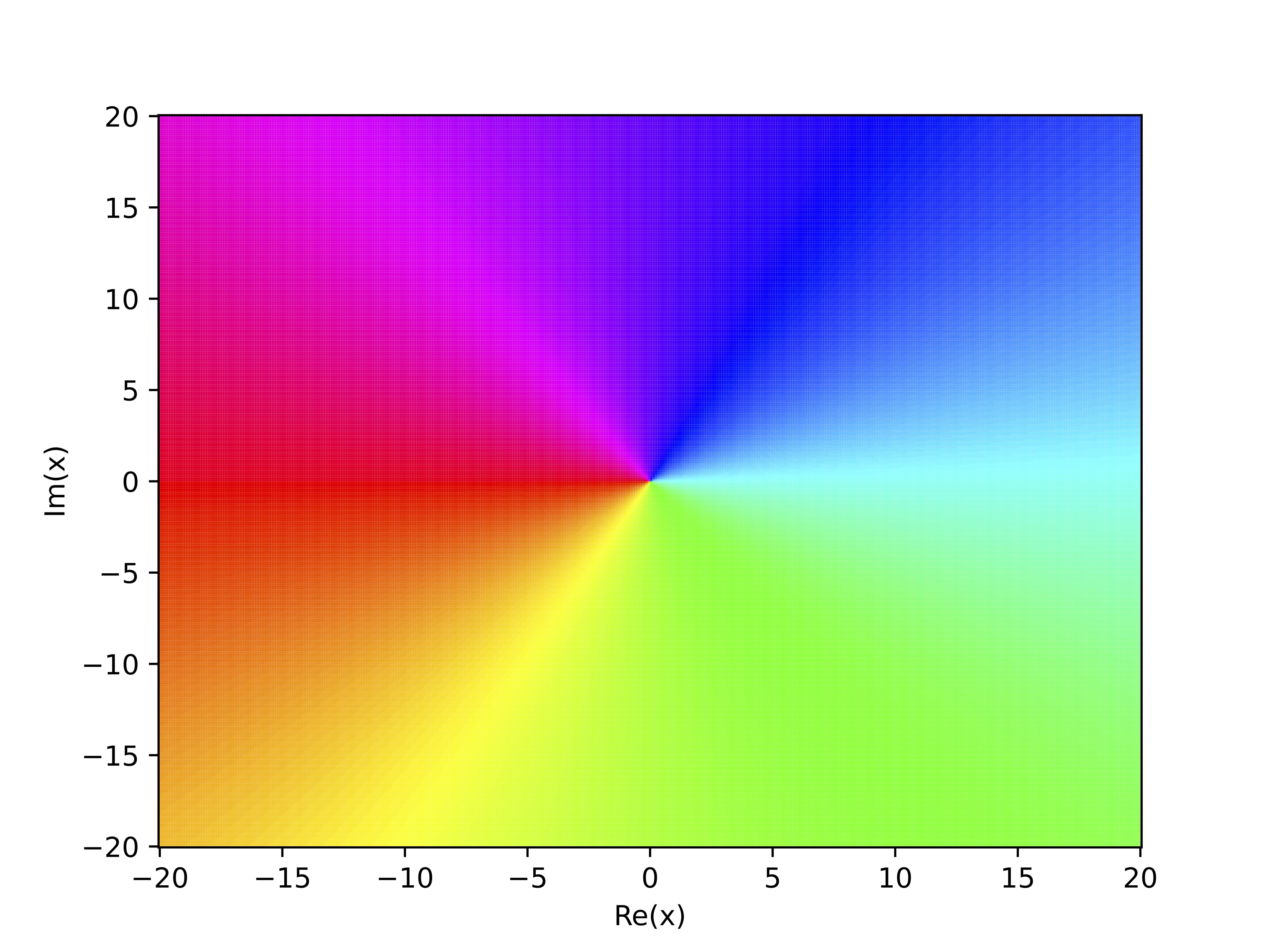}
\caption{Phase plot for the case $f(z)=z$.}\label{illustration0} 
     \end{subfigure}\hspace{0.1cm}
        \begin{subfigure}[t]{0.49\textwidth}
\includegraphics[width=\textwidth]{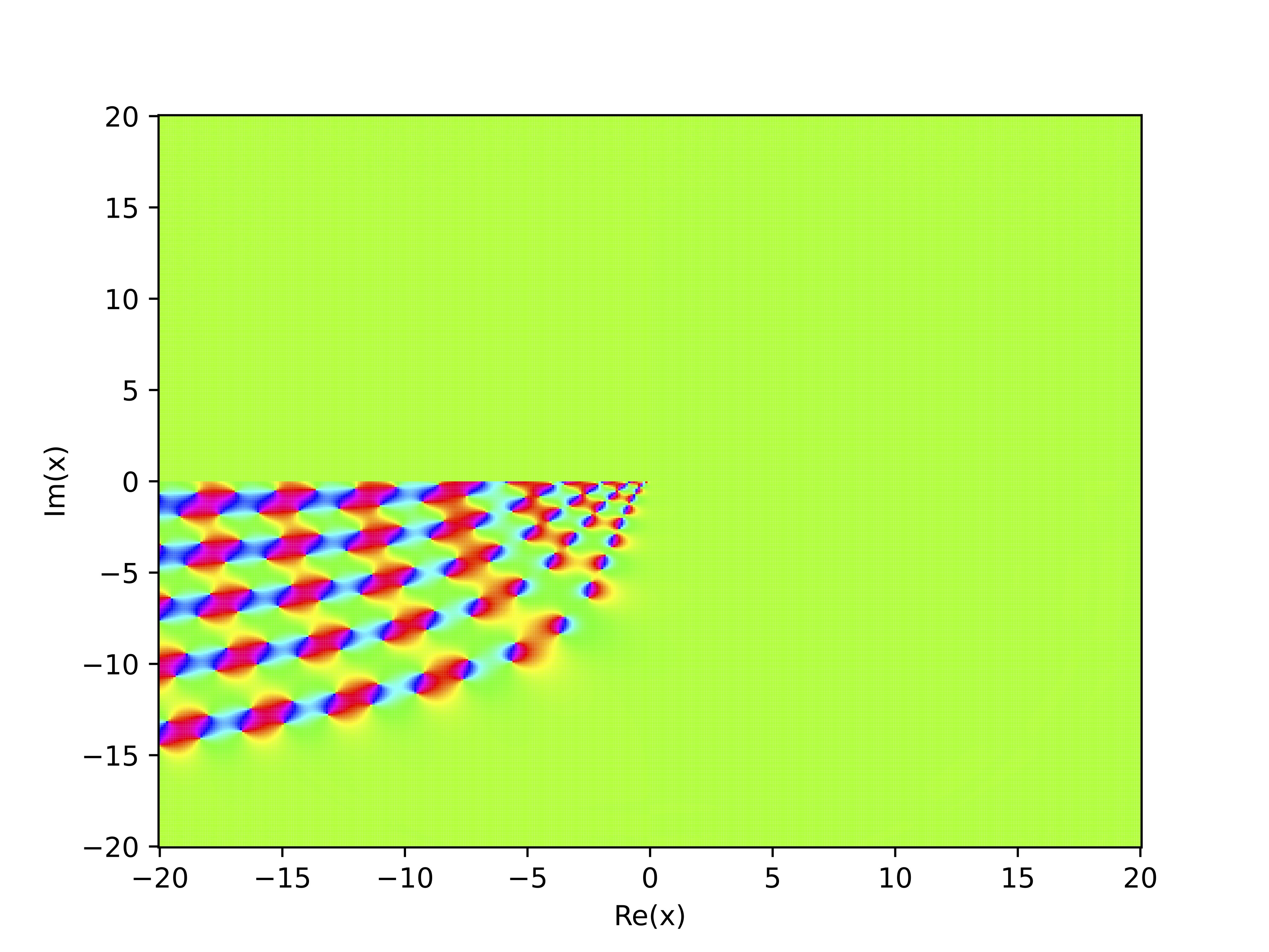}
\caption{Complex phase plot of solution $u_n(x,\alpha)$, for $n=10$ and $\alpha=1.08$, $d_1=0.55$, $d_2=0.55\ii$.}
      \label{illustrationm1}  \end{subfigure}\\
\begin{subfigure}[t]{0.49\textwidth}
\includegraphics[width=\textwidth]{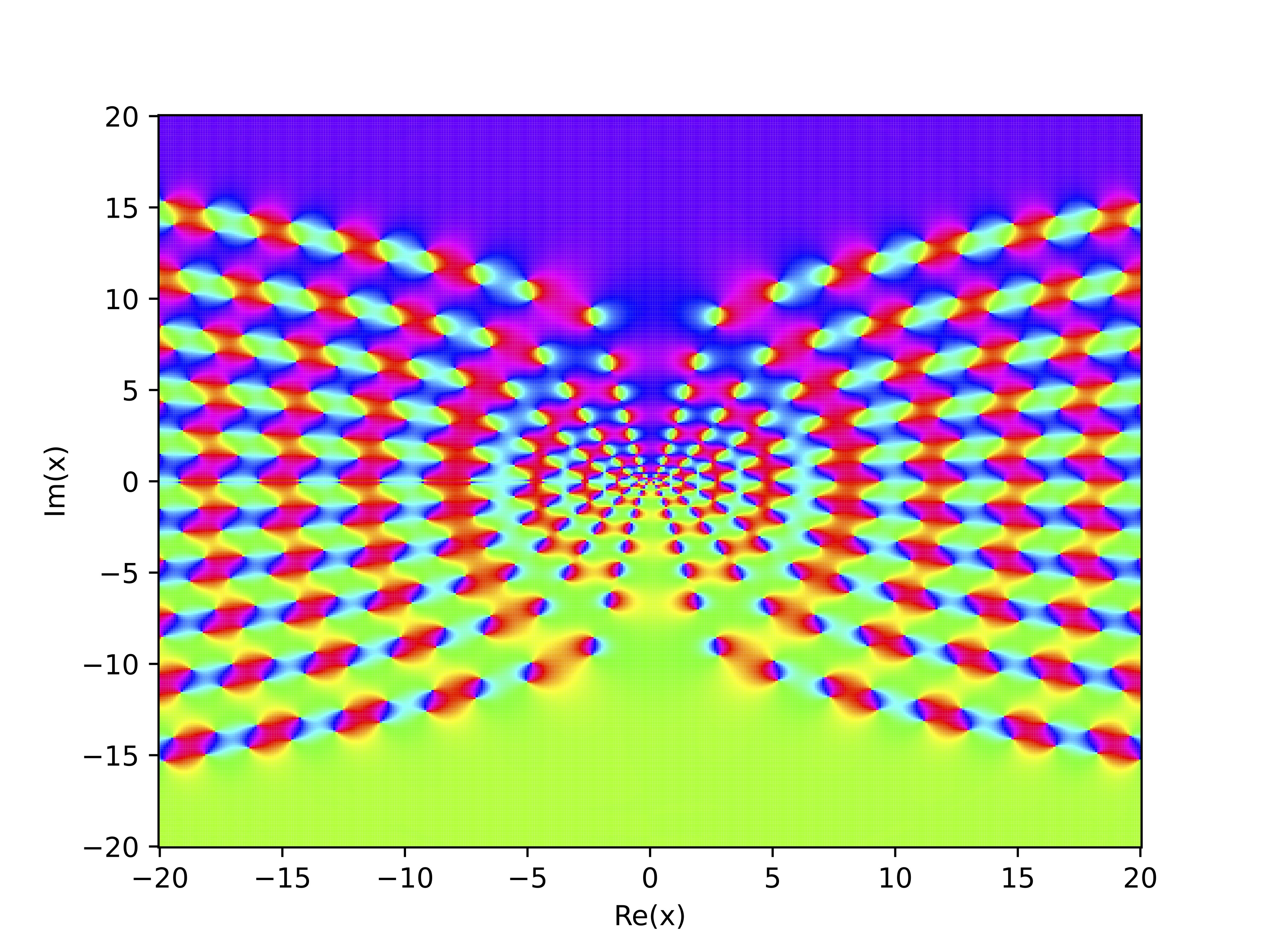}
\caption{Complex phase plot of solution $u_n(x,\alpha)$, for $n=10$ and $\alpha=1.08$, $d_1=0.55$, $d_2=0.71$.}
      \label{illustration}  \end{subfigure}\hspace{0.1cm}
        \begin{subfigure}[t]{0.49\textwidth}
\includegraphics[width=\textwidth]{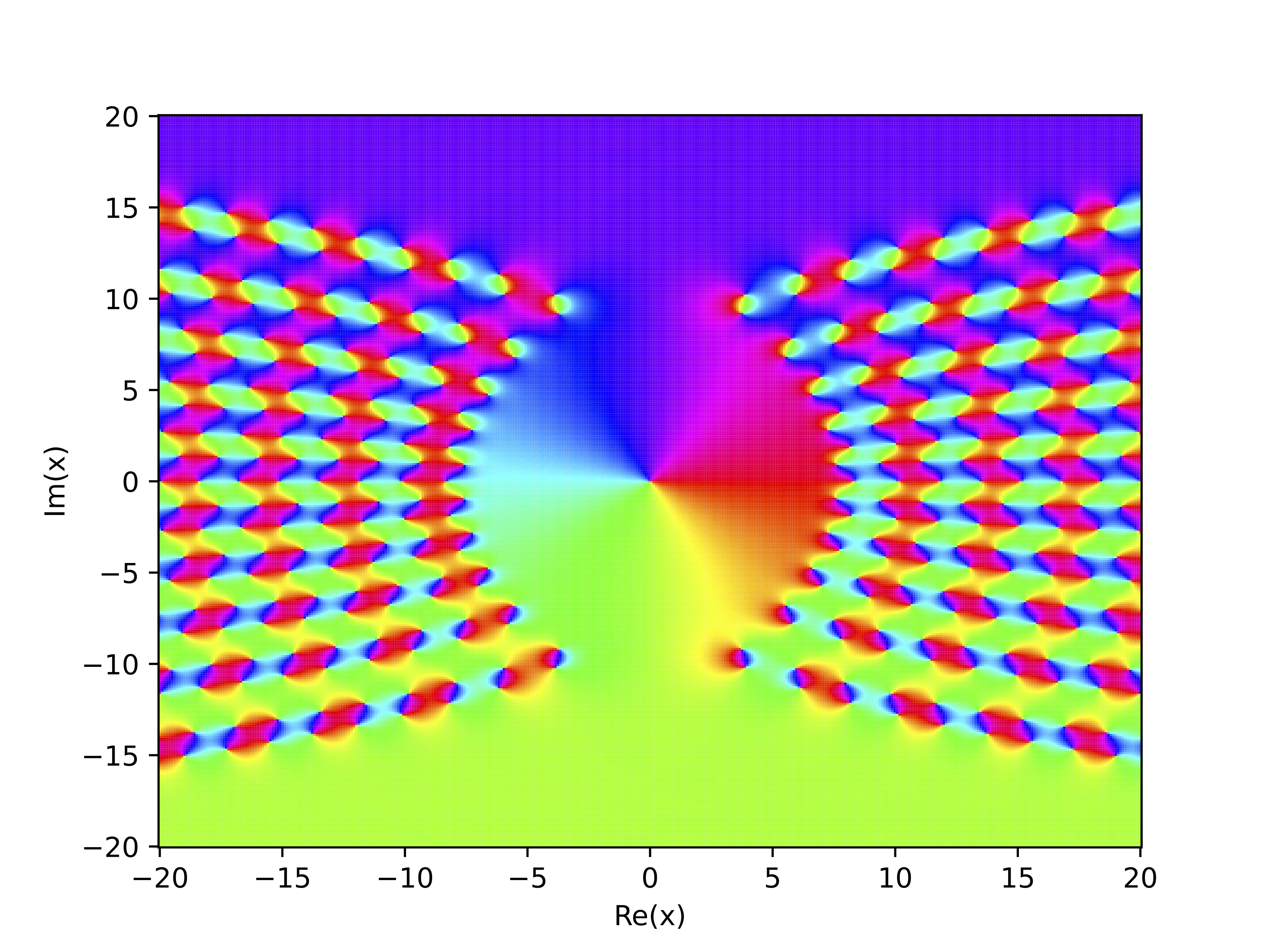}
\caption{Complex phase plot of solution $u_n(x,\alpha)$, for $n=10$ and $\alpha=1.08$, $d_1=0.55$, $d_2=0$.}
      \label{illustration2}  \end{subfigure}\\
        \begin{subfigure}[t]{0.49\textwidth}
\includegraphics[width=\textwidth]{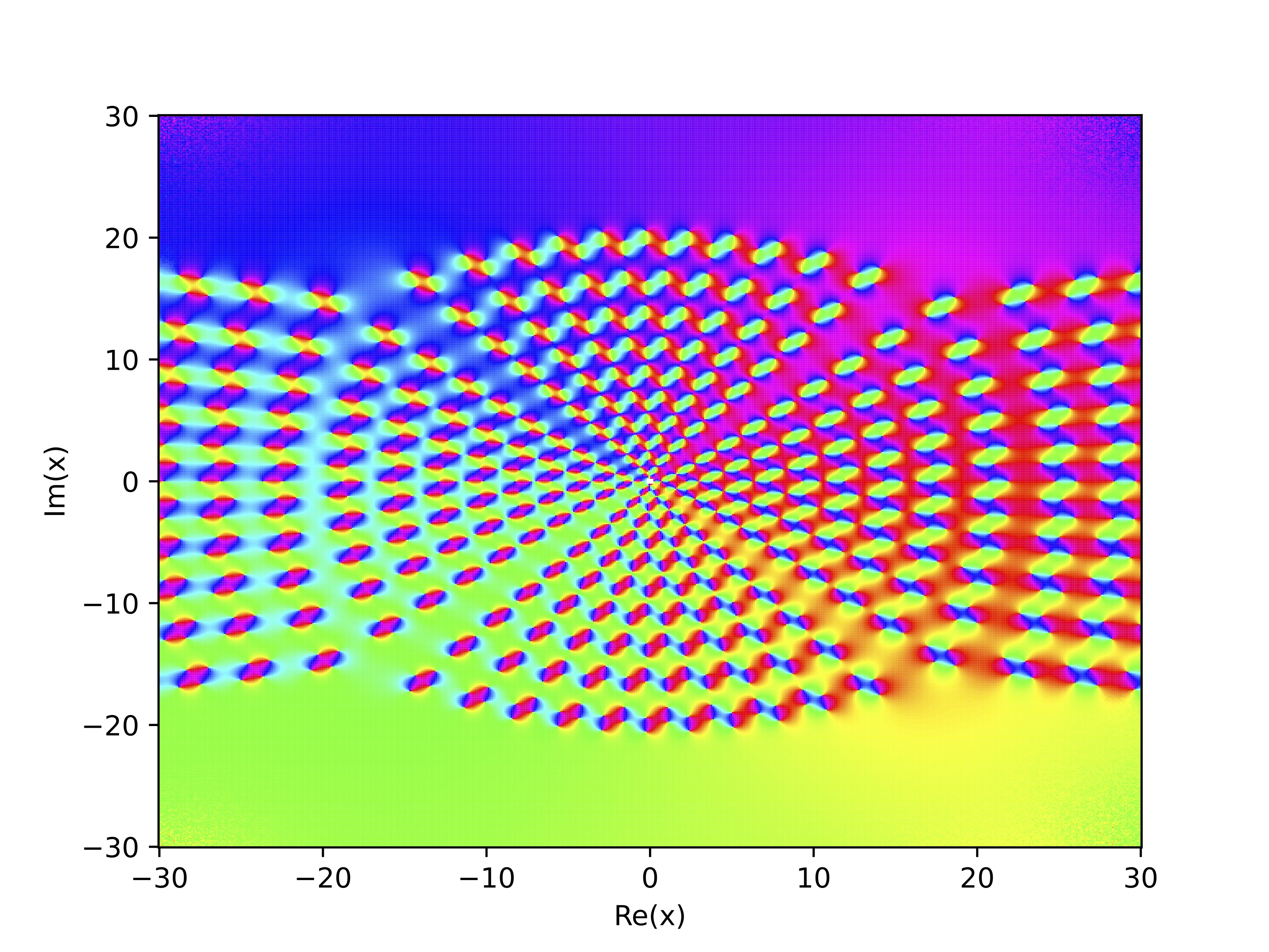}
\caption{Complex phase plot of solution $u_n(x,\alpha)$, for $n=10$ and $\alpha=23.04$, $d_1=0.55$, $d_2=0.71$.}
   \label{illustration3} \end{subfigure}\hspace{0.1cm}
   \begin{subfigure}[t]{0.49\textwidth}
\includegraphics[width=\textwidth]{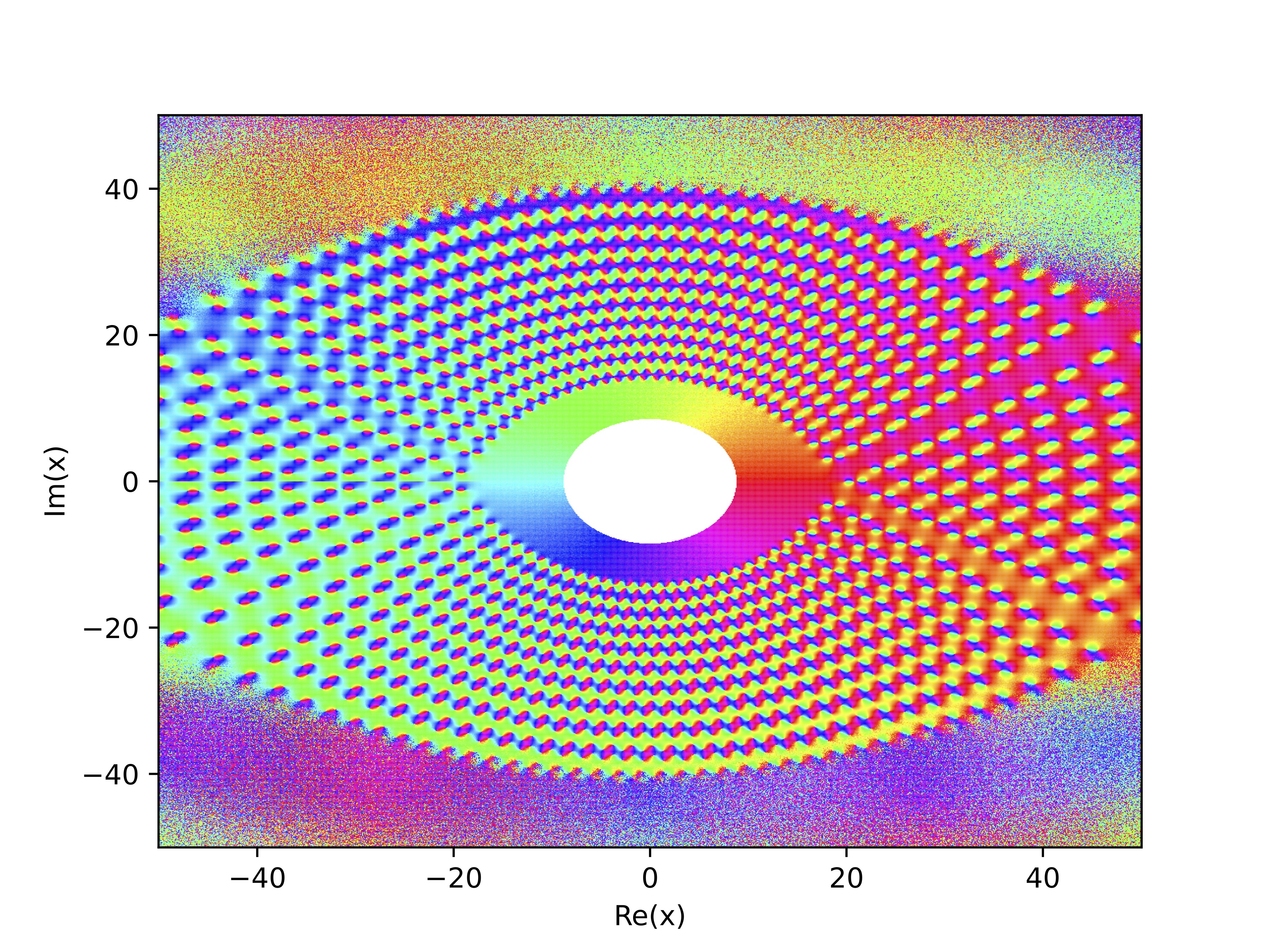}
\caption{Complex phase plot of solution $u_n(x,\alpha)$, for $n=10$ and $\alpha=80.04$, $d_1=0.55$, $d_2=0.71$.}
   \label{illustration4} \end{subfigure}
    \caption{Complex argument plots of solutions for various values of parameters. The color for each value of the argument can be found in Figure \ref{illustration0}}.
    \label{fig:complex_illustration}
        \end{figure*}

\subsection{Overview of the paper}
We start with the reminder of the construction of a special function solution using the Ricatti equation in Section \ref{sec:ricatti} and the generation of the family of special function solutions using the B\"acklund transformation in Section \ref{sec:backlund}. 

We prove representation for the special function solutions from Proposition \ref{prop:q_n_formula} in Section \ref{sec:prop_1_proof}. We start by introducing the tau function and demonstrating the classical fact that it satisfies the Toda equation in Section \ref{sec:tau_toda}. We use the Deshanot-Jacobi identity and the Toda equation to show the classical Wronskian formula for the tau function associated with the family of special function solutions in Section \ref{sec:linear-algebra}.
The crucial next step is the identification of the Toeplitz determinant \eqref{def:hankel_bessel_determinant} with the Painlev\'e tau function \eqref{def_tau} in the Proposition \ref{lem:tau-formula}, following \cite{FW03}, \cite{okamoto}. The main tools are differential identities \eqref{eq:diff-iden1}, \eqref{eq:diff-iden2}. After a long and tedious computation, we finish the proof of Proposition \ref{prop:q_n_formula}.

We prove Theorem \ref{thm:hankel_bessel_det_asymptotic} in Section \ref{sec:hankel_bessel_det_asymptotic}. We start by rewriting the Toeplitz determinant as a multiple contour integral using the Andr\`eief formula in Section \ref{sec:andreief}. We start with getting asymptotics for $x>0$. The key next step in the proof is splitting the multiple contour integral in the sum of other multiple contour integrals so that it is easy to compute the leading term of the asymptotics for the latter integrals. The result can be found in Lemma \ref{lem:sum_of_integrals}. We determine which integral has the largest leading term in Lemma \ref{prop:piecewise function of r_c}. The next step in the proof of Theorem \ref{thm:hankel_bessel_det_asymptotic} is the evaluation of the multiple contour integral corresponding to the leading term using formulas from \cite{DLMF}. Finally, we extend our  asymptotic formulas to the complex plane with $-\pi<\arg(x)<\pi$ by using the validity of asymptotic series of Bessel functions \eqref{eq:bessel_series} in that range. The alternative way of this extension using analytic continuation formulas is presented in the Appendix \ref{app:alternative_proof_thm_1.1}.

We prove Theorem \ref{thm:q_n_asymptotic} in Section \ref{sec:qn_asymptotics}. It consists of plugging in of result of Theorem \ref{thm:hankel_bessel_det_asymptotic} in Proposition \ref{prop:q_n_formula} and tedious manipulation with piecewise formulas.

We prove Theorem \ref{thm:hankel_bessel_det_asymptotic_at_infinity} in Section \ref{sec:toeplitz_asym_at_infty}. We rewrite the cylinder functions in terms of Hankel functions, which are more convenient for large $x$ asymptotics computation. We again start with the case $x>0$. We apply Andr\`eief identity and use the steepest descent method to determine the final result. The extension to complex plane is provided by the range of validity of asymptotic series of Hankel functions \eqref{eq:hankel1_asym}, \eqref{eq:hankel2_asym}. Alternative method using analytic continuation formulas is presented in Appendix \ref{app:alternative_proof_thm_1.3}. Theorem \ref{thm:q_n_asymptotic_at_infinity} is obtained in Section \ref{sec:solution_asym_at_infty} by combining Theorem \ref{thm:hankel_bessel_det_asymptotic_at_infinity} and Proposition \ref{prop:q_n_formula}. For some cases we compute additional error terms in Theorem \ref{thm:hankel_bessel_det_asymptotic_at_infinity} to get more meaningful expressions in Theorem \ref{thm:q_n_asymptotic_at_infinity}.

 In Appendix \ref{app:bessel} we derive the convenient contour integral representations \eqref{eq:cy_integral_representation}, \eqref{eq:alt_cy_integral_representation} for the cylinder function \eqref{def:f_nu}, present differential identities \eqref{eq:diff-iden1}, \eqref{eq:diff-iden2}, and analytic continuation formulas.

In Appendix \ref{app:rhp representation} we derive Riemann--Hilbert problem representation for solution \eqref{def:qn} and confirm formulas \eqref{eq:monodromy_data}, \eqref{eq:monodromy_data2} for the monodromy data.

In Appendix \ref{app:special_asymptotics} we determine asymptotics at zero for values of $\alpha$ missing from Theorems \ref{thm:hankel_bessel_det_asymptotic}, \ref{thm:q_n_asymptotic}. 
\section{Construction of Bessel function solutions of the Painlev\'e III equation}

\subsection{The simultaneous solutions of Ricatti and Painlev\'e III equations}\label{sec:ricatti}
The standard way to construct the special function solutions of Painlev\'e equations is to use a Ricatti equation, see \cite[Theorem 3.5]{C23} and \cite[\href{http://dlmf.nist.gov/32.10.iii}{\S 32.10(iii)}]{DLMF}. More precisely, we look for the simultaneous solutions of Painlev\'e III equation \eqref{eq:PIII} and the Ricatti equation
    \begin{equation}
\label{eq:ricatti}u'(x)=a(x)u^2(x)+b(x)u(x)+c(x),\quad a(x)\neq 0.
    \end{equation}
Taking the first derivative of \eqref{eq:ricatti} and plugging in the $u'(x)$, we get:
    \begin{equation*}     u''(x)=2a^2(x)u^3(x)+(a'(x)+3a(x)b(x))u^2(x)+(2a(x)c(x)+b^2(x)+b'(x))u(x)+(b(x)c(x)+c'(x))
        \end{equation*}
Meanwhile, plugging \eqref{eq:ricatti} into \eqref{eq:PIII}, we get:        \begin{align*}
            u''(x)=&(a^2(x)+1)u^3(x)+\left((2a(x)b(x)-\frac{a(x)-\alpha}{x}\right)u^2(x)+\left(b^2(x)+2a(x)c(x)-\frac{b(x)}{x}\right)u(x)\\&+\frac{c(x)-1}{u(x)}+2b(x)c(x)-\frac{c(x)-\beta}{x}.
        \end{align*}
    By matching and solving for the coefficients, we have four cases in total. We list them below:
    \begin{align}
        a(x)=1,&& b(x)=\frac{\alpha-1}{x},&& c(x)=1,&& \beta=2-\alpha\label{eq:case1}
\\
         a(x)=-1,&& b(x)=\frac{-1-\alpha}{x},&& c(x)=-1,&& \beta=-2-\alpha
 \label{eq:case2}\\
        a(x)=1,&& b(x)=\frac{\alpha-1}{x},&& c(x)=-1,&& \beta=\alpha-2\label{eq:case3}
\\
         a(x)=-1,&& b(x)=\frac{-1-\alpha}{x},&& c(x)=1,&& \beta=\alpha+2\label{eq:case4}
    \end{align}
    Notice that if $u(x)$ solves the Ricatti equation, then
$w(x)=\exp\left(-\int a(x)u(x)dx\right)$
    solves the following linear ODE:
    \begin{equation}
\label{eq:linear_ricatti1}    a(x)w''(x)-(a'(x)+a(x)b(x))w'(x)+c(x)a^{2}(x)w(x)=0,\quad a(x)\neq 0.
\end{equation}
   From now on, we will only focus on the case \eqref{eq:case1}.  Equation \eqref{eq:linear_ricatti1}  becomes     \begin{equation}\label{eq:linear_ricatti2}
         w{''}(x)+\frac{(1-\alpha)}{x}w{'}(x)+w(x)=0.
    \end{equation}
We can notice that $x^{-\frac{\alpha}{2}} w(x)$ solves the Bessel equation \eqref{eq:bessel} with $\nu=\frac{\alpha}{2}$. For $\alpha\in\mathbb{C}\setminus (2\mathbb{Z}+\ii\mathbb{R})$ we denote the solution of \eqref{eq:linear_ricatti2}
\begin{equation}\label{def:w}
w(x,\alpha)=x^{\frac{\alpha}{2}}\Cy_{\frac{\alpha}{2}}(x),
    \end{equation}
    where cylinder function $\Cy_{\nu}(x)$ is given by \eqref{def:f_nu}.
    Here we assume that $\frac{\alpha}{2}$ is not an integer for convenience in our future computations.
    As a result, we get the following.
    \begin{proposition}[\text{\cite[\href{http://dlmf.nist.gov/32.10.iii}{\S 32.10(iii)}]{DLMF}}]\label{prop:u_0}
Painlev\'e III equation \eqref{eq:PIII} with $\beta=2-\alpha$ and $\alpha\in\mathbb{C}\setminus (2\mathbb{Z}+\ii\mathbb{R})$ admits the special function solution
\begin{equation}\label{def:u_0}
         u_0(x,\alpha)=-\frac{d}{dx}\ln(w(x,\alpha)).
    \end{equation}
    with $w(x,\alpha)$ given by \eqref{def:w}.
    \end{proposition}

\begin{remark}
    In case \eqref{eq:case2}, the relevant solution is also given in terms of Bessel functions, while in the cases \eqref{eq:case3}, \eqref{eq:case4} it is given in terms of modified Bessel functions.
\end{remark}
\subsection{B\"acklund transformation}\label{sec:backlund}
To construct more solutions for the PIII equation with more general parameters, we need to introduce a powerful tool. B\"acklund transformations for the Painlev\'e-III equation
        are given by (see \cite[\href{http://dlmf.nist.gov/32.7.iii}{\S 32.7(iii)}]{DLMF})
        \begin{equation}\label{eq:backlund1}
            B_1: (u(x),\alpha,\beta)\to \left(\frac{xu'(x)+xu^2(x)-\beta u(x)-u(x)+x}{u(x)(xu'(x)+xu^2(x)+\alpha u(x)+u(x)+x)},\alpha+2,\beta+2\right)
        \end{equation}
        \begin{equation}\label{eq:backlund2}
            B_2: (u(x),\alpha,\beta)\to \left(-\frac{xu'(x)-xu^2(x)-\beta u(x)-u(x)+x}{u(x)(xu'(x)-xu^2(x)-\alpha u(x)+u(x)+x)},\alpha-2,\beta+2\right)
        \end{equation}
        
They are used as follows. Assume that $u(x)$ solves the Painlev\'e-III equation \eqref{eq:q_n_painleve_equation} and denote \\$B_1(u(x),\alpha,\beta)=(W(x),\alpha+2,\beta+2)$. Then $W(x)$ solves Painlev\'e-III equation
\begin{equation*}
            W''(x)=\dfrac{\left( W'(x)\right)^2}{W(x)}-\dfrac{W'(x)}{x}  + \dfrac{(\alpha+2) W^2(x) + (\beta+2)}{x}+W^3(x)-\frac{1}{W(x)}.
        \end{equation*}
Similarly, if we denote $B_2(u(x),\alpha,\beta)=(W(x),\alpha-2,\beta+2)$ then $W(x)$ solves Painlev\'e-III equation
\begin{equation*}
            W''(x)=\dfrac{\left( W'(x)\right)^2}{W(x)}-\dfrac{W'(x)}{x}  + \dfrac{(\alpha-2) W^2(x) + (\beta+2)}{x}+W^3(x)-\frac{1}{W(x)}.
        \end{equation*}

\begin{proposition}[\text{\cite[\href{http://dlmf.nist.gov/32.7.iii}{\S 32.7(iii)}]{DLMF}}]
    Denote $B_1^n(u_0(x,\alpha),\alpha,2-\alpha)=(u_n(x,\alpha),\alpha+2n,-\alpha+2+2n)$ with $u_0(x,\alpha)$ given by \eqref{def:u_0}. Then $u_n(x,\alpha)$ is the special function solution of Painlev\'e III equation \eqref{eq:q_n_painleve_equation}.
\end{proposition}
We can observe that the parameters of the Painlev\'e-III equation \eqref{eq:q_n_painleve_equation} satisfy $\alpha+\beta\in 2+4\mathbb{N}$. We will use B\"acklund tranformations $B_2$. 
\begin{remark} \label{rem:modified_bessel} Using the transformation $u(x)\rightarrow -u(x)$ we can get solutions with  $\alpha+\beta\in-2-4\mathbb{N}$. Using the transformation $u(x)\rightarrow -\ii u(-\ii x)$ we can get solutions with  $\alpha-\beta \in 2+4\mathbb{N}$. 
\end{remark}

\section{Toeplitz determinants of cylinder functions}\label{sec:prop_1_proof}

\subsection{Hamiltonian system} We use the formulas presented in \cite{C23}.
    \begin{definition}
     We define the momentum associated to the solution of Painlev\'e-III equation using formula
        \begin{equation*}
           v(x)=\frac{1}{2u^2(x)}\left(xu'(x)+xu^2(x)- x+u(x)\left({\beta-1}{}\right)\right)
        \end{equation*}
\end{definition}
\begin{definition}
      We define the Hamiltonian associated with the solution of the Painlev\'e-III equation using formula
        \begin{equation*}
           H(x)={v^2(x)}{}u^2(x)-v(x)\left(xu^2(x)-{x}+u(x)\left({\beta-1}\right)\right)+2xu(x)\left(\frac{\beta-(2+\alpha)}{4}\right)
        \end{equation*}
\end{definition}
One can show that Painlev\'e-III equation is equivalent to the {Hamiltonian system}:
    \begin{equation*}
        x\dfrac{du}{dx}=\dfrac{\partial H}{\partial v},
    \end{equation*}
    \begin{equation*}
        \quad x\dfrac{dv}{dx}=-\dfrac{\partial H}{\partial u}.
    \end{equation*}
  
\subsection{Tau function and Toda equation} \label{sec:tau_toda}For details of this Section, see \cite{okamoto} and \cite{FW03}.
\begin{definition}\label{def:aux-ham}
We define the auxiliary Hamiltonian associated with the solution of the Painlev\'e-III equation using formula
        \begin{equation*}
            h(x)=\frac{1}{2}\left(H(x)+u(x)v(x)- x^2+\frac{1}{4}(\beta-4)(\beta+(\alpha-2))\right).
        \end{equation*}
\end{definition}
In this Section we will deal with a generic solution $u(x)$ of \eqref{eq:PIII} . Since momentum, Hamiltonian and auxiliary Hamiltonian are expressed in terms of $u(x)$, the action of the B\"acklund transformation $B_1$ can be extended to them by formulas $(u(x),\alpha,\beta)$ to $v(x)$, $H(x)$ and $h(x)$. We denote 
        \begin{align}
(u_n(x),\alpha+2n,\beta+2n)=B_1^n(u(x),\alpha,\beta)
       \label{def:qn}\\v_n(x)=\left.v(x)\right|_{u(x)\to u_n(x),\beta\to\beta+2n}
 \label{def:pn}   \\H_n(x)=\left.H(x)\right|_{u(x)\to u_n(x),v(x)\to v_n(x),\alpha\to\alpha+2n,\beta\to\beta+2n}
 \label{def:Hn}      \\h_n(x)=\left.h(x)\right|_{H(x)\to H_n(x),u(x)\to u_n(x),v(x)\to v_n(x),\alpha\to\alpha+2n,\beta\to\beta+2n}\label{def:hn}
        \end{align}
On the path to derive the representation of Proposition \ref{prop:q_n_formula} we introduce the tau function associated with the solution $u(x)$.
\begin{definition}
\label{def:tau-fauntion}      The tau function associated to the solution of Painlev\'e-III equation is defined using the formula
\begin{equation}\label{def_tau}
           x\dfrac{d}{dx}\ln(\tau_n(x))=h_n(x).
        \end{equation}
        It is defined up to a multiplicative constant.
\end{definition}

\begin{proposition}[\text{\cite[Proposition 4.2]{FW03}}]
      Tau function for the Painlev\'e-III equation given by \eqref{def_tau} satisfies {Toda equation} 
        \begin{equation*}
           x\dfrac{d}{dx}x\dfrac{d}{dx}\ln(\tau_n(x))=c_n\frac{\tau_{n+1}(x)\tau_{n-1}(x)}{\tau_n^2(x)}
        \end{equation*} 
    for some constants $c_n$. Moreover, multiplicative constants can be chosen in the definition \eqref{def_tau} so that $c_n= 1$.
    \end{proposition}  
    \begin{proof}
    Using the B\"acklund transformation $B_1$ we can check the identity
        \begin{equation}\label{eq:identity_1}
            h_{n+1}(x)=h_n(x)-v_{n}(x)u_{n}(x)-\frac{3}{2}+\frac{\alpha}{4}-\frac{3\beta}{4}+2n
        \end{equation}   
   Denote $A_n(x)=x\dfrac{d}{dx}x\dfrac{d}{dx}\ln(\tau_n(x))$ and $B_n(x)=\frac{\tau_{n+1}(x)\tau_{n-1}(x)}{\tau_n^2(x)}$. We want to show that $A_n(x)=c_nB_n(x)$. Taking a natural log on both sides, we get $\ln A_n(x)=\ln B_n(x)+\ln c_n$. Therefore, it is sufficient to show \begin{equation}\dfrac{d}{dx}(\ln A_n(x)-\ln B_n(x))=0.\label{eq:identity_2}\end{equation}
    Well, using Definition \ref{def:tau-fauntion} and identity \eqref{eq:identity_1}, we have:
    \begin{equation}
        \dfrac{d}{dx}\ln (B_n(x))=\frac{v_{n-1}(x)u_{n-1}(x)-v_n(x)u_n(x)+2}{x},\label{eq:g}
    \end{equation}
     \begin{equation}
        \dfrac{d}{dx}\ln (A_n(x))=\frac{h'_n(x)+xh''_n(x)}{xh'_n(x)}.\label{eq:f} 
    \end{equation}
    Using Definitions \ref{def:aux-ham} and \eqref{def:qn}--\eqref{def:hn}, we rewrite \eqref{eq:f}, \eqref{eq:g} in terms of $u_n(x)$. After a long computation, we obtain \eqref{eq:identity_2}.
    
    Let us show that using transformation $\tau_n(x)\to a_n\tau_n(x)$ one can make sure that the constant $c_n$ in the Toda equation is $1$. We notice that for that to happen $a_n$ has to satisfy a difference equation 
        \begin{equation*}
            c_n^{-1}a_n^2=a_{n+1}a_{n-1}.
        \end{equation*}   
    Its general solution is given by
        \begin{equation*}
            a_n=\frac{a_1^n}{a_0^{n-1}}\prod_{j=1}^{n-1}\prod_{i=1}^jc_i^{-1},\ \ n\in \mathbb{N}.
        \end{equation*} 
    We can pick the initial conditions $a_0=a_1=1$ and choose
        \begin{equation*}
            a_n=\prod_{j=1}^{n-1}\prod_{i=1}^jc_i^{-1},\ \ n\in \mathbb{N}.
        \end{equation*}
    \end{proof}

\subsection{Wronskian solutions of Toda equation}\label{sec:linear-algebra}
Toda equation
 \begin{equation*}
           x\dfrac{d}{dx}x\dfrac{d}{dx}\ln(\tau_n(x))=\frac{\tau_{n+1}(x)\tau_{n-1}(x)}{\tau_n^2(x)}
        \end{equation*} 
determines the tau function recursively given initial conditions. If we want to derive some nice formula for it, we need some properties of the determinants.
     The Leibniz formula for the determinant of $n\times n $ matrix $A=\left\{a_{ij}\right\}_{i,j=1}^n$ is given by
     \begin{equation*}
           \det(A)=\sum_{\sigma\in S_n}\mathrm{sgn}(\sigma)\prod_{k=1}^na_{k,\sigma(k)},
        \end{equation*}
        where $S_n$ is the set of permutations of $n$ elements and $\mathrm{sgn}(\sigma)$ is a sign of permutation $\sigma$.
Directly using the Leibniz formula above, we can show the following formulas for the derivative of a determinant
\begin{equation}\label{eq:det_deriv_row}
           \dfrac{d}{dx}\det(A)=\sum_{j=1}^n\sum_{\sigma\in S_n}\mathrm{sgn}(\sigma)\left(\dfrac{d}{dx}a_{j,\sigma(j)}\right)\prod_{\substack{k=1\\k\neq j}}^na_{k,\sigma(k)}.
        \end{equation}
        Remembering that $\det(A^T)=\det(A)$ we can write the alternative formula
        \begin{equation}\label{eq:det_deriv_column}
           \dfrac{d}{dx}\det(A)=\sum_{j=1}^n\sum_{\sigma\in S_n}\mathrm{sgn}(\sigma)\left(\dfrac{d}{dx}a_{\sigma(j),j}\right)\prod_{\substack{k=1\\k\neq j}}^na_{\sigma(k),k}.
        \end{equation}
    Denote by $A_{i|j}$ the matrix obtained from $A$ by deleting its $i$th row and $j$th column.  The Laplace expansion for the determinant along the $j$th row is given by
    \begin{equation*}
        \det(A)=\sum_{k=1}^{n}(-1)^{k+j}a_{jk}\det(A_{j|k}).
    \end{equation*}
\begin{proposition}[see \cite{VV}]\label{thm:deshanot-jacobi}
    Denote $A_{ij|kl}$ the matrix obtained from $A$ by deleting the $i$th and $j$th rows and $k$th and $l$th columns. Determinants of these matrices satisfy Deshanot-Jacobi identity
        \begin{equation}   \label{eq:deshanot_jacobi}  \det(A)\det(A_{ij|ij})=\det(A_{i|i})\det(A_{j|j})-\det(A_{i|j})\det(A_{j|i}),\quad 1\leq i,j\leq n.         \end{equation}
\end{proposition}
\begin{proposition}[\text{\cite[(2.43)]{FW03}}]\label{prop:fn_sol}
The sequence of functions \begin{equation}\label{def:fn}
           f_n(x)=\det\left(\left\{\left(x\dfrac{d}{dx}\right)^{i+j}f_0(x)\right\}_{i,j=0}^{n}\right)
        \end{equation}
with infinitely differentiable $f_0(x)$ solves Toda equation corresponding to Painlev\'e-III equation
\begin{equation}\label{def:fn_toda}
           \left(x\dfrac{d}{dx}\right)^2\ln(f_n(x))=\frac{f_{n+1}(x)f_{n-1}(x)}{f_n^2(x)},\quad n\geq 1.
        \end{equation}
\end{proposition}
\begin{proof}
Specifically, to match the expression in Proposition \ref{thm:deshanot-jacobi}, we rewrite \eqref{def:fn_toda} as:
    \begin{equation*}
        f_{n-1}(x)f_{n+1}(x)=f_n(x)\left(x\dfrac{d}{dx}\right)^2f_n(x)-\left(x\dfrac{d}{dx}f_n(x)\right)^2,\quad n\geq 1
    \end{equation*}
    Put $f_{n+1}(x)=\det(A)$. It follows that $f_n(x)=\det(A_{n+2|n+2})$ and $f_{n-1}(x)=\det(A_{n+1,n+2|n+1,n+2})$. We take the first derivative of the determinant in \eqref{def:fn} by multilinearity with respect to rows using \eqref{eq:det_deriv_row}. Since a determinant with two identical rows is zero, we end up with $x\dfrac{d}{dx}f_n(x)=\det(A_{n+1|n+2})$. Since $A_{n+1|n+2}=(A_{n+2|n+1})^{T}$, this implies that $x\dfrac{d}{dx}f_n(x)=\det(A_{n+1|n+2})=\det(A_{n+2|n+1})$. Then we take the second derivative of the determinant in \eqref{def:fn} successively by multilinearity with respect to columns using \eqref{eq:det_deriv_column}. Similarly, since a determinant with two identical columns is zero, we end up with $\left(x\dfrac{d}{dx}\right)^2f_n(x)=\det(A_{n+1|n+1})$. Using \eqref{eq:deshanot_jacobi}, we obtain \eqref{def:fn_toda}.
\end{proof}
Now, let us return to the special function solutions $u_n(x,\alpha)$.  We compute the corresponding auxiliary Hamiltonians $h_0(x,\alpha)$, $h_1(x,\alpha)$, and $h_2(x,\alpha)$. It turns out that corresponding tau functions can be chosen as
        \begin{equation}\label{eq:tau_0}
           \tau_0(x,\alpha)=1,
        \end{equation}
        \begin{equation}\label{eq:tau_1}
           \tau_1(x,\alpha)=\Cy_\frac{\alpha}{2}(x),
        \end{equation}
        \begin{equation}\label{eq:tau_2}
           \tau_2(x,\alpha)=\det\left(\begin{array}{cc}
   \tau_1(x,\alpha)  & x\dfrac{d}{dx}\tau_1(x,\alpha) \vspace{0.2cm}\\
  x\dfrac{d}{dx}\tau_1(x,\alpha)   & \left(x\dfrac{d}{dx}\right)^2\tau_1(x,\alpha)
\end{array}\right).
        \end{equation}
where $\Cy_\nu(x)$ is given by \eqref{def:f_nu}. It indicates that special function solutions can be represented using determinants as in Proposition \ref{prop:fn_sol} which is not true for arbitrary family of B\"acklund iterates. Moreover, it is the only family of solutions with this property. More specifically, condition \eqref{eq:tau_0} puts the restriction on parameters $\alpha$ and $\beta$ and imposes the Bessel differential equation for $\tau_1(x,\alpha)$, see \cite[Proposition 4.3]{FW03}. After that, the Toda equation determines the tau function uniquely given the initial condition, which produces determinantal formulas \eqref{eq:tau_2}, \eqref{eq:tau_first_det_formula}, see \cite[(2.43)]{FW03}. To summarize the solution with initial conditions \eqref{eq:tau_1}, \eqref{eq:tau_2} is given by \eqref{def:fn} in Proposition \ref{prop:fn_sol}. As a result, we get
\begin{proposition} Tau functions associated with the special function solutions $u_n(x,\alpha)$ can be chosen as
\begin{equation}\label{eq:tau_first_det_formula}
\tau_n(x,\alpha)=\det\left(\left\{\left(x\dfrac{d}{dx}\right)^{i+j-2}\tau_1(x,\alpha)\right\}_{i,j=1}^{n}\right)
        \end{equation}
        with $\tau_1(x,\alpha)=\Cy_\frac{\alpha}{2}(x)$, $\tau_0(x,\alpha)=1$.
\end{proposition}
       
\begin{proposition}
    \label{lem:tau-formula}
 Formula \eqref{eq:tau_first_det_formula} can be alternatively written as
     \begin{equation}\label{eq:tau_through_Delta}
         \tau_n(x,\alpha)=x^{n(n-1)}(-1)^{\frac{n(n-1)}{2}}\Delta_n(x,\alpha),
     \end{equation}
     where $\Delta_n(x,\alpha)$ is given by \eqref{def:hankel_bessel_determinant}.
\end{proposition}
     \begin{proof}
        Using mathematical induction and identity \eqref{eq:diff-iden1} one can compute the structure of $\left(x\dfrac{d}{dx}\right)^j\Cy_{\frac{\alpha}{2}}(x)$:
        \begin{equation}\label{eq:diff-iden5}
         \left(x\dfrac{d}{dx}\right)^j\Cy_{\frac{\alpha}{2}}(x)=x^j\Cy_{\frac{\alpha}{2}-j}(x)+\sum_{k=0}^{j-1}c_{kj}x^k\Cy_{\frac{\alpha}{2}-k}(x),
        \end{equation}
where $c_{kj}$ are constant coefficients.        Well, we furtherly simplify the determinant using \eqref{eq:diff-iden5}:
        \begin{align*}
           & \det\left(\left\{\left(x\dfrac{d}{dx}\right)^{k+j}\Cy_{\frac{\alpha}{2}}(x)\right\}_{k,j=0}^{n-1}\right)
            \\
            &=\begin{vmatrix}
                \Cy_{\frac{\alpha}{2}}(x) & x\dfrac{d}{dx}\Cy_{\frac{\alpha}{2}}(x) & \cdots\\
                x\Cy_{\frac{\alpha}{2}-1}(x)+c_{01}\Cy_{\frac{\alpha}{2}}(x) & x\dfrac{d}{dx}\left(x\Cy_{\frac{\alpha}{2}-1}(x)+c_{01}\Cy_{\frac{\alpha}{2}}(x)\right)& \cdots \\
                x^2\Cy_{\frac{\alpha}{2}-2}(x)+c_{12}x\Cy_{\frac{\alpha}{2}-1}(x)+c_{02}\Cy_{\frac{\alpha}{2}}(x)& x\dfrac{d}{dx}\left(x^2\Cy_{\frac{\alpha}{2}-2}(x)+c_{12}\Cy_{\frac{\alpha}{2}-1}+c_{02}\Cy_{\frac{\alpha}{2}}(x)\right)&\cdots\\
                \vdots&\vdots&\ddots
            \end{vmatrix}.
        \end{align*}
    Observe that by elementary row operations, we can always use the previous rows to eliminate the $\sum_{k=0}^{j-1}c_{kj}x^k\Cy_{\frac{\alpha}{2}-k}(x)$ part in a fixed row and the value of the determinant doesn't change. Doing that, we end up with:
    \begin{align*}
            &\det\left(\left\{\left(x\dfrac{d}{dx}\right)^{k+j}\Cy_{\frac{\alpha}{2}}(x)\right\}_{k,j=0}^{n-1}\right)\\
            &=\begin{vmatrix}
                \Cy_{\frac{\alpha}{2}}(x) & x\dfrac{d}{dx}\Cy_{\frac{\alpha}{2}}(x) & \cdots\vspace{0.2cm}\\
                x\,\Cy_{\frac{\alpha}{2}-1}(x)& x\dfrac{d}{dx}(x\,\Cy_{\frac{\alpha}{2}-1}(x))& \cdots \vspace{0.2cm}\\
                x^2\Cy_{\frac{\alpha}{2}-2}(x)& x\dfrac{d}{dx}(x^2\Cy_{\frac{\alpha}{2}-2}(x))&\cdots\\
                \vdots&\vdots&\ddots
            \end{vmatrix}=\det\left(\left\{\left(x\dfrac{d}{dx}\right)^{k}x^j\Cy_{\frac{\alpha}{2}-j}\right\}_{k,j=0}^{n-1}\right).
        \end{align*}
    By relation \eqref{eq:diff-iden2}, by induction, we can show:
    \begin{align}\label{eq:diff-iden6}
         \left(x\dfrac{d}{dx}\right)^kx^j\Cy_{\frac{\alpha}{2}-j}(x)&=(-1)^kx^{j+k}\Cy_{\frac{\alpha}{2}-j+k}(x)+\sum_{n=0}^{k-1}d_{nkj}x^{j+n}\Cy_{\frac{\alpha}{2}-j+n}(x)
        \end{align}
    To prove \eqref{eq:diff-iden6} by \eqref{eq:diff-iden2}, we first fix $k=1$ and induct on $j$. We have
    \begin{equation}
    \begin{aligned}
        x\dfrac{d}{dx}x^j\Cy_{\frac{\alpha}{2}-j}(x)=-x^{j+1}\Cy_{\frac{\alpha}{2}-j+1}(x)+x^j\left(\frac{\alpha}{2}-j\right)\Cy_{\frac{\alpha}{2}-j}(x)+jx^j\Cy_{\frac{\alpha}{2}-j}(x)\\=-x^{j+1}\Cy_{\frac{\alpha}{2}-j+1}(x)+x^j\frac{\alpha}{2}\Cy_{\frac{\alpha}{2}-j}(x)\label{eq:diff-iden7}
        \end{aligned}
        \end{equation}
After showing \eqref{eq:diff-iden7} we induct on $k$.        
    We also can notice that coefficient depending in $j$ cancels in the right-hand side of \eqref{eq:diff-iden7}. It implies that $d_{nkj}$ actually does not depend on $j$. So \eqref{eq:diff-iden6} can be written as 
    \begin{align}
         \left(x\dfrac{d}{dx}\right)^kx^j\Cy_{\frac{\alpha}{2}-j}(x)&=(-1)^kx^{j+k}\Cy_{\frac{\alpha}{2}-j+k}(x)+\sum_{n=0}^{k-1}d_{nk}x^{j+n}\Cy_{\frac{\alpha}{2}-j+n}(x)\label{eq:diff-iden8}
        \end{align}
    Again, we furtherly simplify the determinant by \eqref{eq:diff-iden8}:
    \begin{align*}
     \det\left(\left\{\left(x\dfrac{d}{dx}\right)^{k+j}\Cy_{\frac{\alpha}{2}}(x)\right\}_{k,j=0}^{n-1}\right)=
            \begin{vmatrix}
                \Cy_{\frac{\alpha}{2}}(x) & -x\,\Cy_{\frac{\alpha}{2}+1}(x)+d_{01}\Cy_{\frac{\alpha}{2}}(x)&\cdots\\
                x\,\Cy_{\frac{\alpha}{2}-1}(x) & -x^2\Cy_{\frac{\alpha}{2}}(x)+d_{01}x\,\Cy_{\frac{\alpha}{2}-1}(x)&\cdots\\
                \vdots&\vdots&\ddots
            \end{vmatrix}
    \end{align*}
    Similarly, by applying elementary column operations, we can always use the previous columns to eliminate the $\sum_{n=0}^{k-1}d_{nk}x^{j+n}\Cy_{\frac{\alpha}{2}-j+n}(x)$ part in a fixed column and the value of the determinant does not change. Finally, we will end up with:
    \begin{align*}
     \det\left(\left\{\left(x\dfrac{d}{dx}\right)^{k+j}\Cy_{\frac{\alpha}{2}}(x)\right\}_{k,j=0}^{n-1}\right)=\det\left((-1)^kx^{j+k}\Cy_{\frac{\alpha}{2}-j+k}(x)\right)
    \end{align*}
    By multi-linearity of determinant, we can factor out $(-1)^kx^{j+k}$ and reach the conclusion:
    \begin{align*}
      \tau_n(x,\alpha)=\det\left((-1)^kx^{j+k}\Cy_{\frac{\alpha}{2}-j+k}(x)\right)=x^{n(n-1)}(-1)^{\frac{n(n-1)}{2}}\Delta_n(x,\alpha)
    \end{align*}
    where $\Delta_n(x,\alpha)$ is given by \eqref{def:hankel_bessel_determinant}.
    That completes the proof.
     \end{proof}
  \subsection{Proof of Proposition \ref{prop:q_n_formula}}

  Before starting the proof we need to prove the following lemma.

  \begin{lemma}\label{lem:B1B2_identity}
The special function solution  $u_{n+1}(x,\alpha-2)$ admits the following formula in terms of $u_n(x,\alpha)$.
\begin{align}
(u_{n+1}(x,\alpha-2),\alpha+2n,-\alpha+6+2n)=B_1B_2(u_n(x,\alpha),\alpha+2n,-\alpha+2+2n).\label{eq:B1B2_identity}
\end{align}
  \end{lemma}
  \begin{proof}
We start by considering the B\"acklund transformations $B_1$ and $B_2$. Using the explicit formulas \eqref{eq:backlund1}, \eqref{eq:backlund2} and equation \eqref{eq:PIII} we can show that these transformations commute: $B_1B_2=B_2B_1$.

Let us now consider $B_2$ applied to the special function solution \eqref{def:u_0}. After using differential equation \eqref{eq:linear_ricatti2} for $w(x)$ we get
\begin{align*}
B_2(u_0(x,\alpha),\alpha,2-\alpha)[1]=\frac{2-\alpha}{x}+\frac{w(x)}{w'(x)}=\frac{2-\alpha}{x}+\dfrac{1}{\dfrac{\alpha}{2x}+\dfrac{\Cy'_\frac{\alpha}{2}(x)}{\Cy_\frac{\alpha}{2}(x)}}.
\end{align*}
We use the notation $[1]$ above for the first component of the output of B\"acklund transformation. We use \eqref{eq:diff-iden2} to get
\begin{align*}
B_2(u_0(x,\alpha),\alpha,2-\alpha)[1]=\frac{2-\alpha}{x}+\dfrac{\Cy_{\frac{\alpha}{2}}(x)}{\Cy_{\frac{\alpha}{2}-1}(x)}.
\end{align*}
Using identity \eqref{eq:diff-iden1} we can rewrite it as 
\begin{align}
B_2(u_0(x,\alpha),\alpha,2-\alpha)[1]=\frac{2-\alpha}{2x}-\dfrac{\Cy'_{\frac{\alpha}{2}-1}(x)}{\Cy_{\frac{\alpha}{2}-1}(x)}=-\dfrac{d}{dx}\ln(w(x,\alpha-2))=u_0(x,\alpha-2).
\end{align}
Now using commutativity of $B_1$ and $B_2$ we get
\begin{align*}
B_2(u_n(x,\alpha),\alpha+2n,-\alpha+2+2n)=B_2B_1^n(u_0(x,\alpha),\alpha,-\alpha+2)\\=B_1^n(u_0(x,\alpha-2),\alpha-2,-\alpha+4)=(u_n(x,\alpha-2),\alpha-2+2n,-\alpha+4+2n).
\end{align*}
Using similar logic, we arrive at \eqref{eq:B1B2_identity} and finish the proof.
  \end{proof}
\begin{proof}[Proof of Proposition \ref{prop:q_n_formula}]

We start our proof by introducing the following sequence of functions
\begin{equation}
    \widetilde{u}_n(x,\alpha)=-\frac{\Delta_{n+1}(x,\alpha-2)\Delta_n(x,\alpha)}{\Delta_{n+1}(x,\alpha)\Delta_n(x,\alpha-2)}.
        \end{equation}
Using Proposition \ref{lem:tau-formula} we can rewrite it in terms of tau functions
\begin{equation}\label{eq:qn-through-tau}
    \widetilde{u}_n(x,\alpha)=-\frac{\tau_{n+1}(x,\alpha-2)\tau_n(x,\alpha)}{\tau_{n+1}(x,\alpha)\tau_n(x,\alpha-2)}.
        \end{equation}
Using the Toda equation \eqref{def:fn_toda} and the definition of tau function \eqref{def_tau} we can see that 
\begin{equation}\label{eq:un_tilde_recurrence1}
    \widetilde{u}_{n+1}(x,\alpha)= \widetilde{u}_{n}(x,\alpha)\dfrac{  \left(x\dfrac{d}{dx}\right)^2\ln(\tau_{n+1}(x,\alpha-2))}{  \left(x\dfrac{d}{dx}\right)^2\ln(\tau_{n+1}(x,\alpha))}=\widetilde{u}_{n}(x,\alpha)\dfrac{  \left(x\dfrac{d}{dx}\right)h_{n+1}(x,\alpha-2)}{  \left(x\dfrac{d}{dx}\right)h_{n+1}(x,\alpha)}.
\end{equation}
Using Lemma \ref{lem:B1B2_identity} and the definition of $h_{n+1}(x,\alpha)$ we can express the right hand side of \eqref{eq:un_tilde_recurrence1} in terms of $u_n(x,\alpha)$. We also provide intermediate formulas
\begin{align}
    h_{n+1}(x,\alpha)=&-\frac{1}{8 u_n(x,\alpha)^2}\Big[u_n(x,\alpha)^2 \left(\alpha^2-4 n^2+2 x^2+3\right)+2 x u_n(x,\alpha)
   \left(\alpha-2 n+u_n'(x,\alpha)-2\right)\Big.\\&\Big.+2 x (\alpha+2 n+2) u_n(x,\alpha)^3-x^2
\left(\left(u_n'(x,\alpha)\right)^2-1\right)+x^2 u_n(x,\alpha)^4\Big],
\end{align}
\begin{align}
    h_{n+1}(x,\alpha-2)=&-\frac{1}{8 u_n(x,\alpha)^2}\Big[u_n(x,\alpha)^2 \left(a^2-4 a-4 n^2+2 x^2+7\right)+x^2 u_n(x,\alpha)^4\\&-2 x u_n(x,\alpha)
   \left(-a+2 n+u_n'(x,\alpha)+4\right)+2 x (a+2 n) u_n(x,\alpha)^3-x^2
   \left(u_n'(x,\alpha)^2-1\right)\Big].
\end{align}

As the result we get
\begin{equation}\label{eq:un_tilde_recurrence2}
    \widetilde{u}_{n+1}(x,\alpha)= \widetilde{u}_{n}(x,\alpha)\frac{(-3+\alpha+2n)u_n(x,\alpha)+xu_n^2(x,\alpha)+x+xu_n'(x,\alpha)}{u_n^2(x,\alpha)((1+\alpha+2n)u_n(x,\alpha)+xu_n^2(x,\alpha)+x+xu_n'(x,\alpha))}.
\end{equation}
Using identity \eqref{eq:diff-iden1}, initial conditions \eqref{eq:tau_0}, \eqref{eq:tau_1} and definition \eqref{eq:qn-through-tau} we can observe that 
\begin{equation}
\widetilde{u}_0(x,\alpha)=u_0(x,\alpha).
\end{equation}
Using explicit formula \eqref{eq:backlund1} for B\"acklund transformation $B_1$, relation
\eqref{eq:un_tilde_recurrence2}  and mathematical induction we get the desired result
\begin{equation}
    \widetilde{u}_n(x,\alpha)=u_n(x,\alpha).
\end{equation}

\end{proof}
\section{Asymptotics of Toeplitz determinant at zero}\label{sec:hankel_bessel_det_asymptotic}
\subsection{Andr\`eief identity}\label{sec:andreief}
To prove our result, we rewrite Toeplitz determinant \eqref{def:hankel_bessel_determinant} as a multiple contour integral. 
\begin{proposition}[see \cite{forrester}]\label{prop:andreief}
    Andr\`eief identity is given by the following formula
        \begin{equation*}
           \intop_{\Gamma}\ldots\intop_\Gamma \det\left(\left\{f_j(x_k)\right\}_{j,k=1}^n\right)\det\left(\left\{g_j(x_k)\right\}_{j,k=1}^n\right)\prod_{k=1}^n h(x_k)dx_k=n!\det\left(\left\{\intop_{\Gamma}f_j(x)g_k(x)h(x)dx\right\}_{j,k=1}^n\right)
        \end{equation*}
        where $\Gamma$ is some contour in the complex plane, such that the corresponding integral is finite.
\end{proposition}
    We apply the Andr\`eief identity and get the following result.
\begin{theorem} \label{thm:tau-integral-formula}
The Toeplitz determinant $\Delta_n(x,\alpha)$ given by \eqref{def:hankel_bessel_determinant} can be rewritten as
    \begin{equation}
\Delta_n(x,\alpha)=\frac{1}{n!}\intop_{\Gamma_1\cup\Gamma_2}\ldots \intop_{\Gamma_1\cup\Gamma_2}\prod_{1\leq j<k\leq n}(t_k-t_j)\left(\frac{1}{t_k}-\frac{1}{t_j}\right)\prod_{k=1}^{n} h_1(t_k)dt_k \label{eq:tau-formula}
        \end{equation}
        where
        \begin{equation*}
        h_1(t)=\frac{\ee^{\frac{x}{2}\left(t-\frac{1}{t}\right)}}{2 \pi \ii t^{1+\frac{\alpha}{2}}}\left(\left(d_1+d_2\cot\left(\frac{\pi\alpha}{2}\right)\right)\chi_{\Gamma_1}(t)+d_2\csc\left(\frac{\pi\alpha}{2}\right)\ee^{\frac{\ii\pi\alpha}{2}}\chi_{ \Gamma_2}(t)\right),
        \end{equation*}
       and contours of integration $\Gamma_1$, $\Gamma_2$ are shown on Figure \ref{fig:gamma12Cy} and they don't intersect. We use the notation $\chi_{ \Gamma_j}(t)$ for the characteristic function of the contour $\Gamma_j$.  We assume $-\pi<\arg(t)<\pi$ on the contour $\Gamma_1$ and $0<\arg(t)<2\pi $ on the contour $\Gamma_2$.
\end{theorem}
\begin{proof}
  Using contour integral representation \eqref{eq:cy_integral_representation} we get
    \begin{equation*}
    \Cy_{\frac{\alpha}{2}-j+k}(x)=\intop_{\Gamma_1\cup\Gamma_2}t^{-k}t^jh_1(t)dt
     \end{equation*}
Put $g_k(t)=t^{-k}$ and $f_j(t)=t^j$. By Proposition \ref{prop:andreief}, we get
     \begin{align*}
         \Delta_n(x,\alpha)&=\frac{1 }{n!}\intop_{\Gamma_1\cup\Gamma_2}\ldots\intop_{\Gamma_1\cup\Gamma_2}\det\left(\left\{f_j(t_k)\right\}_{j,k=0}^{n-1}\right)\det\left(\left\{g_j(t_k)\right\}_{j,k=0}^{n-1}\right)\prod_{k=0}^{n-1} h_1(t_k)dt_k\\&=\frac{1}{n!}\intop_{\Gamma_1\cup\Gamma_2}\ldots\intop_{\Gamma_1\cup\Gamma_2}\det\left(\left\{t_k^j\right\}_{j,k=0}^{n-1}\right)\det\left(\left\{t_k^{-j}\right\}_{j,k=0}^{n-1}\right)\prod_{k=0}^{n-1} h_1(t_k)dt_k
     \end{align*}
By the formula for the Vandermonde determinant we can simplify the integrand and get
    \begin{equation*}
       \det\left(\left\{t_k^j\right\}_{j,k=0}^{n-1}\right)=\prod_{0\leq j<k\leq n-1}(t_k-t_j)
     \end{equation*}
    \begin{equation*}
       \det\left(\left\{t_k^{-j}\right\}_{j,k=0}^{n-1}\right)=\prod_{0\leq j<k\leq n-1}\left(\frac{1}{t_k}-\frac{1}{t_j}\right)
     \end{equation*}
Thus, the explicit formula for $\Delta_n(x,\alpha)$ is given by:
        \begin{equation*}
         \Delta_n(x,\alpha)=\frac{1}{n!}\intop_{\Gamma_1\cup\Gamma_2}\ldots \intop_{\Gamma_1\cup\Gamma_2}\prod_{0\leq j<k\leq n-1}(t_k-t_j)\left(\frac{1}{t_k}-\frac{1}{t_j}\right)\prod_{k=0}^{n-1} h_1(t_k)dt_k.
        \end{equation*}
For convenience we shift the index of variables $t_k$.

\end{proof}
\subsection{Basic strategies}
Up to this point, we have enough preparation to compute the asymptotics at zero. Our goal is to get asymptotics $\Delta_n(x,\alpha)\sim b(n)x^{a(n)}$ when $x\rightarrow 0$, $x>0$. This is a reasonable expectation, since the Bessel function $J_\nu(x)$ admits series representation \eqref{eq:bessel_series}.
We summarize several key ideas to achieve this goal.
\begin{itemize}
\item The contours $\Gamma_1$ and $\Gamma_2$ spread to zero and infinity in formula \eqref{eq:tau-formula}.   We cannot put $x=0$ here without losing convergence of the integral. 
     \item Expanding the product $\prod_{k=0}^{n-1} h(t_k)$ in the integrand of \eqref{eq:tau-formula} we get the sum of expressions, each of them has some of the variables $t_k$ belonging to the contour $\Gamma_1$ and others belonging to $\Gamma_2$.
    \item We apply the change of variables $t=\frac{2}{x}s$ to variables on contours $\Gamma_1$.  The integrand will maintain an exponential decay at infinity when we put $x=0$. On the other hand, we can apply the change of variable $t=\frac{x}{2}s$ to the variables on contours $\Gamma_2$. In this case, the integrand will preserve exponential decay at zero when we put $x=0$.
     \item From first glance, it seems that it would be more convenient to use the expression $d_1J_{\nu}(x)+d_2J_{-\nu}(x)$ instead of \eqref{def:f_nu}. But unfortunately, Proposition \ref{prop:q_n_formula} would fail if we replace $Y_\nu(x)$ with $J_{-\nu}(x)$. It follows from the fact that the differential identities \eqref{eq:diff-iden1}, \eqref{eq:diff-iden2} would fail for this alternative choice. And we need them to hold, since they were used extensively in the proof of Proposition \ref{lem:tau-formula}. 
\end{itemize}

\subsection{Expanded formula for \texorpdfstring{$\Delta_n(x,\alpha)$}{Deltan(x,alpha)}}
    
We start with the following observation. Let $I$ denote a subset of the set of indices $\{1,\ldots,n\}$, $r$ denote its cardinality $|I|$ and $I^c$ denote its complement. The following identity holds
         \begin{equation}
           \label{eq:expansion}\prod\limits_{k=1}^n(c_1\chi_{\Gamma_1}(t_k)+c_2\chi_{ \Gamma_2} 
 (t_k))=\sum_{r=0}^n\sum_{\substack{I\subset\{1,...,n\}\\ \left\vert I\right\vert=r}}c_1^{r}c_2^{n-r}\prod_{k\in I}\chi_{\Gamma_1}(t_k)\prod_{j\in I^c}\chi_{\Gamma_2}(t_j).
          \end{equation}

Now we apply \eqref{eq:expansion} to the expression in Theorem \ref{thm:tau-integral-formula} to convert the formula into a summation form and decouple the contours. Denote
\begin{equation}
    \label{def:c1_c2}{}c_1=d_1+d_2\cot\left(\frac{\pi\alpha}{2}\right),\quad c_2=d_2\csc\left(\frac{\pi\alpha}{2}\right)\ee^{\frac{\ii\pi\alpha}{2}}.
\end{equation}
We have
    \begin{align*}
&\Delta_n(x,\alpha)=\frac{(-1)^{\frac{n(n-1)}{2}}}{n!}\intop_{\Gamma_1\cup\Gamma_2}\ldots \intop_{\Gamma_1\cup\Gamma_2}\prod_{m=1}^{n}(c_1\chi_{\Gamma_1}(t_m)+c_2\chi_{\Gamma_{2}}(t_m))\prod_{1\leq j<k\leq n}\frac{(t_j-t_k)^2}{t_jt_k}\prod_{l=1}^{n}\frac{\ee^{\frac{x}{2}(t_l-\frac{1}{t_l})}}{2\pi \ii t_{l}^{\frac{\alpha}{2}+1}}dt_l\\=&\frac{(-1)^{\frac{n(n-1)}{2}}}{n!}\intop_{\Gamma_1\cup\Gamma_2}\ldots \intop_{\Gamma_1\cup\Gamma_2}\sum_{r=0}^{n}\sum_{\substack{I\subset\{1,...,n\}\\ \left\vert I\right\vert=r}}c_1^{r}c_2^{n-r}\prod_{i\in I}\chi_{\Gamma_1}(t_i)\prod_{j\in I^c}\chi_{\Gamma_2}(t_j)\prod_{1\leq j<k\leq n}\frac{(t_j-t_k)^2}{t_jt_k}\prod_{l=1}^{n}\frac{\ee^{\frac{x}{2}(t_l-\frac{1}{t_l})}}{2\pi \ii t_{l}^{\frac{\alpha}{2}+1}}dt_l\\=&\frac{(-1)^{\frac{n(n-1)}{2}}}{(2\pi \ii)^{n}n!}\sum_{r=0}^{n}\sum_{\substack{I\subset\{1,...,n\}\\ \left\vert I\right\vert=r}}c_1^{r}c_2^{n-r}\intop_{\Gamma_1\cup\Gamma_2}\ldots \intop_{\Gamma_1\cup\Gamma_2}{\prod_{\substack{j<k\\j,k\in I}}\frac{(t_j-t_k)^2}{t_jt_k}}{\prod_{\substack{j<k\\j,k\in I^{c}}}\frac{(t_j-t_k)^2}{t_jt_k}}{\prod_{\substack{j\in I\\k\in I^{c}}}\frac{(t_j-t_k)^2}{t_jt_k}}\\&\prod_{l\in I}\chi_{\Gamma_1}(t_l)\prod_{l\in I^c}\chi_{\Gamma_2}(t_l)\prod_{l=1}^n\frac{\ee^{\frac{x}{2}(t_l-\frac{1}{t_l})}}{t_{l}^{\frac{\alpha}{2}+1}}dt_l
        \end{align*}
We remind that we assume $-\pi<\arg(t)<\pi$ on contour $\Gamma_1$ and $0<\arg(t)<2\pi $ on contour $\Gamma_2$. We can see that by renaming variables on the right hand side we can guarantee that $I=\{1,2,\ldots,r\}$, and $I^c=\{r+1,r+2,\ldots,n\}$ for each integral in the sum. Combining the same integrals together we get
\begin{align}
&\Delta_n(x,\alpha)=\frac{(-1)^{\frac{n(n-1)}{2}}}{(\pi \ii)^{n}}\sum_{r=0}^n\frac{c_1^{r}c_2^{n-r}}{r!(n-r)!}\intop_{\Gamma_1}\ldots\intop_{\Gamma_1}\intop_{\Gamma_2}\ldots\intop_{\Gamma_2}\prod_{{1\leq j<k\leq r}}\frac{(t_j-t_k)^2}{t_jt_k}{\prod_{r+1\leq j<k\leq n}\frac{(t_j-t_k)^2}{t_jt_k}}\\&\times{\prod_{{j=1}}^r\prod_{{k=r+1}}^n\frac{(t_j-t_k)^2}{t_jt_k}}\prod_{l=1}^n\frac{\ee^{\frac{x}{2}(t_l-\frac{1}{t_l})}}{t_{l}^{\frac{\alpha}{2}+1}}dt_l{},
\label{eq:delta_n_expansion0}
\end{align}
where we have $r$ integrals over $\Gamma_1$ and $n-r$ integrals over $\Gamma_2$.

For $t_k\in I$, we use change of variable $t_k=\frac{2}{x}s_k$. On the other hand, for $t_j\in I^{c}$, we use change of variable $t_j=\frac{x}{2}s_j$. Since the only singularities of the integrand are at zero and infinity, we can deform contours of integration back to $\Gamma_1$ and $\Gamma_2$. The formula above becomes 
    \begin{align*}
        &\Delta_{n}(x,\alpha)=\frac{(-1)^{\frac{n(n-1)}{2}}}{(2\pi \ii)^{n}}\sum_{r=0}^n\frac{c_1^{r}c_2^{n-r}}{r!(n-r)!}\intop_{\Gamma_1}\ldots\intop_{\Gamma_1}\intop_{\Gamma_2}\ldots\intop_{\Gamma_2}\prod_{1\leq j<k\leq r}\frac{(s_j-s_k)^2}{s_js_k}\prod_{r+1\leq j<k\leq n}\frac{(s_j-s_k)^2}{s_js_k}\\&\prod_{j=1}^r\prod_{k=r+1}^n\left(\frac{2}{x}\right)^2\frac{(s_j(1+\mathcal{O}(x^2))^2}{s_js_k}\prod_{l=1}^r\frac{\ee^{s_l}(1+\mathcal{O}(x^2))}{\left(\frac{2}{x}\right)^{\frac{\alpha}{2}+1}s_l^{\frac{\alpha}{2}+1}}\frac{2}{x}ds_l\prod_{l=r+1}^n\frac{\ee^{-\frac{1}{s_l}}(1+\mathcal{O}(x^2))}{s_l^{\frac{\alpha}{2}+1}\left(\frac{x}{2}\right)^{\frac{\alpha}{2}+1}}\frac{x}{2}ds_l
    \end{align*}
\normalsize
    Grouping all the $\frac{2}{x}$ factors together and pulling them out of the summation, 
    \begin{align*}
&\Delta_n(x,\alpha)=\frac{(-1)^{\frac{n(n-1)}{2}}}{(2\pi \ii)^{n}}\sum_{r=0}^{n}\frac{c_1^{r}c_2^{n-r}}{r!(n-r)!}\left(\frac{2}{x}\right)^{-\alpha r+2r(n-r)+\frac{\alpha n}{2}}\intop_{\Gamma_1}\ldots\intop_{\Gamma_1}\intop_{\Gamma_2}\ldots\intop_{\Gamma_2}\prod_{1\leq j<k\leq r}\frac{(s_j-s_k)^2}{s_js_k}\\&\prod_{r+1\leq j<k\leq n}\frac{(s_j-s_k)^2}{s_js_k}
        \prod_{j=1}^r\prod_{k=r+1}^n\frac{(s_j(1+\mathcal{O}(x^2))^2}{s_js_k}\prod_{l=1}^r\frac{\ee^{s_l}(1+\mathcal{O}(x^2))}{s_l^{\frac{\alpha}{2}+1}}ds_l\prod_{l=r+1}^n\frac{\ee^{-\frac{1}{s_l}}(1+\mathcal{O}(x^2))}{s_l^{\frac{\alpha}{2}+1}}ds_l
    \end{align*}
\normalsize
We also want to group all the products of variable s together and separate the integrals based on different contours. We rewrite the following three parts,
\begin{align}
    &\prod_{1\leq j<k\leq r}\frac{(s_j-s_k)^2}{s_js_k}= \prod_{1\leq j<k\leq r}(s_j-s_k)^2\prod_{\substack{j\not= k\\1\leq j,k\leq r}}\frac{1}{\sqrt{s_js_k}}=\prod_{1\leq j<k\leq r}(s_j-s_k)^2\prod_{l=1}^rs_l^{\gamma_1}\label{eq:gamma1}\\&\prod_{r+1\leq j<k\leq n}\frac{(s_j-s_k)^2}{s_js_k}= \prod_{r+1\leq j<k\leq n}(s_j-s_k)^2\hspace{-0.5cm}\prod_{\substack{j\not= k\\r+1\leq j,k\leq n}}\frac{1}{\sqrt{s_js_k}}=\prod_{r+1\leq j<k\leq n}(s_j-s_k)^2\prod_{l=r+1}^ns_l^{\gamma_2}\label{eq:gamma2}\\&\prod_{j=1}^r\prod_{k=r+1}^n\frac{s_j}{s_k}=\prod_{l=1}^rs_l^{\gamma_3}\prod_{l=r+1}^ns_l^{\gamma_4}
\end{align}
To find $\gamma_1$ we can interpret the product in \eqref{eq:gamma1} as product a over all elements of $r\times r$ matrix except for the diagonal. The terms with $s_j$ appear along the $j$th row and $j$th column, so there are $2(r-1)$ of them. A similar argument can be used for the computation of $\gamma_2$, but the size of the matrix would be $(n-r)\times (n-r)$.  Keeping in mind that we still have square root and introducing the power $-\frac{1}{2}$, then we get:
\begin{align*}
    &\gamma_1=-\frac{2(r-1)}{2}=1-r,\\
    &\gamma_2=-\frac{2(n-r-1)}{2}=-n+r+1.
\end{align*}
To compute $\gamma_3$ and $\gamma_4$ we visualize the number of the $s$-factors using the following matrix
\begin{equation*}
    \left\{\dfrac{s_{j}}{s_{k}}\right\}_{\substack{1\leq j \leq r\\r+1 \leq k\leq n}}=
    \begin{pmatrix}
    \dfrac{s_{1}}{s_{r+1}} & \dfrac{s_{1}}{s_{r+2}} & \dfrac{s_{1}}{s_{r+3}} & \cdots
    \vspace{0.2cm}\\
     \dfrac{s_{2}}{s_{r+1}} & \dfrac{s_{2}}{s_{r+2}} & \dfrac{s_{2}}{s_{r+3}} & \cdots\\
     \vdots & \vdots & \vdots & \ddots
    \end{pmatrix}
\end{equation*}
We can observe that horizontally, for each $s_{j}$, there are $(n-r)$ factors. Vertically, for each $s_{k}$ there are $r$ factors. We get
\begin{align*}
    &\gamma_3=n-r,\\
    &\gamma_4=-r.
\end{align*}
As a result, we get the following preliminary asymptotic formula for $\Delta_n(x,\alpha)$.
 
    \begin{lemma}\label{lem:sum_of_integrals}
   The Toeplitz determinant  \eqref{def:hankel_bessel_determinant} admits the following  $x \rightarrow 0$, $x>0$ asymptotics for fixed $d_1,d_2\in \mathbb{R}$, $n\in\mathbb{N}\cup\{0\}$, $\alpha\in\mathbb{C}\setminus (2\mathbb{Z}+\ii\mathbb{R})$
   \begin{equation}
       \begin{split}
&\Delta_n(x,\alpha)=\frac{(-1)^{\frac{n(n-1)}{2}}}{(2\pi \ii)^{n}}\sum_{r=0}^n\frac{c_1^{r}c_2^{n-r}}{r!(n-r)!}\left(\frac{2}{x}\right)^{-\alpha r+2r(n-r)+\frac{\alpha n}{2}}(1+\mathcal{O}(x^2))\\&\times\intop_{\Gamma_1}\ldots \intop_{\Gamma_1}\prod_{{1\leq j< k\leq r}}(s_j-s_k)^2
           \prod_{l=1}^r\ee^{s_l}s_l^{-\frac{\alpha}{2}-2r+n}ds_l\intop_{\Gamma_2}\ldots \intop_{\Gamma_2}\prod_{{1\leq j< k\leq n-r}}(s_j-s_k)^2\prod_{l=1}^{n-r}\ee^{-\frac{1}{s_l}}s_l^{-\frac{\alpha}{2}-n}ds_l.
      \label{eq:hankel_alternative}  \end{split}\end{equation}
The coefficients $c_1$, $c_2$ are given by \eqref{def:c1_c2}.       The contours of integration $\Gamma_1$, $\Gamma_2$ are shown in Figure \ref{fig:gamma12Cy}.  We assume $-\pi<\arg(t)<\pi$ on the contour $\Gamma_1$ and $0<\arg(t)<2\pi $ on the contour $\Gamma_2$.
  
    \end{lemma}
    \subsection{Asymptotics of \texorpdfstring{$\Delta_n(x,\alpha)$}{Delta n(x,alpha)}  for \texorpdfstring{$x\to 0$}{x->0}, \texorpdfstring{$x>0$}{x>0}.}
The asymptotics of $\Delta_n(x,\alpha)$ is the leading term of the asymptotic formula \eqref{eq:hankel_alternative}. Denote the power of $x$ appearing in \eqref{eq:hankel_alternative} as $p(r,\alpha,n)$:
\begin{equation}\label{def:p}
    p(r,\alpha,n)=\alpha r-\frac{n\alpha}{2}-2r(n-r).
\end{equation}
We need to find the minimum of $\re(p(r,\alpha,n))$ with respect to $r$. Introduce notation for the index which realizes this minimum 
\begin{equation}\label{def:rc}
\mathop{\min}\limits_{\substack{0\leq r\leq n\\r\in \mathbb{N}\cup\{0\}}}\mathrm{Re}(p(r,\alpha,n))=p(r_c(\alpha,n),\alpha,n).
\end{equation}
We have the following formula for it.
\begin{lemma}\label{prop:piecewise function of r_c}
The critical index $r_c(\alpha,n)$ defined by \eqref{def:rc} admits the following piecewise formula
\begin{equation}\label{eq:rc}
r_c(\alpha,n)=
    \begin{cases}
    0 & \text{if } \re(\alpha)>2n-2\\
    j & \text{if } 2n-4j-2<\re(\alpha)<2n-4j+2\\& \text{ and }j=1, 2,\ldots,n-1\\
    n &\text{if } \re(\alpha)<-2n+2
    \end{cases}
\end{equation}
\end{lemma}
\begin{proof}
Since $0\leq r\leq n$ and $r\in \mathbb{Z}$, $p(r,\alpha,n)$ only takes values on that discrete set. It is clear that $\re(p(r,\alpha,n))$ is an upward parabola in variable $r$ and has a minimum value at $r_{\min}(n,\alpha)=\frac{2n-\re(\alpha)}{4}$. We will discuss different cases of relative positions between $r_{\min}(\alpha,n)$ and $r_c(\alpha,n)$. If $r_{\min}(\alpha,n)\leq 0$, then $r_c(\alpha,n)=0$. If $r_{\min}(\alpha,n)\geq n$, then $r_c(\alpha,n)=n$. Let $0\leq j\leq n$ and $j\leq r_{\min}(\alpha,n)\leq j+1$. If $j\leq r_{\min}(\alpha,n)< j+\frac{1}{2}$, then $r_c(\alpha,n)=j$. If $j+\frac{1}{2}< r_{\min}(\alpha,n)\leq j+1$, then $r_c(\alpha,n)=j+1$. In other words,
\begin{itemize}
    \item $r_c(\alpha,n)=0$ when $\frac{2n-\re(\alpha)}{4}<\frac{1}{2}$,
    \item $r_c(\alpha,n)=j$ when $j-\frac{1}{2}<\frac{2n-\re(\alpha)}{4}<j+\frac{1}{2}$ and $j=1,2,\ldots, n-1$,
    \item $r_c(\alpha,n)=n$ when $\frac{2n-\re(\alpha)}{4}>n-\frac{1}{2}$.
\end{itemize}
These conditions can be rewritten as \eqref{eq:rc}.
\end{proof}
\begin{remark}
 The floor function gives a more compact form for $r_c(\alpha,n)$. Indeed $\forall$ $0\leq j\leq n$, we know $r_c(\alpha,n)=j$ if and only if $j-\frac{1}{2}\leq \frac{n}{2}-\frac{\re(\alpha)}{4}\leq j+\frac{1}{2}$ or $j\leq \frac{n}{2}-\frac{\re(\alpha)}{4}+\frac{1}{2}\leq j+1$. Therefore
\begin{equation*}
r_c(\alpha,n)=
    \begin{cases}
        n & \text{if } {\re(\alpha)}{}<2-2n,\\
        \lfloor \frac{n}{2}-\frac{\re(\alpha)}{4}+\frac{1}{2}\rfloor & \text{if } 2-2n<{\re(\alpha)}{}<2n-2,\\
        0 &\text{if } \re(\alpha)>2n-2.
    \end{cases}
\end{equation*}
\end{remark}
\begin{remark}\label{rem:complex_alpha}
    We avoid the case $\re(\alpha)=2n-4j-2$, $j=0,\ldots,n-1$, $\alpha\notin 2\mathbb{Z}$ to make sure that leading contribution in the asymptotics comes from the one value of $r_c$. Otherwise we need to include two contributions in the leading term of asymptotics of $\Delta_n(x,\alpha)$, which produces qualitatively oscillating solutions. See Appendix \ref{app:boundary_case} for details.
\end{remark}
As the result we have
    \begin{equation}
       \begin{split}
&\Delta_n(x,\alpha)\sim\frac{(-1)^{\frac{n(n-1)}{2}}}{(2\pi \ii)^{n}}\frac{c_1^{r_c}c_2^{n-r_c}}{r_c!(n-r_c)!}\left(\frac{x}{2}\right)^{p(r_c,\alpha,n)}\intop_{\Gamma_1}\ldots \intop_{\Gamma_1}\prod_{{1\leq j< k\leq r_c}}(s_j-s_k)^2
           \prod_{l=1}^{r_c}\ee^{s_l}s_l^{-\frac{\alpha}{2}-2r_c+n}ds_l\\&\times\intop_{\Gamma_2}\ldots \intop_{\Gamma_2}\prod_{{1\leq j< k\leq n-r_c}}(s_j-s_k)^2\prod_{l=1}^{n-r_c}\ee^{-\frac{1}{s_l}}s_l^{-\frac{\alpha}{2}-n}ds_l,\quad \text{ as }x\to 0,\quad x>0.
\end{split}\end{equation}
Denote 
\begin{equation}\label{def:H1}
H_1(\alpha,n)=\intop_{\Gamma_1}\ldots \intop_{\Gamma_1}\prod_{\substack{1\leq j< k\leq r_c}}(s_j-s_k)^2\prod_{\substack{l=1}}^{r_c}\ee^{s_l}s_l^{-\frac{\alpha}{2}-2r_c+n}ds_l
\end{equation}
and
\begin{equation}\label{def:H2}
    H_{2}(\alpha,n)=\intop_{\Gamma_2}\ldots \intop_{\Gamma_2}\prod_{\substack{1\leq j< k\leq n-r_c}}(s_j-s_k)^2\prod_{l=1}^{n-r_c}\ee^{-\frac{1}{s_l}}s_l^{-\frac{\alpha}{2}-n}ds_l.
\end{equation}
To evaluate \eqref{def:H1}, \eqref{def:H2}, we reduce them to multiple integrals with Laguerre weight $w(x)=\ee^{-x}x^{\alpha}$, $\re(\alpha)>-1$ on the contour $\Gamma=[0,\infty)$.
In $H_1(\alpha,n)$, we make the change of variable $s=\widetilde{s}\ee^{-\ii\pi}$. More specifically the modulus and argument of the variable transform as
\begin{align*}
    &|s|=|\widetilde{s}|\\&\arg(s)=\arg(\widetilde{s})-\pi, \quad 0<\arg(\widetilde{s})<2\pi.
\end{align*}
The contour $\Gamma_{1}$ becomes $\widetilde{\Gamma_{1}}$ as shown in Figure \ref{fig:gamma1tilde}. 
\begin{figure}[ht]
	\centering
	\includegraphics{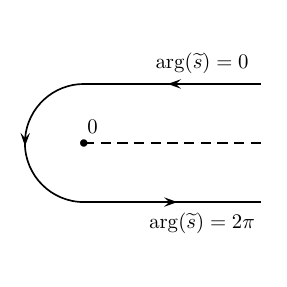}
        \caption{Contour $\widetilde{\Gamma_1}$}
        \label{fig:gamma1tilde}
\end{figure}
\\Also notice that 
\begin{align*}
    s_l^{-\frac{\alpha}{2}-2r_c+n}&=\ee^{(-\frac{\alpha}{2}-2r_c+n)\ln|s_l|+(-\frac{\alpha}{2}-2r_c+n)i\arg(s_{l})}\\&=\ee^{(-\frac{\alpha}{2}-2r_c+n)\ln|\widetilde{s_l}|+(-\frac{\alpha}{2}-2r_c+n)i(\arg(\widetilde{s_{l}})-\pi)}\\&=\widetilde{s_l}^{-\frac{\alpha}{2}-2r_c+n}\ee^{-\ii\pi(-\frac{\alpha}{2}-2r_c+n)}
\end{align*}
As the result we have
\begin{align*}
    H_1(\alpha,n)&=\intop_{\widetilde{\Gamma_1}}\ldots \intop_{\widetilde{\Gamma_1}}\prod_{\substack{1\leq j< k\leq r_c}}(\widetilde{s_j}-\widetilde{s_k})^2\prod_{l=1}^{r_c}\ee^{-\widetilde{s_l}}\widetilde{s_l}^{-\frac{\alpha}{2}-2r_c+n}\ee^{-\ii\pi(-\frac{\alpha}{2}-2r_c+n)}(-1)d\widetilde{s_l}\\&=(-1)^{r_{c}}(\ee^{-\ii\pi(-\frac{\alpha}{2}-2r_c+n)})^{r_{c}}\intop_{\widetilde{\Gamma_1}}\ldots \intop_{\widetilde{\Gamma_1}}\prod_{\substack{1\leq j< k\leq r_c}}(\widetilde{s_j}-\widetilde{s_k})^2\prod_{\substack{l=1}}^{r_c}\ee^{-\widetilde{s_l}}\widetilde{s_l}^{-\frac{\alpha}{2}-2r_c+n}d\widetilde{s_l}
\end{align*}
To continue our evaluation we need the following Lemma.
\begin{lemma}\label{claim:norm-const}
            Using the properties of power function we can show the following identity \begin{equation}  \label{eq:hook_identity1}\intop_{\widetilde{\Gamma_{1}}}s^j\ee^{-s}s^{\gamma}ds=({\ee^{2\pi \ii \gamma}-1})\intop_{0}^{\infty}s^j\ee^{-s}s^{\gamma}ds,\quad j\in\mathbb{Z}, \quad j+\re(\gamma)>-1.
        \end{equation}
\end{lemma}
With the aid of \eqref{eq:hook_identity1} we can rewrite our contour integrals in terms of real line integrals under the convergence condition $\re(\alpha)<2n+2-4r_c$. 
\begin{align*}
    H_1(\alpha,n)&=(-1)^{r_{c}}(\ee^{-\ii\pi(-\frac{\alpha}{2}-2r_c+n)})^{r_{c}}(\ee^{2\pi \ii(-\frac{\alpha}{2}-2r_c+n)}-1)^{r_{c}}\\&\times\intop_{0}^{\infty}\ldots \intop_{0}^{\infty}\prod_{\substack{1\leq j< k\leq r_c}}(\widetilde{s_j}-\widetilde{s_k})^2\prod_{\substack{l=1}}^{r_c}\ee^{-\widetilde{s_l}}\widetilde{s_l}^{-\frac{\alpha}{2}-2r_c+n}d\widetilde{s_l}
\end{align*}
By using \cite[\href{http://dlmf.nist.gov/5.14.E5}{(5.14.5)}]{DLMF} we get for $\re(\alpha)<2n+2-4r_c$
\begin{equation}\label{eq:H1_formula}
    H_1(\alpha,n)=(-1)^{nr_c}(2\ii)^{r_c}\sin^{r_c}\left(\frac{\pi\alpha}{2}\right)\frac{G\left(1+n-r_c-\frac{\alpha}{2}\right)G(r_c+2)}{G\left(1-\frac{\alpha}{2}-2r_c+n\right)}.
\end{equation}
We can notice that left and right hand sides of \eqref{eq:H1_formula} are entire functions of $\alpha$. Therefore by the uniqueness of analytic continuation we can say that \eqref{eq:H1_formula} holds for all $\alpha\in\mathbb{C}$.

Similarly, in $ H_2(\alpha,n)$, we put $s=\frac{1}{\widetilde{s}}$. Then the modulus and argument of the variable respectively transform as
\begin{align*}
    &|s|=\frac{1}{|\widetilde{s}|}\\&\arg(\widetilde{s})=-\arg(s), \quad -2\pi<\arg(\widetilde{s})<0
\end{align*}
The contour $\Gamma_{2}$ becomes $\widetilde{\Gamma_{2}}$ as shown in the Figure \ref{fig:gamma2tilde}.
\begin{figure}[ht]
	\centering
\includegraphics{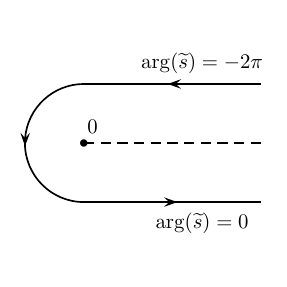}
        \caption{Contour $\widetilde{\Gamma_2}$}
        \label{fig:gamma2tilde}
\end{figure}

The second multi-integral becomes
\begin{align*}
H_2(\alpha,n)&=\intop_{\widetilde{\Gamma_2}}\ldots \intop_{\widetilde{\Gamma_2}}\prod_{\substack{1\leq j< k\leq n-r_c}}\left(\frac{1}{\widetilde{s_j}}-\frac{1}{\widetilde{s_k}}\right)^2\prod_{\substack{l=1}}^{n-r_c}\ee^{-\widetilde{s_l}}\widetilde{s_l}^{\frac{\alpha}{2}+n}\left(-\frac{1}{s_l^2}\right)d\widetilde{s_l}\\&=(-1)^{n-r_{c}}\intop_{\widetilde{\Gamma_2}}\ldots \intop_{\widetilde{\Gamma_2}}\prod_{\substack{1\leq j< k\leq n-r_c}}\left(\frac{\widetilde{s_j}-\widetilde{s_k}}{\widetilde{s_j}\widetilde{s_k}}\right)^2\prod_{\substack{l=1}}^{n-r_c}\ee^{-\widetilde{s_l}}\widetilde{s_l}^{\frac{\alpha}{2}+n-2}d\widetilde{s_l}\\&=(-1)^{n-r_{c}}\intop_{\widetilde{\Gamma_2}}\ldots \intop_{\widetilde{\Gamma_2}}\prod_{\substack{1\leq j< k\leq n-r_c}}(\widetilde{s_j}-\widetilde{s_k})^2\prod_{\substack{l=1}}^{n-r_c}\widetilde{s_{l}}^{2r_{c}+2-2n}\ee^{-\widetilde{s_l}}\widetilde{s_l}^{\frac{\alpha}{2}+n-2}d\widetilde{s_l}\\&=(-1)^{n-r_{c}}\intop_{\widetilde{\Gamma_2}}\ldots \intop_{\widetilde{\Gamma_2}}\prod_{\substack{1\leq j< k\leq n-r_c}}(\widetilde{s_j}-\widetilde{s_k})^2\prod_{\substack{l=1}}^{n-r_c}\ee^{-\widetilde{s_l}}\widetilde{s_l}^{\frac{\alpha}{2}+2r_{c}-n}d\widetilde{s_l}
\end{align*}
Here we used the square of \eqref{eq:gamma2}
\begin{equation*}
    \prod_{\substack{1\leq j< k\leq n-r_c}}\left(\frac{\widetilde{s_j}-\widetilde{s_k}}{\widetilde{s_j}\widetilde{s_k}}\right)^2=\prod_{\substack{1\leq j< k\leq n-r_c}}(\widetilde{s_j}-\widetilde{s_k})^2\prod_{\substack{j\not= k\\1\leq j,k\leq n-r_c}}\frac{1}{\widetilde{s_j}\widetilde{s_k}}=\prod_{\substack{1\leq j< k\leq n-r_c}}(\widetilde{s_j}-\widetilde{s_k})^2\prod_{l=1}^{n-r_c}\widetilde{s_l}^{2r_c+2-2n}
\end{equation*}
To continue our evaluation we need the following analog of Lemma \ref{claim:norm-const}.
\begin{lemma}
            Using the properties of power function we can show the following identity  \begin{equation}  \label{eq:hook_identity2}\intop_{\widetilde{\Gamma_{2}}}s^j\ee^{-s}s^{\gamma}ds=({1-\ee^{-2\pi \ii \gamma}})\intop_{0}^{\infty}s^j\ee^{-s}s^{\gamma}ds,\quad j\in\mathbb{Z}, \quad j+\re(\gamma)>-1.
        \end{equation}
\end{lemma}
With the aid of \eqref{eq:hook_identity2} we can rewrite our contour integrals in terms of real line integrals under the convergence condition $\re(\alpha)>2n-2-4r_c$ 
\begin{align}
     H_2(\alpha,n)&=(-1)^{n-r_{c}}\intop_{\widetilde{\Gamma_2}}\ldots \intop_{\widetilde{\Gamma_2}}\prod_{\substack{1\leq j< k\leq n-r_c}}(\widetilde{s_j}-\widetilde{s_k})^2\prod_{\substack{l=1}}^{n-r_c}\ee^{-\widetilde{s_l}}\widetilde{s_l}^{\frac{\alpha}{2}+2r_c-n}d\widetilde{s_l}\\&=(-1)^{n-r_{c}}(1-\ee^{-2\pi \ii(\frac{\alpha}{2}+2r_c-n)})^{n-r_{c}}\intop_{0}^{\infty}\ldots \intop_{0}^{\infty}\prod_{1\leq j< k\leq n-r_c}(\widetilde{s_j}-\widetilde{s_k})^2\prod_{\substack{l=1}}^{n-r_c}\ee^{-\widetilde{s_l}}\widetilde{s_l}^{\frac{\alpha}{2}+2r_c-n}d\widetilde{s_l}
\end{align}
Using \cite[\href{http://dlmf.nist.gov/5.14.E5}{(5.14.5)}]{DLMF} we arrive to the formula
\begin{align}
H_2(\alpha,n)=(-1)^{n-r_{c}}\ee^{-\ii\pi(n-r_c) \frac{\alpha}{2}}(2\ii)^{n-r_c}\sin^{n-r_c}\left(\frac{\pi\alpha}{2}\right)\frac{{G(\frac{\alpha}{2}+r_c+1)}G(n+2-r_c)}{G(\frac{\alpha}{2}+2r_c-n+1)}\label{eq:H2_formula}
\end{align}
We repeat the uniqueness of analytic continuation argument above to claim that formula \eqref{eq:H2_formula}  holds for all $\alpha\in\mathbb{C}$. Combining \eqref{eq:H1_formula}, \eqref{eq:H2_formula}, \eqref{eq:hankel_alternative} and the definitions \eqref{def:H1}, \eqref{def:H2} we get
\begin{equation}
       \begin{split}
&\Delta_n(x,\alpha)\sim\frac{(-1)^{\frac{n(n-1)}{2}}}{\pi^{n}}{c_1^{r_c}c_2^{n-r_c}}\left(\frac{x}{2}\right)^{p(r_c,\alpha,n)}(-1)^{n-r_{c}+nr_c}\ee^{-\ii\pi(n-r_c) \frac{\alpha}{2}}\sin^{n}\left(\frac{\pi\alpha}{2}\right)\\&\frac{{G(\frac{\alpha}{2}+r_c+1)}G(n+1-r_c)G\left(1+n-r_c-\frac{\alpha}{2}\right)G(r_c+1)}{G(\frac{\alpha}{2}+2r_c-n+1)G\left(1-\frac{\alpha}{2}-2r_c+n\right)},\quad \text{ as }x\to 0, \quad x>0.
\end{split}\end{equation}
Using definitions \eqref{def:p}, \eqref{def:c1_c2} and formula \eqref{eq:rc} we finish the proof of Theorem \ref{thm:hankel_bessel_det_asymptotic}  for $x>0$. 
\subsection{Proof of Theorem \ref{thm:hankel_bessel_det_asymptotic}}\label{sec: proof_thm_1.1}
We observe that the asymptotics $x\to 0$, $x>0$ of the Bessel function given by \eqref{eq:bessel_series} holds for $-\pi<\arg(x)<\pi$. The same is true for the cylinder function $\mathcal{C}_\nu(x)$. Since the Toeplitz determinant $\Delta_n(x,\alpha)$ is the linear combination of products of cylinder functions, its asymptotic formula is obtained by inserting \eqref{eq:bessel_series} and computing the leading term. As a result, our formula obtained initially for $x>0$ is also valid for the entire sector $-\pi<\arg(x)<\pi$. 

If we try to use the multiple contour integral representation \eqref{eq:tau-formula} to get asymptotics for complex values of $x$, we would notice that it is valid only for $-\frac{\pi}{2}<\arg(x)<\frac{\pi}{2}$. To follow this path further, we would need to use the analytic continuation formulas presented in the Appendix \ref{sec:analytic_continuation}. We checked that this computation confirms Theorem \ref{thm:hankel_bessel_det_asymptotic} and we present it in Appendix \ref{app:alternative_proof_thm_1.1}.  
\section{Asymptotics of special function solutions at zero}
\label{sec:qn_asymptotics}
To deduce the asymptotics of $u_n(x,\alpha)$, we need to use our main result Theorem \ref{thm:hankel_bessel_det_asymptotic} and Proposition \ref{prop:q_n_formula}. We have to shift the index of the Toeplitz determinant $n$ and the parameter $\alpha$. 
\subsection{Piecewise function for the power of \texorpdfstring{$x$}{x} in the asymptotic of \texorpdfstring{$u_n(x,\alpha)$}{un(x,alpha)}.}

Let's introduce notation for the power of $x$ in the asymptotic of $\Delta_n(x,\alpha)$ in Theorem \ref{thm:hankel_bessel_det_asymptotic}
\begin{align}
p_c(\alpha,n)=p(r_c(\alpha,n),\alpha,n)=\begin{cases}
-\frac{\alpha n}{2},&\text{ if }\re(\alpha)>2n-2,\\
(\alpha-2n+2j)j-\frac{\alpha n}{2},&\text{ if }2n-4j-2<\re(\alpha)<2n-4j+2,\\&\text{ for }{j=1,\ldots,n-1}\\
\frac{\alpha n}{2},&\text{ if }\re(\alpha)<-2n+2
\end{cases}\label{eq:pc}
\end{align}
Then the power of $x$ in the asymptotic of $u_n(x,\alpha)$ based on Proposition \ref{prop:q_n_formula} is given by
\begin{align}
e(\alpha,n)=p_c(\alpha-2,n+1)-p_c(\alpha-2,n)+p_c(\alpha,n)-p_c(\alpha,n+1).\label{eq:ec}
\end{align}
\begin{lemma}\label{thm:exp of leading term}
The piecewise function for the power of $x$ in the asymptotic of $u_n(x,\alpha)$ is given by 
\begin{equation}
e(\alpha,n)=
    \begin{cases}
        1 & \text{if } \re(\alpha)>2+2n\\
        \alpha-2n+4j-1 & \text{if } 2n-4j<\re(\alpha)<2n-4j+2 \text{ and }j=0, 1,\ldots,n\\
        -\alpha+2n-4j-1 &\text{if } 2n-4j-2<\re(\alpha)<2n-4j \text{ and }j=0, 1,\ldots,n-1\\
        -1 &\text{if } \re(\alpha)<-2n
    \end{cases}
    \label{eq:e}
\end{equation}
\end{lemma}
\begin{proof}
    We plug in formula \eqref{eq:pc} in the expression \eqref{eq:ec}.\\
    If $\re(\alpha) > 2+2n$, then we have
    \begin{align}
        &p_c(\alpha,n)=-\frac{\alpha n}{2},& p_c(\alpha,n+1)=-\frac{\alpha n}{2}-\frac{\alpha}{2},\\&p_c(\alpha-2,n)={-\frac{\alpha n}{2}+n},&p_c(\alpha-2,n+1)=-\frac{\alpha n}{2}+n+1-\frac{\alpha}{2}.
    \end{align}
    That confirms the first case.
    Similarly, if $2n-4j< \re(\alpha)<2n-4j+2$ and $j=0, 1,\ldots,n$, then we have:
     \begin{align*}
        &p_c(\alpha,n)=(\alpha-2n+2j)j-\frac{\alpha n}{2},\\
        &p_c(\alpha,n+1)=(\alpha-2n+2j)j-2j-\frac{\alpha n}{2}-\frac{\alpha}{2},\\
        &p_c(\alpha-2,n)=(\alpha-2n+2j)j-2j-\frac{\alpha n}{2}+n,\\
        &p_c(\alpha-2,n+1)= (\alpha-2n+2j)j-\frac{\alpha n}{2}-n-1+\frac{\alpha}{2}.
     \end{align*}
     Notice that for $p_c(\alpha-2,n+1)$ we had to shift the index $j$ to $j+1$ in \eqref{eq:pc} to get the correct formula.
    We confirmed the second case.
     If $2n-4j-2<\re(\alpha)<2n-4j$ and $j=0, 1,\ldots,n-1 $, then we have
     \begin{align*}
        &p_c(\alpha,n)=(\alpha-2n+2j)j-\frac{\alpha n}{2},\\
        &p_c(\alpha,n+1)=(\alpha-2n+2j)j-2n+2j-\frac{\alpha n}{2}+\frac{\alpha}{2},\\
        &p_c(\alpha-2,n)=(\alpha-2n+2j)j+\alpha-n+2j-\frac{\alpha n}{2},\\
        &p_c(\alpha-2,n+1)= (\alpha-2n+2j)j-\frac{\alpha n}{2}-n-1+\frac{\alpha}{2}.
     \end{align*}
    Notice that now we had to change the index $j$ to $j+1$ in \eqref{eq:pc} for $p_c(\alpha-2,n+1)$, $p_c(\alpha,n+1)$, and $p_c(\alpha-2,n)$ to obtain the correct formula. We confirmed the third case.
     Finally if $\re(\alpha)<-2n$, then all terms in \eqref{eq:ec} change sign compared to the case $\re(\alpha)>2n-2$. That confirms the last case and finishes the proof of Lemma \ref{thm:exp of leading term}.
\end{proof}
\subsection{Proof of Theorem \ref{thm:q_n_asymptotic}}
In this part, we will compute the coefficients in the asymptotics of Theorem \ref{thm:q_n_asymptotic}. We introduce the notation for the coefficient in the asymptotics of $\Delta_n(x,\alpha)$ in Theorem \ref{thm:hankel_bessel_det_asymptotic}. We included factor of $2$ inside of power of $x$ for convenience
\begin{align}
c(\alpha,n)=\begin{cases}
(-1)^{\frac{n(n+1)}{2}}\left(\frac{d_2}{\pi}\right)^{n}\frac{G(n+1)G(\frac{\alpha}{2}+1)}{G(\frac{\alpha}{2}-n+1)},&\text{ if }\re(\alpha)>2n-2,\\ \\
(-1)^{\frac{n(n-1)}{2}+nj+n-j}\left( \frac{d_2}{\pi}\right)^n\left(\frac{d_1}{d_2}\sin\left(\frac{\pi\alpha}{2}\right)+\cos\left(\frac{\pi\alpha}{2}\right)\right)^j&\text{ if }2n-4j-2<\re(\alpha)<2n-4j+2,\\\times\frac{G(n-j+1)G(-\frac{\alpha}{2}+n-j+1)G(j+1)G(j+\frac{\alpha}{2}+1)}{G(-\frac{\alpha}{2}+n-2j+1)G(\frac{\alpha}{2}-n+2j+1)},&\text{ for }{j=1,\ldots,n-1}\\ \\
\frac{(-1)^{\frac{n(n-1)}{2}+n^2}}{\pi^n}\left(d_1\sin\left(\frac{\pi\alpha}{2}\right)+d_2\cos\left(\frac{\pi\alpha}{2}\right)\right)^{n}&\text{ if }\re(\alpha)<-2n+2\\\times\frac{G(n+1)G(-\frac{\alpha}{2}+1)}{G(-\frac{\alpha}{2}-n+1)},
\end{cases}\label{eq:c}
\end{align}
Let's denote the constant coefficient in the asymptotic of $u_n(x,\alpha)$ as $q(\alpha,n)$. We have
\begin{align}
    q(\alpha,n)=-\frac{c(\alpha-2,n+1)c(\alpha,n)}{c(\alpha-2,n)c(\alpha,n+1)}.\label{def:q_alpha_n}
\end{align}
\begin{lemma}\label{lem:compatible formulae of r_c}
The piecewise formula for $q(\alpha,n)$ is given by
\begin{align}
q(\alpha,n)=\begin{cases}
\frac{2}{2n+2-\alpha},&\text{ if }\re(\alpha)>2n+2,\\ \\
(-1)^{n}\left(\frac{d_1}{d_2}\sin\left(\frac{\pi\alpha}{2}\right)+\cos\left(\frac{\pi\alpha}{2}\right)\right)\left(\frac{\Gamma(-\frac{\alpha}{2}+n-2j+1)}{\Gamma(\frac{\alpha}{2}-n+2j)}\right)^2&\text{ if }2n-4j<\re(\alpha)<2n-4j+2,\\ \times\frac{\Gamma(j+\frac{\alpha}{2})\Gamma(j+1)}{\Gamma(-\frac{\alpha}{2}+n-j+1)\Gamma(n-j+1)},&\text{ for }{j=0,\ldots,n}\\ \\
(-1)^{n}\left(\frac{d_1}{d_2}\sin\left(\frac{\pi\alpha}{2}\right)+\cos\left(\frac{\pi\alpha}{2}\right)\right)^{-1}\left(\frac{\Gamma(\frac{\alpha}{2}-n+2j+1)}{\Gamma(-\frac{\alpha}{2}+n-2j)}\right)^2&\text{ if }2n-4j-2<\re(\alpha)<2n-4j,\\ \times\frac{\Gamma(-\frac{\alpha}{2}+n-j+1)\Gamma(n-j)}{\Gamma(j+\frac{\alpha}{2}+1)\Gamma(j+1)},&\text{ for }{j=0,\ldots,n-1}\\ \\
-\frac{\alpha}{2}-n&\text{ if }\re(\alpha)<-2n
\end{cases}\label{eq:q}
\end{align}
 \end{lemma}
\begin{proof}
    For $\re(\alpha)>2n+2$ we have
    \begin{align}
    c(\alpha,n)&=(-1)^{\frac{n(n+1)}{2}}\left(\frac{d_2}{\pi}\right)^{n}\frac{G(n+1)G(\frac{\alpha}{2}+1)}{G(\frac{\alpha}{2}-n+1)},\\
     c(\alpha-2,n)&=(-1)^{\frac{n(n+1)}{2}}\left(\frac{d_2}{\pi}\right)^{n}\frac{G(n+1)G(\frac{\alpha}{2}+1)}{G(\frac{\alpha}{2}-n+1)}\frac{\Gamma(\frac{\alpha}{2}-n)}{\Gamma(\frac{\alpha}{2})},\\
     c(\alpha,n+1)&=(-1)^{\frac{(n+2)(n+1)}{2}}\left(\frac{d_2}{\pi}\right)^{n+1}\frac{G(n+1)G(\frac{\alpha}{2}+1)}{G(\frac{\alpha}{2}-n+1)}{\Gamma(n+1)}{\Gamma\left(\frac{\alpha}{2}-n\right)},\\
     c(\alpha-2,n+1)&=(-1)^{\frac{(n+2)(n+1)}{2}}\left(\frac{d_2}{\pi}\right)^{n+1}\frac{G(n+1)G(\frac{\alpha}{2}+1)}{G(\frac{\alpha}{2}-n+1)}\frac{\Gamma(n+1)\Gamma\left(\frac{\alpha}{2}-n\right)\Gamma\left(\frac{\alpha}{2}-n-1\right)}{\Gamma\left(\frac{\alpha}{2}\right).}
    \end{align}
    That implies 
    \begin{align}
        q(\alpha,n)=-\frac{\Gamma\left(\frac{\alpha}{2}-n-1\right)}{\Gamma\left(\frac{\alpha}{2}-n\right)}=\frac{2}{2n+2-\alpha}.
    \end{align}
    That confirms the first case.
    
    For $2n-4j<\re(\alpha)<2n-4j+2$ we have 
     \begin{align}
    &c(\alpha,n)=(-1)^{\frac{n(n-1)}{2}+nj+n-j}\left( \frac{d_2}{\pi}\right)^n\left(\frac{d_1}{d_2}\sin\left(\frac{\pi\alpha}{2}\right)+\cos\left(\frac{\pi\alpha}{2}\right)\right)^j\\&\times\frac{G(n-j+1)G(-\frac{\alpha}{2}+n-j+1)G(j+1)G(j+\frac{\alpha}{2}+1)}{G(-\frac{\alpha}{2}+n-2j+1)G(\frac{\alpha}{2}-n+2j+1)},\\
     &c(\alpha-2,n)=(-1)^{\frac{n(n-1)}{2}+nj+n}\left( \frac{d_2}{\pi}\right)^n\left(\frac{d_1}{d_2}\sin\left(\frac{\pi\alpha}{2}\right)+\cos\left(\frac{\pi\alpha}{2}\right)\right)^j\\&\times\frac{G(n-j+1)G(-\frac{\alpha}{2}+n-j+1)G(j+1)G(j+\frac{\alpha}{2}+1)}{G(-\frac{\alpha}{2}+n-2j+1)G(\frac{\alpha}{2}-n+2j+1)}\frac{\Gamma(-\frac{\alpha}{2}+n-j+1)\Gamma(\frac{\alpha}{2}-n+2j)}{\Gamma(j+\frac{\alpha}{2})\Gamma(-\frac{\alpha}{2}+n-2j+1)},\\
     &c(\alpha,n+1)=(-1)^{\frac{n(n+1)}{2}+nj+n+1}\left( \frac{d_2}{\pi}\right)^{n+1}\left(\frac{d_1}{d_2}\sin\left(\frac{\pi\alpha}{2}\right)+\cos\left(\frac{\pi\alpha}{2}\right)\right)^j\\&\times\frac{G(n-j+1)G(-\frac{\alpha}{2}+n-j+1)G(j+1)G(j+\frac{\alpha}{2}+1)}{G(-\frac{\alpha}{2}+n-2j+1)G(\frac{\alpha}{2}-n+2j+1)}\frac{\Gamma(-\frac{\alpha}{2}+n-j+1)\Gamma(n-j+1)\Gamma(\frac{\alpha}{2}-n+2j)}{\Gamma(-\frac{\alpha}{2}+n-2j+1)},\\
     &c(\alpha-2,n+1)=(-1)^{\frac{n(n+1)}{2}+nj+j+2n+2}\left( \frac{d_2}{\pi}\right)^{n+1}\left(\frac{d_1}{d_2}\sin\left(\frac{\pi\alpha}{2}\right)+\cos\left(\frac{\pi\alpha}{2}\right)\right)^{j+1}\\&\times\frac{G(n-j+1)G(-\frac{\alpha}{2}+n-j+1)G(j+1)G(j+\frac{\alpha}{2}+1)}{G(-\frac{\alpha}{2}+n-2j+1)G(\frac{\alpha}{2}-n+2j+1)}{\Gamma(j+1)\Gamma\left(-\frac{\alpha}{2}+n-j+1\right)}{}.
    \end{align}
    Combining these formulas we confirm the second case.   Notice that for $c(\alpha-2,n+1)$ we had to shift index $j$ to $j+1$ in \eqref{eq:c} to get correct formula.
    
    For $2n-4j-2<\re(\alpha)<2n-4j$ we have 
     \begin{align}
    &c(\alpha,n)=(-1)^{\frac{n(n-1)}{2}+nj+n-j}\left( \frac{d_2}{\pi}\right)^n\left(\frac{d_1}{d_2}\sin\left(\frac{\pi\alpha}{2}\right)+\cos\left(\frac{\pi\alpha}{2}\right)\right)^j\\&\times\frac{G(n-j+1)G(-\frac{\alpha}{2}+n-j+1)G(j+1)G(j+\frac{\alpha}{2}+1)}{G(-\frac{\alpha}{2}+n-2j+1)G(\frac{\alpha}{2}-n+2j+1)},\\
     &c(\alpha-2,n)=(-1)^{\frac{n(n-1)}{2}+nj+2n}\left( \frac{d_2}{\pi}\right)^n\left(\frac{d_1}{d_2}\sin\left(\frac{\pi\alpha}{2}\right)+\cos\left(\frac{\pi\alpha}{2}\right)\right)^{j+1}\\&\times\frac{G(n-j+1)G(-\frac{\alpha}{2}+n-j+1)G(j+1)G(j+\frac{\alpha}{2}+1)}{G(-\frac{\alpha}{2}+n-2j+1)G(\frac{\alpha}{2}-n+2j+1)}\frac{\Gamma(j+1)\Gamma(-\frac{\alpha}{2}+n-2j)}{\Gamma(n-j)\Gamma(\frac{\alpha}{2}-n+2j+1)},\\
     &c(\alpha,n+1)=(-1)^{\frac{n(n+1)}{2}+nj+2n+1}\left( \frac{d_2}{\pi}\right)^{n+1}\left(\frac{d_1}{d_2}\sin\left(\frac{\pi\alpha}{2}\right)+\cos\left(\frac{\pi\alpha}{2}\right)\right)^{j+1}\\&\times\frac{G(n-j+1)G(-\frac{\alpha}{2}+n-j+1)G(j+1)G(j+\frac{\alpha}{2}+1)}{G(-\frac{\alpha}{2}+n-2j+1)G(\frac{\alpha}{2}-n+2j+1)}\frac{\Gamma(j+1)\Gamma(j+\frac{\alpha}{2}+1)\Gamma(-\frac{\alpha}{2}+n-2j)}{\Gamma(\frac{\alpha}{2}-n+2j+1)},\\
     &c(\alpha-2,n+1)=(-1)^{\frac{n(n+1)}{2}+nj+2n+j+2}\left( \frac{d_2}{\pi}\right)^{n+1}\left(\frac{d_1}{d_2}\sin\left(\frac{\pi\alpha}{2}\right)+\cos\left(\frac{\pi\alpha}{2}\right)\right)^{j+1}\\&\times\frac{G(n-j+1)G(-\frac{\alpha}{2}+n-j+1)G(j+1)G(j+\frac{\alpha}{2}+1)}{G(-\frac{\alpha}{2}+n-2j+1)G(\frac{\alpha}{2}-n+2j+1)}{\Gamma(j+1)\Gamma\left(-\frac{\alpha}{2}+n-j+1\right)}{}.
    \end{align}
    That confirms the third case.     Notice that now we had to shift index $j$ to $j+1$ in \eqref{eq:c} for $c(\alpha-2,n+1)$, $c(\alpha,n+1)$, and $c(\alpha-2,n)$ to get correct formula.

    Finally for $\re(\alpha)<-2n$ we have 
 \begin{align}
c(\alpha,n)=&\frac{(-1)^{\frac{n(n-1)}{2}+n^2}}{\pi^n}\left(d_1\sin\left(\frac{\pi\alpha}{2}\right)+d_2\cos\left(\frac{\pi\alpha}{2}\right)\right)^{n}\frac{G(n+1)G(-\frac{\alpha}{2}+1)}{G(-\frac{\alpha}{2}-n+1)},\\
c(\alpha-2,n)=&\frac{(-1)^{\frac{n(n-1)}{2}+n^2+n}}{\pi^n}\left(d_1\sin\left(\frac{\pi\alpha}{2}\right)+d_2\cos\left(\frac{\pi\alpha}{2}\right)\right)^{n}\frac{G(n+1)G(-\frac{\alpha}{2}+1)}{G(-\frac{\alpha}{2}-n+1)}\frac{\Gamma(-\frac{\alpha}{2}+1)}{\Gamma(-\frac{\alpha}{2}-n+1)},\\
c(\alpha,n+1)=&\frac{(-1)^{\frac{(n+1)n}{2}+(n+1)^2}}{\pi^{n+1}}\left(d_1\sin\left(\frac{\pi\alpha}{2}\right)+d_2\cos\left(\frac{\pi\alpha}{2}\right)\right)^{n+1}\\&\times\frac{G(n+1)G(-\frac{\alpha}{2}+1)}{G(-\frac{\alpha}{2}-n+1)}{\Gamma(n+1)}{\Gamma\left(-\frac{\alpha}{2}-n\right)},\\
c(\alpha-2,n+1)=&\frac{(-1)^{\frac{(n+1)n}{2}+n+1+(n+1)^2}}{\pi^{n+1}}\left(d_1\sin\left(\frac{\pi\alpha}{2}\right)+d_2\cos\left(\frac{\pi\alpha}{2}\right)\right)^{n+1}\\&\times\frac{G(n+1)G(-\frac{\alpha}{2}+1)}{G(-\frac{\alpha}{2}-n+1)}{\Gamma(n+1)}{\Gamma\left(-\frac{\alpha}{2}+1\right)}.
    \end{align}
        That implies 
    \begin{align}
        q(\alpha,n)=\frac{\Gamma\left(-\frac{\alpha}{2}-n+1\right)}{\Gamma\left(-\frac{\alpha}{2}-n\right)}=-\frac{\alpha}{2}-n.
    \end{align}
    That confirms the last case.
\end{proof}
\section{Asymptotics at infinity}
\subsection{Asymptotics of Toeplitz determinant}\label{sec:toeplitz_asym_at_infty}
To compute the asymptotics at infinity, it is convenient to use the Hankel functions instead of the Bessel functions. We rewrite the formula \eqref{def:f_nu} as \eqref{def:alt_f_nu_hankel}.
Following the same argument as in Theorem \ref{thm:tau-integral-formula} we get the following multiple integral representation using the Andr\`eief identity. 
\begin{theorem} \label{thm:tau-integral-formula-Hankel}
The Toeplitz determinant $\Delta_n(x,\alpha)$ given by \eqref{def:hankel_bessel_determinant} can be rewritten as
    \begin{equation}
\Delta_n(x,\alpha)=\frac{1}{n!}\intop_{\Gamma_3\cup\Gamma_4}\ldots \intop_{\Gamma_3\cup\Gamma_4}\prod_{1\leq j<k\leq n}(t_k-t_j)\left(\frac{1}{t_k}-\frac{1}{t_j}\right)\prod_{k=1}^{n} h_2(t_k)dt_k \label{eq:alt-tau-formula}
        \end{equation}
        where
        \begin{equation*}
        h_2(t)=\frac{\ee^{\frac{x}{2}\left(t-\frac{1}{t}\right)}}{ 2\pi \ii t^{1+\frac{\alpha}{2}}}\left((d_1-\ii d_2)\chi_{\Gamma_3}(t)-(d_1+\ii d_2)\chi_{ \Gamma_4}(t)\right),
        \end{equation*}
       and contours of integration $\Gamma_3$, $\Gamma_4$ are shown on Figure \ref{fig:contour_gamma34}. We use the notation $\chi_{ \Gamma_j}(t)$ for the characteristic function of contour $\Gamma_j$.  We assume $-\pi<\arg(t)<\pi$ on the contours $\Gamma_3$ and $\Gamma_4$.
\end{theorem}

We would like to compute the large $x$ asymptotics of $\Delta_n(x,\alpha)$. We start by using the expansion \eqref{eq:expansion} and implementing notation \eqref{def:b1b2}
  \begin{align}
&\Delta_n(x,\alpha)=\frac{(-1)^{\frac{n(n-1)}{2}}}{n!}\intop_{\Gamma_3\cup\Gamma_4}\ldots \intop_{\Gamma_3\cup\Gamma_4}\prod_{m=1}^{n}(b_1\chi_{\Gamma_3}(t_m)-b_2\chi_{\Gamma_{4}}(t_m))\prod_{1\leq j<k\leq n}\frac{(t_j-t_k)^2}{t_jt_k}\prod_{l=1}^{n}\frac{\ee^{\frac{x}{2}(t_l-\frac{1}{t_l})}}{\pi \ii t_{l}^{\frac{\alpha}{2}+1}}dt_l\\=&\frac{(-1)^{\frac{n(n-1)}{2}}}{(\pi \ii)^{n}n!}\sum_{r=0}^n\sum_{\substack{I\subset\{1,...,n\}\\ \left\vert I\right\vert=r}}b_1^{r}(-b_2)^{n-r}\intop_{\Gamma_3\cup\Gamma_4}\ldots \intop_{\Gamma_3\cup\Gamma_4}{\prod_{\substack{j<k\\j,k\in I}}\frac{(t_j-t_k)^2}{t_jt_k}}{\prod_{\substack{j<k\\j,k\in I^{c}}}\frac{(t_j-t_k)^2}{t_jt_k}}{\prod_{\substack{j\in I\\k\in I^{c}}}\frac{(t_j-t_k)^2}{t_jt_k}}\\&\prod_{l\in I}\chi_{\Gamma_{3}}(t_l)\prod_{l\in I^c}\chi_{\Gamma_{4}}(t_l)\prod_{l\in I}\frac{\ee^{\frac{x}{2}(t_l-\frac{1}{t_l})}}{t_{l}^{\frac{\alpha}{2}+1}}dt_l\prod_{l\in I^{c}}\frac{\ee^{\frac{x}{2}(t_l-\frac{1}{t_l})}}{t_{l}^{\frac{\alpha}{2}+1}}dt_l
\label{eq:delta_n_expansion1}
        \end{align}
We can see that by renaming variables on the right hand side we can guarantee that $I=\{1,2,\ldots,r\}$, and $I^c=\{r+1,r+2,\ldots,n\}$ for each integral in the sum. Combining the same integrals together we get
\begin{align}
&\Delta_n(x,\alpha)=\frac{(-1)^{\frac{n(n-1)}{2}}}{(\pi \ii)^{n}}\sum_{r=0}^n\frac{b_1^{r}(-b_2)^{n-r}}{r!(n-r)!}\intop_{\Gamma_3}\ldots\intop_{\Gamma_3}\intop_{\Gamma_4}\ldots\intop_{\Gamma_4}\prod_{{1\leq j<k\leq r}}\frac{(t_j-t_k)^2}{t_jt_k}{\prod_{r+1\leq j<k\leq n}\frac{(t_j-t_k)^2}{t_jt_k}}\\&\times{\prod_{{j=1}}^r\prod_{{k=r+1}}^n\frac{(t_j-t_k)^2}{t_jt_k}}\prod_{l=1}^n\frac{\ee^{\frac{x}{2}(t_l-\frac{1}{t_l})}}{t_{l}^{\frac{\alpha}{2}+1}}dt_l{},
\label{eq:delta_n_expansion}
\end{align}
where we have $r$ integrals over $\Gamma_3$ and $n-r$ integrals over $\Gamma_4$.

We would like to compute asymptotics of multiple integrals in \eqref{eq:delta_n_expansion} using the steepest descent method for $x\to \infty$, $x>0$. The critical points of the exponent $\Xi(t)=t-\frac{1}{t}$ are $t=\pm \ii$. We chose the contours $\Gamma_3$ and $\Gamma_4$ as the contours of the steepest descent $\im(\Xi(t))=\im(\Xi(\pm\ii))$. They can be described using cubic equations $(\re \,t)^2(\im\, t\pm2)+\im\, t\,(\im\, t\pm1)^2=0$. The main contribution to the $x\to \infty$ asymptotics comes from the neighborhoods of the critical points. We make a change of variable in the local integrals: $(t_l-\ii)=s_l\ee^{\frac{3\pi \ii}{4}}$ for $l\in {I}$ and $(t_l+\ii)=s_l\ee^{-\frac{3\pi \ii}{4}}$ for $l\in {I^{c}}$. After that we replace the local integrals with real line integrals. 
The result is
the following
\begin{align}\label{eq:det_asym_inf_prelim}
&\Delta_n(x,\alpha)\sim{}{}\sum_{r=0}^n\frac{b_1^{r}b_2^{n-r}}{r!(n-r)!}\frac{4^{r(n-r)}}{\pi^n}{\ee^{\frac{\ii\pi}{4}(n^2-4r+6nr-4r^2)}\ee^{\frac{\ii\pi\alpha}{4}(n-2r)}\ee^{\ii x(2r-n)}}\\&\intop_{\mathbb{R}}\ldots\intop_{\mathbb{R}}{\prod_{1\leq j<k\leq r}{(s_j-s_k)^2}{}}\prod_{l=1}^r{\ee^{-\frac{x}{2}s_l^2}}{}ds_l\intop_{\mathbb{R}}\ldots\intop_{\mathbb{R}}{\prod_{1\leq j<k\leq n-r}{(s_j-s_k)^2}{}}{}\prod_{l=1}^{n-r}{\ee^{-\frac{x}{2}s_l^2}}{}ds_l
,\, x\to\infty,\, x>0.\end{align}
We make the change of variables in the integrals $s\to\frac{s}{\sqrt{x}}$ and evaluate the resulting integrals using standard formula \cite[\href{http://dlmf.nist.gov/5.14.E6}{(5.14.6)}]{DLMF} rewritten in terms of Barnes $G$-function.

\begin{align}\label{eq:delta_asym_inf}
\Delta_n(x,\alpha)\sim&\sum_{r=0}^{n}{b_1^{r}b_2^{n-r}}{}\left(\frac{2}{\pi}\right)^{\frac{n}{2}}4^{r(n-r)}\ee^{\frac{\ii\pi}{4}(n^2-4r+6nr-4r^2)}\ee^{\frac{\ii\pi\alpha}{4}(n-2r)}\\&\times G(r+1)G(n-r+1)\ee^{\ii(2r-n)x}x^{-r^2+nr-\frac{n^2}{2}},\quad x\to\infty,\quad x>0.
\end{align}

We notice that this result can be obtained by plugging in asymptotic series of Hankel functions \eqref{eq:hankel1_asym}, \eqref{eq:hankel2_asym} in the Toeplitz determinant $\Delta_n(x,\alpha)$. As a result, we claim that \eqref{eq:delta_asym_inf} holds for $-\pi<\arg(x)<\pi$.

We observe that the multiple contour integral representation \eqref{eq:alt-tau-formula} holds only for $-\frac{\pi}{2}<\arg{x}<\frac{\pi}{2}$. We could use the analytic continuation formulas presented in the Appendix \ref{sec:analytic_continuation} to extend the asymptotics to other values or $\arg(x)$. We confirmed that this computation produces the same result and present it in the Appendix \ref{app:alternative_proof_thm_1.3}. 

To find the leading term for $x>0$ we observe that
\begin{align}
\mathop{\max}\limits_{0\leq r\leq n}\left(-r^2+nr-\frac{n^2}{2}\right)=\begin{cases}
-\frac{n^2}{4}&\mbox{ when $n$--even using $r=\frac{n}{2}$},\\
-\frac{n^2}{4}-\frac{1}{4}&\mbox{ when $n$--odd using $r=\frac{n\pm 1}{2}$}.
    \end{cases}
\end{align}

For even $n$, we plug in $r=\frac{n}{2}$.  We combine factorials with Barnes $G$-function. Since for even $n$ the number $\frac{3n^2-2n}{4}$ is even, we get
\begin{align}
\Delta_n(x,\alpha)\sim \left(\frac{d_1^2+d_2^2}{2\pi}\right)^{\frac{n}{2}}{}{}\left(G\left(\frac{n}{2}+1\right)\right)^2\left(\frac{x}{4}\right)^{-\frac{n^2}{4}}, \quad x\to\infty ,\, x>0
\end{align}

For odd $n$, we need to combine contributions from $r=\frac{n-1}{2}$ and $r=\frac{n+1}{2}$. 
\begin{align}
\Delta_n(x,\alpha)\sim \left(\frac{d_1^2+d_2^2}{2\pi}\right)^{\frac{n}{2}}G\left(\frac{n+1}{2}\right)G\left(\frac{n+3}{2}\right)\left(\frac{x}{4}\right)^{-\frac{n^2+1}{4}}\ee^{\frac{\ii\pi}{4}(3n^2-2n-1)}\\\times\frac{1}{2}\left(\ee^{\ii\frac{\pi}{4}(\alpha+2-n)-\ii x+\ii\phi}+\ee^{\ii\frac{\pi}{4}(-\alpha-2+n)+\ii x-\ii\phi}\right)\\=\left(\frac{d_1^2+d_2^2}{2\pi}\right)^{\frac{n}{2}}G\left(\frac{n+1}{2}\right)G\left(\frac{n+3}{2}\right)\left(\frac{x}{4}\right)^{-\frac{n^2+1}{4}}(-1)^{\frac{n-1}{2}}\sin\left(x-\phi+\frac{\pi}{4}(n-\alpha)\right), \quad x\to\infty ,\, x>0.
\end{align}
where $\phi=\frac{1}{2\ii}\ln\left({d_1+\ii d_2}\right)-\frac{1}{2\ii}\ln\left({d_1-\ii d_2}\right)$ and $(d_1^2+d_2^2)^\frac
{n}{2}=\ee^{\frac{n}{2}\ln(d_1+\ii d_2)+\frac{n}{2}\ln(d_1-\ii d_2)}$. We see that the expression for asymptotics do not depend on the choice of the branch of the logarithm. For real values of $d_1$ and $d_2$ we have $\phi=\arg(d_1+\ii d_2)$.

For $-\pi<\arg(x)<0$ the leading term is given by $r=n$:
\begin{align} \label{eq:delta_asym_inf_lower}
\Delta_n(x,\alpha)\sim&{(d_1-\ii d_2)^{n}}{}\left(\frac{1}{2\pi}\right)^{\frac{n}{2}}\ee^{-\frac{\ii\pi n^2}{4}}G(n+1)\ee^{-\frac{ \ii\pi n \alpha}{4}}\ee^{\ii n x}x^{-\frac{n^2}{2}},\quad x\to\infty
\end{align}
We notice that if $d_1+\ii d_2=0$, then \eqref{eq:delta_asym_inf_lower} holds for $-\pi<\arg(x)<\pi$.

For $0<\arg(x)<\pi$ the leading term is given by $r=0$:
\begin{align}  \label{eq:delta_asym_inf_upper}
\Delta_n(x,\alpha)\sim&(d_1+\ii d_2)^{n}{}\left(\frac{1}{2\pi}\right)^{\frac{n}{2}}\ee^{\frac{\ii\pi n^2}{4}} G(n+1)\ee^{\frac{\ii\pi n \alpha}{4}}\ee^{-\ii n x}x^{-\frac{n^2}{2}},\quad x\to\infty
\end{align}
We notice that if $d_1-\ii d_2=0$, then \eqref{eq:delta_asym_inf_upper} holds for $-\pi<\arg(x)<\pi$.

As the result we get Theorem \ref{thm:hankel_bessel_det_asymptotic_at_infinity}. 
\subsection{Asymptotics of special function solutions}\label{sec:solution_asym_at_infty}
In this part, we still need to use Proposition \ref{prop:q_n_formula} and Theorem \ref{thm:hankel_bessel_det_asymptotic_at_infinity} to get the large $x$ asymptotics of $u_n(x,\alpha)$. We start with the case $x>0$. Fortunately, life becomes much easier in this scenario. Notice that when $n$ is even, the Toeplitz determinant doesn't depend on $\alpha$ at all. Hence when shifting the indices, we have for even $n$:
\begin{align*}
u_n(x,\alpha)&=-\frac{\Delta_{n+1}(x,\alpha-2)}{\Delta_{n+1}(x,\alpha)}\sim-\frac{\sin\left(x-\phi+\frac{\pi}{4}(n+1-\alpha)+\frac{\pi}{2}\right)}{\sin\left(x-\phi+\frac{\pi}{4}(n+1-\alpha)\right)}=-\cot\left(x-\phi+\frac{\pi}{4}(n+1-\alpha)\right)
\end{align*}
Similarly, when $n$ is odd and $n+1$ is even, we get:
\begin{align*}
u_n(x,\alpha)&=-\frac{\Delta_{n}(x,\alpha)}{\Delta_{n}(x,\alpha-2)}\sim-\frac{\sin\left(x-\phi+\frac{\pi}{4}(n-\alpha)\right)}{\sin\left(x-\phi+\frac{\pi}{4}(n-\alpha)+\frac{\pi}{2}\right)}=-\tan\left(x-\phi+\frac{\pi}{4}(n-\alpha)\right)
\end{align*}
        Regarding the asymptotics in the complex plane to get the expression for the exponential error in the case $b_1, b_2\neq 0$ it is sufficient to include extra terms from \eqref{eq:delta_asym_inf}. More specifically, we can get the following asymptotics in the upper halfplane using only terms with $r=0$ and $r=1$ of \eqref{eq:delta_asym_inf}
        \begin{align}
        u_n(x,\alpha)-\ii\sim  \left(\frac{d_1-\ii d_2}{d_1+\ii d_2}\right)\frac{2^{2n-1}}{(n-1)!}\ee^{-\frac{\ii\pi}{2}(n+1+\alpha)}x^{n-1}\ee^{2\ii x },\quad x\to \infty,\quad 0<\arg (x)<\pi,
        \end{align}
Similarly, in the lower half plane we use $r=n$ combined with $r=n-1$ of \eqref{eq:delta_asym_inf} and get \begin{align}
        u_n(x,\alpha)+\ii\sim  \left(\frac{d_1+\ii d_2}{d_1-\ii d_2}\right)\frac{2^{2n-1}}{(n-1)!}\ee^{\frac{\ii\pi}{2}(n-1+\alpha)}x^{n-1}\ee^{-2\ii x },\quad x\to \infty,\quad -\pi<\arg (x)<0.
        \end{align}

If $2b_1=d_1-\ii d_2=0$, then the expression \eqref{eq:alt-tau-formula} for $\Delta_n(x,\alpha)$ contains only integrals over $\Gamma_4$. To get the error term of asymptotics of $\Delta_n(x,\alpha)$ we need to look at the steepest descent procedure in more details. We start with the change of variables 
\begin{align}\label{eq:change_of_variable}
    t_l+\ii=-\frac{s_l}{2}\left(\sqrt{4\ii+s_l^2}+s_l\right).
\end{align}
As the result of this transformation we get
\begin{align}
    \ee^{\frac{x}{2}\left(t_l-\frac{1}{t_l}\right)}=\ee^{-\ii x- \frac{x s_l^2}{2} }.
\end{align}
In the computation of asymptotics of $\Delta_n(x,\alpha)$ we will need the expan sion for $t_l$ as $s_l\to0$ 
\begin{align}
    t_l=-\ii\left(1+\ee^{-\frac{\ii\pi}{4}}s_l-\frac{\ii s_l^2}{2}+\frac{\ee^{-\frac{3\pi\ii}{4}}}{8}s_l^3+\mathcal{O}(s_l^4)\right)
\end{align}
Using it and \eqref{eq:gamma1} we see that
\begin{align}
& \prod_{1\leq j<k\leq n}\frac{1}{t_jt_k}\prod_{l=1}^{n}\frac{1}{t_l^{\frac{\alpha}{2}+1}}=\prod_{l=1}^{n}t_l^{-n-\frac{\alpha}{2}}=\ii^{n^2}\ee^{\frac{\ii\pi\alpha n}{4}}\left(1+\ee^{\frac{3\pi\ii}{4}}\left(n+\frac{\alpha}{2}\right)\sum_{l=1}^ns_l-\ii\left(n+\frac{\alpha}{2}\right)^2\sum_{1\leq j<k\leq n}s_js_k\right.\\&\left.-
\frac{\ii}{2}\left(n+\frac{\alpha}{2}\right)^2\sum_{1=1}^ns_l^2 +\mathcal{O}\left(\sum_{l=1}^n|s_l|^3\right)\right)
\end{align}
Similarly
\begin{align}
\prod_{l=1}^ndt_l=\ee^{-\frac{3\pi\ii n}{4}}\left(1+\ee^{-\frac{\ii\pi}{4}}\sum_{l=1}^ns_l-\ii\sum_{1\leq j<k\leq n}s_js_k-\frac{3\ii}{8}\sum_{1=1}^ns_l^2 +\mathcal{O}\left(\sum_{l=1}^n|s_l|^3\right)\right)\prod_{l=1}^nds_l
\end{align}
\begin{align}
    &\prod_{1\leq j<k\leq n}(t_j-t_k)^2=\ee^{-\frac{3\pi\ii}{4}(n^2-n)}\prod_{1\leq j<k\leq n}\left(1+\ee^{-\frac{\ii\pi}{4}}(s_j+s_k)-\frac{\ii}{2}(s_j^2+s_k^2)-\frac{3\ii}{4}s_js_k+\mathcal{O}(|s_j|^3+|s_k|^3)\right)\\&\times\prod_{1\leq j<k\leq n}(s_j-s_k)^2=\ee^{-\frac{3\pi\ii}{4}(n^2-n)}\prod_{1\leq j<k\leq n}(s_j-s_k)^2\left(1+\ee^{-\frac{\ii\pi}{4}}\sum_{1\leq j<k\leq n}(s_j+s_k)-\frac{\ii}{2}\sum_{1\leq j<k\leq n}(s_j^2+s_k^2)\right.\\&\left.-\frac{3\ii}{4}\sum_{1\leq j<k\leq n}s_js_k-\ii\sum_{\substack{1\leq j<k\leq n\\1\leq l<m\leq n
    \\(j,k)\prec (l,m)}}(s_j+s_k)(s_l+s_m)+\mathcal{O}\left(\sum_{l=1}^n|s_l|^3\right)\right)=\ee^{-\frac{3\pi\ii}{4}(n^2-n)}\prod_{1\leq j<k\leq n}(s_j-s_k)^2\\&\times\left(1+\ee^{-\frac{\ii\pi}{4}}(n-1)\sum_{l=1}^ns_l-\frac{\ii}{2}(n-1)^2\sum_{l=1}^ns_l^2-\frac{\ii}{4}(2n-3)(2n-1)\sum_{1\leq j<k\leq n}s_js_k+\mathcal{O}\left(\sum_{l=1}^n|s_l|^3\right)\right)
\end{align}
where symbol $\prec$ means lexicographic ordering. In the computation above we used the following observation:
\begin{align}
\sum_{\substack{1\leq j<k\leq n\\1\leq l<m\leq n
    \\(j,k)\prec (l,m)}}(s_j+s_k)(s_l+s_m)=\frac{(n-1)^2-(n-1)}{2}\sum_{l=1}^n s_l^2+((n-1)^2-1)\sum_{1\leq j<k\leq n}s_js_k
\end{align}
This identity is the result of the following computation. First, we notice that the left hand side is a symmetric polynomial of degree $2$ in variables $s_j$, $j=1\ldots n$. Therefore, to determine the coefficients near the sums $\sum_{l=1}^n s_l^2$
and $\sum_{1\leq j<k\leq n}s_js_k$ it is sufficient to determine the coefficients near $s_1^2$ and $s_1s_n$. 

The term $s_1^2$ can only be obtained for $j=l=1$ in the sum on the left hand side. Condition $(1,k)\prec (1,m)$ implies that there are $\frac{(n-1)^2-(n-1)}{2}$ values for $k$ and $m$. That determines the first coefficient. 

Condition $(j,k)\prec (l,m)$ implies that the term $s_1s_n$ can only be obtained for $j=1$ and $m=n$ on the left hand side. Indices $k$ and $l$ can take any value except for one choice of $k=n$ and $l=1$ simultaneously. Therefore, there are $(n-1)^2-1$ of them, which determines the second coefficient.

We plug these formulas in \eqref{eq:alt-tau-formula}. We notice that some terms evaluate to zero. We end up with the following result
\begin{align}
\Delta_n(x,\alpha)=\frac{b_2^{n}}{\pi^nn!}{\ee^{\frac{\ii\pi n^2}{4}}\ee^{\frac{\ii\pi n\alpha}{4}}\ee^{-\ii n x}}\left(\intop_{\mathbb{R}}\ldots\intop_{\mathbb{R}}{\prod_{1\leq j<k\leq n}{(s_j-s_k)^2}{}}\prod_{l=1}^n{\ee^{-\frac{x}{2}s_l^2}}{}ds_l\right.\\\left.+\frac{\alpha}{2}\ee^{\frac{3\pi\ii}{4}}\intop_{\mathbb{R}}\ldots\intop_{\mathbb{R}}\sum_{l=1}^ns_l{\prod_{1\leq j<k\leq n}{(s_j-s_k)^2}{}}\prod_{l=1}^n{\ee^{-\frac{x}{2}s_l^2}}{}ds_l\right.
\\\left.+\left(\ii n^2-\ii n+\frac{\ii n \alpha}{2}+\ii\right)\intop_{\mathbb{R}}\ldots\intop_{\mathbb{R}}\left(\sum_{l=1}^ns_l\right)^2{\prod_{1\leq j<k\leq n}{(s_j-s_k)^2}{}}\prod_{l=1}^n{\ee^{-\frac{x}{2}s_l^2}}{}ds_l\right.
\\\left.+\left(-\ii n^2+\ii n-\frac{\ii n \alpha}{2}-\frac{\ii\alpha^2}{8}-\frac{7\ii}{8}\right)\intop_{\mathbb{R}}\ldots\intop_{\mathbb{R}}\sum_{l=1}^ns_l^2{\prod_{1\leq j<k\leq n}{(s_j-s_k)^2}{}}\prod_{l=1}^n{\ee^{-\frac{x}{2}s_l^2}}{}ds_l\right.
\\\left.+2\left(-\ii n^2+\ii n-\frac{\ii n \alpha}{2}-\frac{\ii\alpha^2}{8}-\frac{7\ii}{8}\right)\intop_{\mathbb{R}}\ldots\intop_{\mathbb{R}}\sum_{1\leq j<k\leq n}^ns_js_k{\prod_{1\leq j<k\leq n}{(s_j-s_k)^2}{}}\prod_{l=1}^n{\ee^{-\frac{x}{2}s_l^2}}{}ds_l\right)(1+\mathcal{O}(x^{-1}))\\
=\frac{b_2^{n}}{\pi^nn!}{\ee^{\frac{\ii\pi n^2}{4}}\ee^{\frac{\ii\pi n\alpha}{4}}\ee^{-\ii n x}}\left(\intop_{\mathbb{R}}\ldots\intop_{\mathbb{R}}{\prod_{1\leq j<k\leq n}{(s_j-s_k)^2}{}}\prod_{l=1}^n{\ee^{-\frac{x}{2}s_l^2}}{}ds_l\right.
\\\left.+\frac{\ii(1-\alpha^2)}{8}\intop_{\mathbb{R}}\ldots\intop_{\mathbb{R}}\left(\sum_{l=1}^ns_l\right)^2{\prod_{1\leq j<k\leq n}{(s_j-s_k)^2}{}}\prod_{l=1}^n{\ee^{-\frac{x}{2}s_l^2}}{}ds_l\right)(1+\mathcal{O}(x^{-1}))
\end{align}
To evaluate integrals above, we make the change of variables $t_l=\sqrt{2a}y_l-\frac{b}{\sqrt{2a}}$ in the formula \cite[\href{http://dlmf.nist.gov/5.14.E6}{(5.14.6)}]{DLMF} to get 
\begin{align}\label{eq:dlmf_extra_identity}
    \intop_{-\infty}^{\infty}\ldots\intop_{-\infty}^{\infty}\prod_{1\leq j<k\leq n}(y_j-y_k)^2\prod_{l=1}^n\ee^{-ay_l^2+by_l}dy_l=(2\pi)^{\frac{n}{2}}G(n+2)\ee^{\frac{nb^2}{4a}}(2a)^{-\frac{n^2}{2}}.
\end{align}
Taking the second derivative with respect to $b$ and evaluating $b\to 0$, $a\to \frac{x}{2}$ we arrive at
\begin{align}
    \label{eq:delta_asym_b1_zero}\Delta_n(x,\alpha)=\frac{(d_1+\ii d_2)^{n}}{(2\pi)^\frac{n}{2}}{\ee^{\frac{\ii\pi n^2}{4}}\ee^{\frac{\ii\pi n\alpha}{4}}G(n+1)\ee^{-\ii n x}}x^{-\frac{n^2}{2}}\left(1+\ii\frac{(1-\alpha^2)n}{8x}+\mathcal{O}(x^{-2})\right)
\end{align}
In conclusion
 \begin{align}
        u_n(x,\alpha)-\ii\sim \frac{1-\alpha}{2x},\quad x\to \infty,\quad -\pi<\arg (x)<\pi.
        \end{align}
        For $2b_2=d_1+\ii d_2=0$ the computation is similar. The expression \eqref{eq:alt-tau-formula} for $\Delta_n(x,\alpha)$ contains only integrals over $\Gamma_3$. We start with the change of variables 
\begin{align}
    t_l-\ii=-\frac{s_l}{2}\left(\sqrt{-4\ii+s_l^2}+s_l\right).
\end{align}
We notice that it is just a complex conjugation of \eqref{eq:change_of_variable}. As the result all the computations can be complex conjugated and we get 
\begin{align}
&\Delta_n(x,\alpha)=\frac{(d_1-\ii d_2)^{n}}{(2\pi)^\frac{n}{2}}{\ee^{-\frac{\ii\pi n^2}{4}}\ee^{-\frac{\ii\pi n\alpha}{4}}G(n+1)\ee^{\ii n x}}x^{-\frac{n^2}{2}}\left(1-\ii\frac{(1-\alpha^2)n}{8x}+\mathcal{O}(x^{-2})\right)\\
 &u_n(x,\alpha)+\ii\sim \frac{1-\alpha}{2x},\quad x\to \infty,\quad -\pi<\arg (x)<\pi.
\end{align}
That proves Theorem \ref{thm:q_n_asymptotic_at_infinity}. 

\section*{Acknowledgment}
Most of the work was performed during math REU program at University of Michigan in Summer 2023.  H.P. would like to thank his mentor A.P. for his patient and professional instructions and training during the whole summer. Besides some more advanced contents in differential equations and asymptotic analysis, he also learned research skills in modern mathematics with A.P.. There were many insightful and smart ideas throughout the project, which really motivated and inspired him to explore the related topics deeply. He also would like to thank A.P's help on numerical check of the results and polishing the paper during the following whole year. This is also his first experience to tackle such a big problem, so without A.P.'s assistance, there is no way to generate this paper. He is really appreciated very much. H.P. would also thank Professor Lun Zhang from Fudan University and Professor Peter Miller from University of Michigan for useful discussions and suggestions and for providing us with monodromy data for our solution from \cite{BDM}. A.P. was supported NSF MSPRF grant DMS-2103354,  and RSF grant 22-11-00070.
\printbibliography
\appendix
\section{Bessel equation and its solutions}\label{app:bessel}

      Bessel equation is given by
        \begin{equation}\label{eq:bessel}
        u''(x)+\dfrac{u'(x)}{x}+\left(1-\frac{\nu^2}{x^2}\right)u(x)=0.
    \end{equation}
One of the standard solutions in the form of series representation is given by (see \cite[\href{http://dlmf.nist.gov/10.2.E2}{(10.2.2)}]{DLMF})
        \begin{equation}
J_\nu(x)=\frac{x^{\,\nu}}{2^{\,\nu}}\sum_{k=0}^\infty\frac{(-1)^kx^{\,2k}}{2^{\,2k}k!\Gamma(\nu+k+1)}\label{eq:bessel_series},\quad -\pi<\arg (x)<\pi. 
        \end{equation}
where $\Gamma(x)$ is the Gamma function. 

\subsection{Contour integral representations}
The contour integral representation for  $-\frac{\pi}{2}<\arg(x)<\frac{\pi}{2}$ is given by (see \cite[\href{http://dlmf.nist.gov/10.9.E17}{(10.9.17)}]{DLMF})
\begin{align}
      &J_\nu(x) = \frac{1}{2\pi \ii}\intop_{\infty-\ii\pi}^{\infty+\ii\pi}\ee^{x\sinh(z)-\nu z}dz,\label{eq:jnu}\\
\end{align}
The contour of integration in \eqref{eq:jnu} is shown on Figure \ref{fig:contour_jnu}.
\begin{figure}[ht]
	\centering
\includegraphics{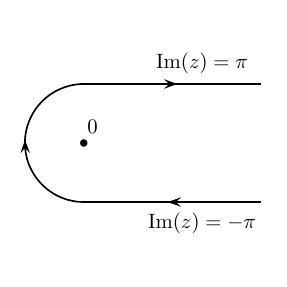}    
        \caption{Contour for $J_\nu(x)$.}\label{fig:contour_jnu}
\end{figure}

\noindent Bessel function of second kind can be written as (see \cite[\href{http://dlmf.nist.gov/10.2.E3}{(10.2.3)} ]{DLMF})
\begin{equation*}
Y_{\nu}(x)=\cot(\pi\nu)J_{\nu}(x)-\csc(\pi\nu)J_{-%
\nu}(x)\end{equation*}
By formula \eqref{eq:jnu} of $J_{\nu}(x)$ it follows that cylinder function $\Cy_{\nu}(x)$ that we introduced in \eqref{def:f_nu} can be written as
 \begin{equation*}
       \Cy_{\nu}(x)=\left(d_1+d_2\cot(\pi\nu)\right)\frac{1}{2\pi \ii}\intop_{\infty-\ii\pi}^{\infty+\ii\pi}\ee^{x\sinh(z)-\nu z}dz-d_2\csc(\pi\nu)\frac{1}{2\pi \ii}\intop_{\infty-\ii\pi}^{\infty+\ii\pi}\ee^{x\sinh(z)+\nu z}dz
    \end{equation*}
Making the change of variable $z\to \ii\pi-z$ in the second integral we get
    \begin{equation*}
       \Cy_{\nu}(x)=\left(d_1+d_2\cot(\pi\nu)\right)\frac{1}{2\pi \ii}\intop_{\infty-\ii\pi}^{\infty+\ii\pi}\ee^{x\sinh(z)-\nu z}dz+d_2\csc(\pi\nu)\frac{\ee^{i \pi\nu}}{2\pi \ii}\intop_{-\infty+2\ii\pi}^{-\infty}\ee^{x\sinh(z)-\nu z}dz
    \end{equation*}
Making another change of variable $\ee^{z}=t$, we get
    \begin{equation}\label{eq:cy_integral_representation}
       \Cy_{\nu}(x)=\intop_{\Gamma_1\cup\Gamma_2}\frac{\ee^{\frac{x}{2}\left(t-\frac{1}{t}\right)}}{2 \pi \ii t^{1+\nu}}(\left(d_1+d_2\cot(\pi\nu)\right)\chi_{\Gamma_1}(t)+d_2\csc(\pi\nu)\ee^{\ii\pi\nu}\chi_{ \Gamma_2}(t))dt=\intop_{\Gamma_1\cup\Gamma_2}h_1(t)dt
    \end{equation}
Here $\Gamma_1$ is contour of integration for $J_\nu(x)$ and $\Gamma_2$ is contour of integration for $J_{-\nu}(x)$ after change of variable $\ee^{z}=t$ shown in the Figure \ref{fig:gamma12Cy}.
\begin{figure*}[ht]
\begin{subfigure}[t]{0.4\textwidth}
\includegraphics{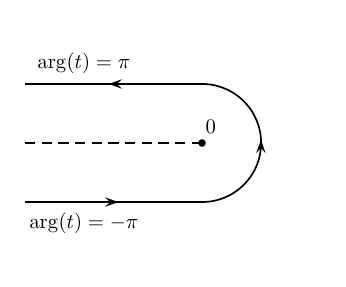}
\caption{Contour $\Gamma_1$}
        \end{subfigure}
        \begin{subfigure}[t]{0.4\textwidth}
\includegraphics{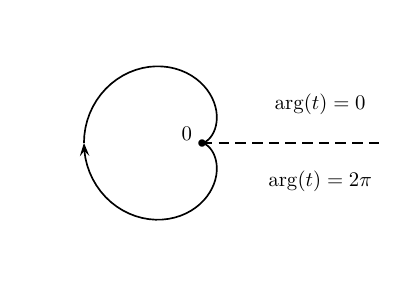}
\caption{Contour $\Gamma_2$}
    \end{subfigure}
    \caption{Contours for $\Cy_\nu(x)$.}
        \label{fig:gamma12Cy}
        \end{figure*}
We should make a remark about power function $\frac{1}{t^{1+\nu}}$. When we use it, we assume $-\pi<\arg(t)<\pi$ on contour $\Gamma_1$ and $0<\arg(t)<2\pi $ on contour $\Gamma_2$.
\begin{remark}\label{rem:integer_alpha}
    For the case $\alpha\in 2\mathbb{Z}$ we could use the formula $Y_n(x)=\dfrac{1}{\pi}\left.\dfrac{\partial J_{\nu}}{\partial\nu}\right|_{\nu=n}+\dfrac{(-1)^n}{\pi}\left.\dfrac{\partial J_{\nu}}{\partial\nu}\right|_{\nu=-n}.$ It would mean that we need to consider orthogonal polynomials with weight $\ee^{-t} t^{a}\ln(t)$. The explicit formula for the corresponding multiple integral can be obtained by using derivatives with respect to parameters, as in \eqref{eq:dlmf_extra_identity}. We use a different approach in Appendix \ref{app:int_alpha_nonzero_d2} by taking limit $\im(\alpha)\to 0$.
\end{remark}
Alternatively, cylinder function $\Cy_\nu(x)$ can be written in terms Hankel functions, which defined as (see \cite[\href{http://dlmf.nist.gov/10.4.E3}{(10.4.3)}]{DLMF})
\begin{align}\label{eq:bessel_to_hankel}
H_\nu^{(1)}(x)=J_\nu(x)+\ii Y_\nu(x),&&
H_\nu^{(2)}(x)=J_\nu(x)-\ii Y_\nu(x).
\end{align}
They admit convenient asymptotic series expansions at infinity. 
\begin{align}\label{eq:hankel1_asym}
    H_\nu^{(1)}(x)\sim \sqrt{\frac{2}{\pi x}} \ee^{\ii x-\frac{\ii\pi\nu}{2}-\frac{\ii\pi}{4}}\sum_{k=0}^\infty \frac{\Gamma(\frac{1}{2}-\nu+k)\Gamma(\frac{1}{2}+\nu+k)}{\Gamma(\frac{1}{2}-\nu)\Gamma(\frac{1}{2}+\nu)(2\ii x)^k k!},\quad x\to\infty,\quad -\pi<\arg(x)<\pi, \\\label{eq:hankel2_asym}
    H_\nu^{(2)}(x)\sim \sqrt{\frac{2}{\pi x}} \ee^{-\ii x+\frac{\ii\pi\nu}{2}+\frac{\ii\pi}{4}}\sum_{k=0}^\infty \frac{\Gamma(\frac{1}{2}-\nu+k)\Gamma(\frac{1}{2}+\nu+k)}{\Gamma(\frac{1}{2}-\nu)\Gamma(\frac{1}{2}+\nu)(-2\ii x)^k k!},\quad x\to\infty,\quad -\pi<\arg(x)<\pi.
\end{align}
The contour integral representations for  $-\frac{\pi}{2}<\arg(x)<\frac{\pi}{2}$ are given by (see \cite[\href{http://dlmf.nist.gov/10.9.E18}{(10.9.18)}]{DLMF})
\begin{align}
      &H^{(1)}_\nu(x) = \frac{1}{\pi \ii}\intop_{-\infty}^{\infty+\ii\pi}\ee^{x\sinh(z)-\nu z}dz,\label{eq:hnu}
       &H^{(2)}_\nu(x) = -\frac{1}{\pi \ii}\intop_{-\infty}^{\infty-\ii\pi}\ee^{x\sinh(z)-\nu z}dz.
\end{align}
The contours of integration in \eqref{eq:hnu} are shown on Figure \ref{fig:contour_hnu}.

\begin{figure*}[ht]
\begin{subfigure}[t]{0.4\textwidth}
\includegraphics{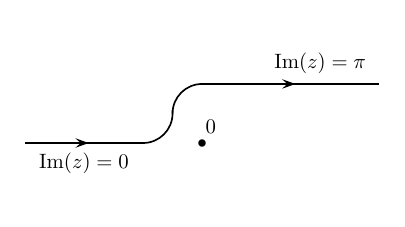}
\caption{Contour for $H_\nu^{(1)}$}
        \end{subfigure}\hspace{0.5cm}
        \begin{subfigure}[t]{0.4\textwidth}
\includegraphics{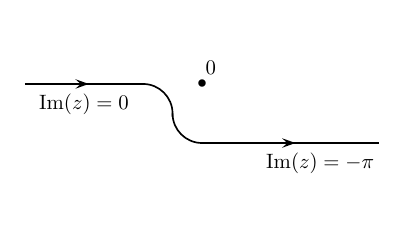}
\caption{Contours for $H_\nu^{(2)}$}
    \end{subfigure}
    \caption{Contours of integration for Hankel functions.}
        \label{fig:contour_hnu}
        \end{figure*}
We would like to also get the alternative integral representation for $\Cy_\nu(x)$. Using \eqref{eq:bessel_to_hankel} we have
\begin{equation}\label{def:alt_f_nu_hankel}
         \Cy_{\nu}(x)=b_1H^{(1)}_\nu(x)+b_2H^{(2)}_{\nu}(x)
     \end{equation}
where
\begin{equation}\label{def:b1b2}
    b_1=\frac{d_1-\ii d_2}{2},\quad b_2=\frac{d_1+\ii d_2}{2}.
\end{equation}
Making a change of variable $\ee^{z}=t$, we get
    \begin{equation}\label{eq:alt_cy_integral_representation}
       \Cy_{\nu}(x)=\intop_{\Gamma_3\cup\Gamma_4}\frac{\ee^{\frac{x}{2}\left(t-\frac{1}{t}\right)}}{2 \pi \ii t^{1+\nu}}((d_1-\ii d_2)\chi_{\Gamma_3}(t)-(d_1+\ii d_2)\chi_{ \Gamma_4}(t))dt=\intop_{\Gamma_3\cup\Gamma_4}h_2(t)dt,
    \end{equation}
    where contours $\Gamma_3$ and $\Gamma_4$ are shown on the Figure \ref{fig:contour_gamma34}.
    \begin{figure*}[ht]
\begin{subfigure}[t]{0.4\textwidth}
\includegraphics{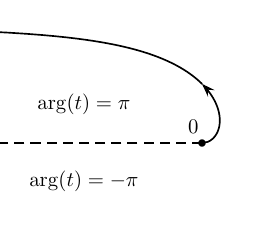}
\caption{Contour $\Gamma_3$}
        \end{subfigure}
        \begin{subfigure}[t]{0.4\textwidth}
\includegraphics{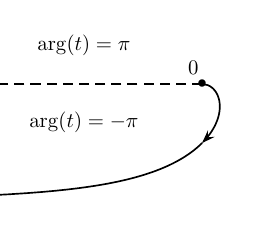}
\caption{Contour $\Gamma_4$}
    \end{subfigure}
    \caption{Alternative contours of integration for $\Cy_\nu(x)$.}
        \label{fig:contour_gamma34}
        \end{figure*}

`\subsection{Differential identities for cylinder functions}\label{sec:bessel}
     Cylinder functions $\Cy_{\nu}(x)$ satisfy the following differential identities (see \cite[\href{http://dlmf.nist.gov/10.6.E2}{(10.6.2)}]{DLMF}):
  \begin{equation}\label{eq:diff-iden1}
        \Cy_\nu'(x)=\frac{\nu}{x}\Cy_\nu(x)-\Cy_{\nu+1}(x),
    \end{equation}
    \begin{equation}\label{eq:diff-iden2}
        \Cy_\nu'(x)=\Cy_{\nu-1}(x)-\frac{\nu}{x}\Cy_\nu(x).
    \end{equation}
    
 \subsection{Analytic continuation}
\label{sec:analytic_continuation}
Cylinder functions satisfy analytic continuation formulas \cite[\href{http://dlmf.nist.gov/10.11}{\S 10.11}]{DLMF}. We rewrite them in a more convenient form.
\begin{align}
   & J_{\nu}(x)=\ee^{-\ii\pi\nu}J_{\nu}(\ee^{\ii\pi}x),\quad Y_{\nu}(x)=\ee^{\ii\pi\nu}Y_{\nu}(\ee^{\ii\pi}x)-2\ii\cos(\pi\nu)J_\nu(\ee^{\ii\pi}x),\label{eq:j_plus}\\
    & J_{\nu}(x)=\ee^{\ii\pi\nu}J_{\nu}(\ee^{-\ii\pi}x),\quad Y_{\nu}(x)=\ee^{-\ii\pi\nu}Y_{\nu}(\ee^{-\ii\pi}x)+2\ii\cos(\pi\nu)J_\nu(\ee^{-\ii\pi}x),\label{eq:j_minus}\\
     & H_{\nu}^{(1)}(x)=\csc(\pi\nu)\sin(2\pi\nu)H_{\nu}^{(1)}(\ee^{\ii\pi}x)+\ee^{-\ii\pi\nu}  H_{\nu}^{(2)}(\ee^{\ii\pi}x),\quad H_{\nu}^{(2)}(x)=-\ee^{\ii\pi\nu}  H_{\nu}^{(1)}(\ee^{\ii\pi}x),\label{eq:h_plus}\\
    &H_{\nu}^{(1)}(x)=-\ee^{-\ii\pi\nu}  H_{\nu}^{(2)}(\ee^{-\ii\pi}x),\quad H_{\nu}^{(2)}(x)=\csc(\pi\nu)\sin(2\pi\nu)H_{\nu}^{(2)}(\ee^{-\ii\pi}x)+\ee^{\ii\pi\nu}  H_{\nu}^{(1)}(\ee^{-\ii\pi}x).\label{eq:h_minus}\\
\end{align}
It implies that
\begin{align}
\mathcal{C}_\nu(x)&=d_1^+ J_{\nu}(\ee^{\ii\pi}x)+d_2^+ Y_{\nu}(\ee^{\ii\pi}x)\label{eq:C_J_+}\\&=d_1^- J_{\nu}(\ee^{-\ii\pi}x)+d_2^- Y_{\nu}(\ee^{-\ii\pi}x)\label{eq:C_J_-}\\&=b_1^+ H_{\nu}^{(1)}(\ee^{\ii\pi}x)+b_2^+ H_{\nu}^{(2)}(\ee^{\ii\pi}x)\label{eq:C_H_+}\\&=b_1^- H_{\nu}^{(1)}(\ee^{-\ii\pi}x)+b_2^- H_{\nu}^{(2)}(\ee^{-\ii\pi}x)\label{eq:C_H_-}.
\end{align}
where
\begin{align}\label{eq:d+}
&d_1^{+}=\ee^{-\ii\pi\nu}d_1- 2\ii\cos(\pi\nu)d_2,\quad d_2^+=\ee^{\ii\pi\nu}d_2,\\\label{eq:d-}
&d_1^{-}=\ee^{\ii\pi\nu}d_1+ 2\ii\cos(\pi\nu)d_2,\quad d_2^-=\ee^{-\ii\pi\nu}d_2,\\
&\label{eq:b_+}b_1^{+}=\csc(\pi\nu)\sin(2\pi\nu)b_1-\ee^{\ii\pi\nu}b_2,\quad b_2^+=\ee^{-\ii\pi\nu}b_1,\\ 
&\label{eq:b_-}b_1^{-}=\ee^{\ii\pi\nu}b_2,\quad b_2^-=\csc(\pi\nu)\sin(2\pi\nu)b_2-\ee^{-\ii\pi\nu}b_1.
\end{align}

\section{Alternative proofs of main theorems using analytic continuation formulas}
As the general idea stated in Section \ref{sec: proof_thm_1.1} and \ref{sec:toeplitz_asym_at_infty}, we can first notice that our method of multiple contour integral representation actually works for $-\frac{\pi}{2}<\arg(x)<\frac{\pi}{2}$. On the next step we apply the analytic continuation formulas from Appendix \ref{sec:analytic_continuation} to the contour integral representations of cylinder functions and extend our results from $-\frac{\pi}{2}<\arg(x)<\frac{\pi}{2}$ to $-\pi<\arg{x}<\pi$. Since the detailed proof along this path involves fairly prolix computations, we presented it in this Appendix as a reference.
\subsection{Alternative proof of Theorem \ref{thm:hankel_bessel_det_asymptotic}}
\label{app:alternative_proof_thm_1.1}
In the first step, we need to extend the asymptotic formula from $x>0$ to $-\frac{\pi}{2}<\arg(x)<\frac{\pi}{2}$. We can notice that multiple contour integral representation \eqref{eq:tau-formula} holds for $-\frac{\pi}{2}<\arg(x)<\frac{\pi}{2}$. Looking in more detail in the proof of Theorem \ref{thm:hankel_bessel_det_asymptotic}, we can notice that the resulting contours after the changes of variables $t_l=\frac{2}{x}s_l$ or $t_l=\frac{x}{2}s_l$  can be deformed back to $\Gamma_1$ and $\Gamma_2$. This is possible due to the exponential decay of the integrand in the halfplane $\re(s)<0$ for the contour $\Gamma_1$ and in the halfplane $\re(s)>0$ for the contour $\Gamma_2$. We perform the deformation in several steps: first, we deform one side of the contour, then move the branch cut, then move the second part of the contour. 

Now we are ready to extend our result to $-\pi<\arg(x)<\pi$. In the first step, we assume $-\pi<\arg(x)<-\frac{\pi}{2}$, which implies $0<\arg(\ee^{\ii\pi}x)<\frac{\pi}{2}$. To obtain the asymptotic formula of the Toeplitz determinant $\Delta_{n}(x,\alpha)$ at zero, we can apply the analytic continuation formulas \eqref{eq:j_plus}. Using them, we express the cylinder function $\mathcal{C}_{\nu}(x)$ as \eqref{eq:C_J_+} with $d_1^+$ and $d_2^+$ given by \eqref{eq:d+}.

Since $0<\arg(\ee^{\ii\pi}x)<\frac{\pi}{2}$ we can substitute $d_{1}^{+}$, $d_{2}^{+}$ and $\ee^{\ii\pi}x$ with $\nu=\frac{\alpha}{2}$ into the asymptotic formulas derived earlier.
If $d_2\neq 0$ and $\re(\alpha)>2n-2 \text{ or }d_1\sin\left(\frac{\pi\alpha}{2}\right)+d_2\cos\left(\frac{\pi\alpha}{2}\right)=0,\text{ as }x\to 0,-\pi<\arg(x)<-\frac{\pi}{2}$, we will get:
\begin{align}
\Delta_n(x,\alpha)&\sim (-1)^{\frac{n(n+1)}{2}}\left(\frac{d_2^{+}}{\pi}\right)^{n}\frac{G(n+1)G(\frac{\alpha}{2}+1)}{G(\frac{\alpha}{2}-n+1)}\left(\frac{\ee^{\ii\pi}x}{2}\right)^{-\frac{n\alpha}{2}}\\&\sim(-1)^{\frac{n(n+1)}{2}}\ee^{\frac{\ii\pi n \alpha}{2}}\left(\frac{d_2}{\pi}\right)^{n}\frac{G(n+1)G(\frac{\alpha}{2}+1)}{G(\frac{\alpha}{2}-n+1)}\ee^{-\frac{\ii\pi n \alpha}{2}}\left(\frac{x}{2}\right)^{-\frac{n\alpha}{2}}\\&\sim (-1)^{\frac{n(n+1)}{2}}\left(\frac{d_2}{\pi}\right)^{n}\frac{G(n+1)G(\frac{\alpha}{2}+1)}{G(\frac{\alpha}{2}-n+1)}\left(\frac{x}{2}\right)^{-\frac{n\alpha}{2}}
\end{align}
If $d_2\neq 0$, $d_1\sin\left(\frac{\pi\alpha}{2}\right)+d_2\cos\left(\frac{\pi\alpha}{2}\right)\neq 0$,  and $2n-4j<\re(\alpha)<2n-4j+2 \text{ for some }j=0, 1,\ldots,n,\text{ as }x\to 0,-\pi<\arg(x)<-\frac{\pi}{2}$, we will get:
\begin{align*}&\Delta_n(x,\alpha)\sim (-1)^{\frac{n(n-1)}{2}+nj+n-j}\left( \frac{d_2^{+}}{\pi}\right)^n\left(\frac{d_1^{+}}{d_2^{+}}\sin\left(\frac{\pi\alpha}{2}\right)+\cos\left(\frac{\pi\alpha}{2}\right)\right)^j\\&\times\frac{G(n-j+1)G(-\frac{\alpha}{2}+n-j+1)G(j+1)G(j+\frac{\alpha}{2}+1)}{G(-\frac{\alpha}{2}+n-2j+1)G(\frac{\alpha}{2}-n+2j+1)}\left(\frac{\ee^{\ii\pi}x}{2}\right)^{(\alpha-2n+2j)j-\frac{n\alpha}{2}}\\ \\&\sim(-1)^{\frac{n(n-1)}{2}+nj+n-j}\ee^{\frac{\ii\pi n \alpha}{2}}\left( \frac{d_2}{\pi}\right)^n\left(\ee^{-\ii\pi\alpha}\frac{d_1}{d_2}\sin\left(\frac{\pi\alpha}{2}\right)-2\ii\sin\left(\frac{\pi\alpha}{2}\right)\ee^{-\frac{\ii\pi \alpha}{2}}\cos\left(\frac{\pi\alpha}{2}\right)+\cos\left(\frac{\pi\alpha}{2}\right)\right)^j\\&\times\frac{G(n-j+1)G(-\frac{\alpha}{2}+n-j+1)G(j+1)G(j+\frac{\alpha}{2}+1)}{G(-\frac{\alpha}{2}+n-2j+1)G(\frac{\alpha}{2}-n+2j+1)}\left(\frac{x}{2}\right)^{(\alpha-2n+2j)j-\frac{n\alpha}{2}}\ee^{\ii\pi((\alpha-2n+2j)j-\frac{n\alpha}{2})}\\ \\&\sim(-1)^{\frac{n(n-1)}{2}+nj+n-j}\left( \frac{d_2}{\pi}\right)^n\left(\frac{d_1}{d_2}\sin\left(\frac{\pi\alpha}{2}\right)+\cos\left(\frac{\pi\alpha}{2}\right)\right)^j\ee^{\frac{\ii\pi n \alpha}{2}-\ii\pi\alpha j+\ii\pi((\alpha-2n+2j)j-\frac{n\alpha}{2})}\\&\times\frac{G(n-j+1)G(-\frac{\alpha}{2}+n-j+1)G(j+1)G(j+\frac{\alpha}{2}+1)}{G(-\frac{\alpha}{2}+n-2j+1)G(\frac{\alpha}{2}-n+2j+1)}\left(\frac{x}{2}\right)^{(\alpha-2n+2j)j-\frac{n\alpha}{2}}\\ \\&\sim(-1)^{\frac{n(n-1)}{2}+nj+n-j}\left( \frac{d_2}{\pi}\right)^n\left(\frac{d_1}{d_2}\sin\left(\frac{\pi\alpha}{2}\right)+\cos\left(\frac{\pi\alpha}{2}\right)\right)^j\\&\times\frac{G(n-j+1)G(-\frac{\alpha}{2}+n-j+1)G(j+1)G(j+\frac{\alpha}{2}+1)}{G(-\frac{\alpha}{2}+n-2j+1)G(\frac{\alpha}{2}-n+2j+1)}\left(\frac{x}{2}\right)^{(\alpha-2n+2j)j-\frac{n\alpha}{2}}
\end{align*}
If $d_1\sin\left(\frac{\pi\alpha}{2}\right)+d_2\cos\left(\frac{\pi\alpha}{2}\right)\neq 0$,  and $\re(\alpha)<-2n \text{ or }d_2=0,\text{ as }x\to 0,-\pi<\arg(x)<-\frac{\pi}{2}$, we will get:
\begin{align*}
&\Delta_n(x,\alpha)\sim \frac{(-1)^{\frac{n(n+1)}{2}}}{\pi^n}\left(d_1^{+}\sin\left(\frac{\pi\alpha}{2}\right)+d_2^{+}\cos\left(\frac{\pi\alpha}{2}\right)\right)^{n}\frac{G(n+1)G(-\frac{\alpha}{2}+1)}{G(-\frac{\alpha}{2}-n+1)}\left(\frac{\ee^{\ii\pi}x}{2}\right)^{\frac{n\alpha}{2}}\\&\sim \frac{(-1)^{\frac{n(n+1)}{2}}}{\pi^n}\left(\ee^{-\frac{\ii\pi\alpha}{2}}d_1\sin\left(\frac{\pi\alpha}{2}\right)-2\ii\sin\left(\frac{\pi\alpha}{2}\right)d_2\cos\left(\frac{\pi\alpha}{2}\right)+\ee^{\frac{\ii\pi\alpha}{2}}d_2\cos\left(\frac{\pi\alpha}{2}\right)\right)^{n}\\&\times\frac{G(n+1)G(-\frac{\alpha}{2}+1)}{G(-\frac{\alpha}{2}-n+1)}\left(\frac{x}{2}\right)^{\frac{n\alpha}{2}}\ee^{\frac{\ii\pi n \alpha}{2}}\\&\sim \frac{(-1)^{\frac{n(n+1)}{2}}}{\pi^n}\ee^{-\frac{\ii\pi n\alpha}{2}}\left(d_1\sin\left(\frac{\pi\alpha}{2}\right)+d_2\cos\left(\frac{\pi\alpha}{2}\right)\right)^{n}\frac{G(n+1)G(-\frac{\alpha}{2}+1)}{G(-\frac{\alpha}{2}-n+1)}\left(\frac{x}{2}\right)^{\frac{n\alpha}{2}}\ee^{\frac{\ii\pi n \alpha}{2}}\\&\sim\frac{(-1)^{\frac{n(n+1)}{2}}}{\pi^n}\left(d_1\sin\left(\frac{\pi\alpha}{2}\right)+d_2\cos\left(\frac{\pi\alpha}{2}\right)\right)^{n}\frac{G(n+1)G(-\frac{\alpha}{2}+1)}{G(-\frac{\alpha}{2}-n+1)}\left(\frac{x}{2}\right)^{\frac{n\alpha}{2}}.
\end{align*}
Therefore, the piecewise formula for sector $-\pi<\arg(x)<-\frac{\pi}{2}$ is consistent with sector $-\frac{\pi}{2}<\arg(x)<0$ and we can certainly extend the validity of our result to sector $-\pi<\arg(x)<0$. 

Similarly, for $\frac{\pi}{2}<\arg(x)<\pi$, we get $-\frac{\pi}{2}<\arg(\ee^{-\ii\pi}x)<0$. We can apply the analytic continuation formulas \eqref{eq:j_minus}.
Consequently, the cylinder function $\mathcal{C}_{\nu}(x)$ can be expressed as \eqref{eq:C_J_-} with $d_1^-$, $d_2^-$ given by \eqref{eq:d-}.
Since  $-\frac{\pi}{2}<\arg(\ee^{-\ii\pi}x)<0$ we can substitute $d_{1}^{-}$,$d_{2}^{-}$ and $e^{-\ii\pi}x$ into the asymptotic formulas obtained earlier. After almost the same simplifications, we will arrive at the conclusion that the piecewise formula on the sector $\frac{\pi}{2}<\arg(x)<\pi$ is consistent with the sector $0<\arg(x)<\frac{\pi}{2}$ and we can also extend the validity of our result to the sector $0<\arg(x)<\pi$. It follows that our result holds for the entire sector $-\pi<\arg(x)<\pi$. That proves Theorem \ref{thm:hankel_bessel_det_asymptotic}.
\subsection{Alternative proof of parts (\ref{thm: toeplitz_infinity_part_3})-(\ref{thm: toeplitz_infinity_part_6}) of Theorem \ref{thm:hankel_bessel_det_asymptotic_at_infinity}}\label{app:alternative_proof_thm_1.3}
Again, we start with extending the asymptotics from $x>0$ to $-\frac{\pi}{2}<\arg(x)<\frac{\pi}{2}$. The multiple contour integral representation \eqref{eq:alt-tau-formula} holds for $-\frac{\pi}{2}<\arg(x)<\frac{\pi}{2}$. We proceed with a steepest descent analysis of \eqref{eq:delta_n_expansion}. The steepest descent contours are given by
$\im(\ee^{\ii\arg(x)}\Xi(t))=\im(\ee^{\ii\arg(x)}\Xi(\pm\ii))$. Since $-\frac{\pi}{2}<\arg(x)<\frac{\pi}{2}$ the integration over these new contours keeps the integral finite. The main contribution to the asymptotics is provided by critical points. The local changes of the variables are given by $(t_l-\ii)=s_l\ee^{-\frac{\ii\arg(x)}{2}}\ee^{\frac{3\pi \ii}{4}}$ for $l\in {I}$ and $(t_l+\ii)=s_l\ee^{-\frac{\ii\arg(x)}{2}}\ee^{-\frac{3\pi \ii}{4}}$ for $l\in {I^{c}}$.  As the result we get 
\begin{align}
&\Delta_n(x,\alpha)\sim{}{}\sum_{r=0}^n\frac{b_1^{r}b_2^{n-r}}{r!(n-r)!}\frac{4^{r(n-r)}}{\pi^n}(-1)^{\frac{(n-r)(n+r+1)+n(n-1)}{2}}\ii^{\frac{r(r-1)+(n-r)(n-r-1)-2n}{2}}\ee^{\frac{\pi \ii}{4}(2r-n)(1-\alpha)}\ee^{\ii x(2r-n)}\\&\times \ee^{\ii \arg(x)(-r^2+nr-\frac{n^2}{2})}\intop_{\mathbb{R}}\ldots\intop_{\mathbb{R}}{\prod_{1\leq j<k\leq r}{(s_j-s_k)^2}{}}\prod_{l=1}^r{\ee^{-\frac{|x|}{2}s_l^2}}{}ds_l\intop_{\mathbb{R}}\ldots\intop_{\mathbb{R}}{\prod_{1\leq j<k\leq n-r}{(s_j-s_k)^2}{}}{}\prod_{l=1}^{n-r}{\ee^{-\frac{|x|}{2}s_l^2}}{}ds_l
,\,\\& x\to\infty,\, -\frac{\pi}{2}<\arg(x)<\frac{\pi}{2}.
\end{align}
After the change of variable $s\to \frac{s}{\sqrt{|x|}}$ we recover \eqref{eq:delta_asym_inf} in the halfplane $-\frac{\pi}{2}<\arg(x)<\frac{\pi}{2}$, which implies the validity of parts (\ref{thm: toeplitz_infinity_part_3})-(\ref{thm: toeplitz_infinity_part_6}) of Theorem \ref{thm:hankel_bessel_det_asymptotic_at_infinity} in the same domain.

Now we are ready to the second step of our program. We take $-\pi<\arg(x)<-\frac{\pi}{2}$ which implies $0<\arg(\ee^{\ii\pi}x)<\frac{\pi}{2}$. To get the formula of the Toeplitz determinant $\Delta_{n}(x,\alpha)$ at infinity, we can apply the analytic continuation formulas \eqref{eq:h_plus}. 
Correspondingly, the cylinder function $\mathcal{C}_{\nu}(x)$ can be expressed as \eqref{eq:C_H_+} with $b_1^{+}$ and $b_2^{+}$ given by \eqref{eq:b_+}.
We substitute $b_1^{+}$, $b_2^{+}$ and $\ee^{\ii\pi}x$ into (\ref{eq:delta_asym_inf}), then the asymptotic formula of the Teoplitz determinant on the sector $-\pi<\arg(x)<-\frac{\pi}{2}$ is obtained as follows:
\begin{align}
\Delta_n(x,\alpha)\sim&\sum_{r=0}^{n}{(b_1^{+})^{r}(b_2^{+})^{n-r}}\left(\frac{2}{\pi}\right)^{\frac{n}{2}}4^{r(n-r)}(-1)^{\frac{(n-r)(n+r+1)+n(n-1)}{2}}\ii^{\frac{r(r-1)+(n-r)(n-r-1)-2n}{2}}\\&\times G(r+1)G(n-r+1)\ee^{\frac{\pi \ii}{4}(2r-n)(1-\alpha)}\ee^{-\ii(2r-n)x}(\ee^{\ii\pi}x)^{-r^2+nr-\frac{n^2}{2}},\quad x\to\infty.
\end{align}
The leading term is exponentially growing and is given by $r=0$:
\begin{align*}
    \Delta_n(x,\alpha)&\sim (b_2^{+})^n\left(\frac{2}{\pi}\right)^{\frac{n}{2}}(-1)^{\frac{n(n+1)+n(n-1)}{2}}\ii^{\frac{n^2-3n}{2}} G(n+1)\ee^{-\frac{\pi \ii}{4}n(1-\alpha)}\ee^{\ii n x}(\ee^{\ii\pi}x)^{-\frac{n^2}{2}},\quad x\to\infty.\\&\sim (b_2)^n\ee^{-\frac{\ii n\pi\alpha}{2}}\left(\frac{2}{\pi}\right)^{\frac{n}{2}}(-1)^{\frac{n(n-1)}{2}}\ee^{\frac{\ii\pi n(n+1)}{2}}\ii^{\frac{n^2-3n}{2}} G(n+1)\ee^{-\frac{\pi \ii}{4}n(1-\alpha)}\ee^{\ii n x}x^{-\frac{n^2}{2}}\ee^{-\frac{\ii\pi n^2}{2}},\quad x\to\infty.\\&\sim (b_2)^n\left(\frac{2}{\pi}\right)^{\frac{n}{2}}(-1)^{\frac{n(n-1)}{2}}\ii^{\frac{n^2-3n}{2}} G(n+1)\ee^{\ii n x}x^{-\frac{n^2}{2}}\ee^{\frac{\pi \ii}{4}n(1-\alpha)},\quad x\to\infty.\\&\sim{(d_1-\ii d_2)^{n}}{}\left(\frac{1}{2\pi}\right)^{\frac{n}{2}}(-1)^{\frac{n(n-1)}{2}}\ii^{\frac{n^2-3n}{2}}G(n+1)\ee^{\frac{\pi \ii}{4}n(1-\alpha)}\ee^{\ii n x}x^{-\frac{n^2}{2}},\quad x\to\infty
\end{align*}
Therefore, the formula on the sector $-\pi<\arg(x)<-\frac{\pi}{2}$ is consistent with the sector $-\frac{\pi}{2}<\arg(x)<0$ and it implies that the validity of our result can be extended to the sector $-\pi<\arg(x)<0$. That proves formula (\ref{eq:delta_asym_inf_lower}). 

On the other hand, let $\frac{\pi}{2}<\arg(x)<\pi$, then $-\frac{\pi}{2}<\arg(\ee^{-\ii\pi}x)<0$. We can apply the analytic continuation formulas \eqref{eq:h_minus}. Consequently, the cylinder function $\mathcal{C}_{\nu}(x)$ can be expressed as \eqref{eq:C_H_-} with $b_1^{-}$ and $b_2^{-}$ given by \eqref{eq:b_-}. We substitute $b_1^{-}$, $b_2^{-}$ and $\ee^{-\ii\pi}x$ into (\ref{eq:delta_asym_inf}), then the asymptotic formula of the Teoplitz determinant on the sector $\frac{\pi}{2}<\arg(x)<\pi$ is obtained as follows:
\begin{align}
\Delta_n(x,\alpha)\sim&\sum_{r=0}^{n}{(b_1^{-})^{r}(b_2^{-})^{n-r}}\left(\frac{2}{\pi}\right)^{\frac{n}{2}}4^{r(n-r)}(-1)^{\frac{(n-r)(n+r+1)+n(n-1)}{2}}\ii^{\frac{r(r-1)+(n-r)(n-r-1)-2n}{2}}\\&\times G(r+1)G(n-r+1)\ee^{-\frac{\pi \ii}{4}(2r-n)(1-\alpha)}\ee^{-\ii(2r-n)x}(\ee^{-\ii\pi}x)^{-r^2+nr-\frac{n^2}{2}},\quad x\to\infty.
\end{align}
The leading term is exponentially growing and is given by $r=n$:
\begin{align*}
    \Delta_n(x,\alpha)&\sim (b_1^{-})^n\left(\frac{2}{\pi}\right)^{\frac{n}{2}}(-1)^{\frac{n(n-1)}{2}}\ii^{\frac{n^2-3n}{2}} G(n+1)\ee^{-\frac{\pi \ii}{4}n(1-\alpha)}\ee^{-\ii n x}(\ee^{\ii\pi}x)^{-\frac{n^2}{2}},\quad x\to\infty.\\&\sim (b_1)^n\ee^{\frac{\ii n\pi\alpha}{2}}\left(\frac{2}{\pi}\right)^{\frac{n}{2}}\ee^{\frac{\ii\pi n(n-1)}{2}}\ii^{\frac{n^2-3n}{2}} G(n+1)\ee^{-\frac{\pi \ii}{4}n(1-\alpha)}\ee^{-\ii n x}x^{-\frac{n^2}{2}}\ee^{\frac{\ii\pi n^2}{2}},\quad x\to\infty.\\&\sim (b_1)^n\left(\frac{2}{\pi}\right)^{\frac{n}{2}}(-1)^{n^2}\ii^{\frac{n^2-3n}{2}} G(n+1)\ee^{\ii n x}x^{-\frac{n^2}{2}}\ee^{-\frac{\pi \ii}{4}n(1-\alpha)},\quad x\to\infty.\\&\sim(d_1+\ii d_2)^{n}{}\left(\frac{1}{2\pi}\right)^{\frac{n}{2}}(-1)^{n^2}\ii^{\frac{n^2-3n}{2}} G(n+1)\ee^{-\frac{\pi \ii}{4}n(1-\alpha)}\ee^{-\ii n x}x^{-\frac{n^2}{2}},\quad x\to\infty.
\end{align*}
Therefore, the formula for sector $\frac{\pi}{2}<\arg(x)<\pi$ is consistent with sector $0<\arg(x)<\frac{\pi}{2}$ and again implies that the validity of our result can be extended to sector $0<\arg (x)<\pi$. That proves formula (\ref{eq:delta_asym_inf_upper}). 
\section{Computation of monodromy}\label{app:rhp representation}
For the reader's convenience, we provide the derivation of the monodromy data formulas \eqref{eq:monodromy_data}-\eqref{eq:monodromy_data2}. We start with the formulation of the corresponding Riemann-Hilbert problem.
\begin{rhp} \label{rhp:initial}
Fix the parameters $b_1, b_2,\alpha\in\C $ and $\re(x)>0$. We seek a $2 \times 2$ matrix function $\lambda\mapsto\mathbf \Psi(\lambda, x)$ satisfying:
\begin{enumerate}
\item[]\textbf{Analyticity:} $\mathbf \Psi (\lambda, x)$ is analytic in $\C \setminus \Gamma$, where $L = \{ | \lambda| = 1\} \cup \ii \R_{-}$ is the jump contour shown in Figure \ref{fig:1}. 
\item[] \textbf{Jump condition:} $\mathbf \Psi (\lambda, x)$ has continuous boundary values on $L \setminus \{0\}$ from each component of $\mathbb{C}\setminus L$, which satisfy
$\mathbf \Psi_+ (\lambda, x) = \mathbf \Psi_- (\lambda, x) \mathbf J(\lambda)$, where $\mathbf J(\lambda)$ is as shown in Figure \ref{fig:1} and where the subscript $+$ (resp., $-$) denotes a boundary value taken from the left (resp., right) of an arc of $L$. The parameters on Figure \ref{fig:1} have values
\begin{align}
&\mathbf{C}_{0\infty}=\begin{pmatrix}
        1&0\\2b_1&1
    \end{pmatrix},\quad e_\infty^{-2}=e_0^2=(-1)^n\ee^{\frac{\ii\pi\alpha}{2}},\\
    & \mathbf{S}_2^{\infty}=\begin{pmatrix}
        1&0\\2(b_1-b_2\ee^{\ii\pi\alpha})&1
    \end{pmatrix},\quad \mathbf{S}_2^{0}=\begin{pmatrix}
        1&0\\2\ee^{\ii\pi\alpha}(b_1-b_2)&1
    \end{pmatrix}.
\end{align}
\item[] \textbf{Normalization:} $\mathbf \Psi(\lambda,x)$ satisfies the asymptotic conditions
\begin{equation}
\label{eq:Psi-asymptotic-infinity_new}
\boldsymbol \Psi(\lambda, x) = \left( \mathbb{1} + \mathbf{\Psi}^\infty_1(x)\lambda^{-1} + \mathcal{O}(\lambda^{-2}) \right) \ee^{\ii x \lambda \sigma_3 /2} \lambda^{-\Theta_\infty  \sigma_3/2} \quad \mbox{as} \quad  \lambda \to \infty,
\end{equation}
and 
\begin{equation}
\label{eq:Psi-asymptotic-zero_new}
\boldsymbol \Psi(\lambda, x) = \left( \mathbf{\Psi}^0_0(x) + \mathbf{\Psi}^0_1(x)\lambda + \right) \ee^{-\ii x \lambda^{-1} \sigma_3 /2} \lambda^{\Theta_0 \sigma_3/2} \quad \mbox{as} \quad  \lambda \to 0, 
\end{equation}
with $\Theta_0=\frac{\alpha}{2}+n,\quad \Theta_\infty=n+2-\frac{\alpha}{2}$. The cut for power function goes along $\ii\mathbb{R}_-$.
\end{enumerate}
\end{rhp}
\begin{figure}
\begin{center}
\includegraphics[scale = 1]{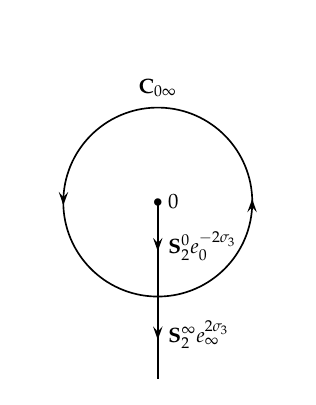}
\end{center}
\caption{The jump contour $L$ for $\mathbf{\Psi}(\lambda, x)$ and definition of $\mathbf J(\lambda)$ when $x>0$. }
\label{fig:1}
\end{figure}
For convenience of the reader we will go over connection of the above Riemann-Hilbert problem to the solution of Painlev\'e-III equation, see  \cite[Lemma 2]{BMS18} , \cite[Theorem 5.4]{FIKN}.
\begin{proposition} Fix the parameters $b_1, b_2,\alpha\in\C $ and $\re(x)>0$. Assume that the Riemann-Hilbert problem \ref{rhp:initial} is solvable for $x$ in some domain $ D\subset \mathbb{C}$.
The combination 
    \begin{equation}
    w_n(x,\alpha)=\frac{-\ii \mathbf{\Psi}^\infty_{1,12}(x)}{\mathbf{\Psi}^0_{1,11}(x)\mathbf{\Psi}^0_{1,12}(x)}
        \label{eq:u-recover-general}
    \end{equation}
solves Painlev\'e III equation \eqref{eq:v_n_painleve_equation}.
\end{proposition}
\begin{proof}
We observe the following expansion for $\lambda\to\infty$,
\begin{equation}
\begin{split}
\mathbf{A}(\lambda,x)&=\frac{\partial \mathbf{\Psi}}{\partial \lambda}\mathbf{\Psi}^{-1}=\frac{\ii x}{2}\sigma_3 + \left(\frac{\ii x}{2}\big[\mathbf{\Psi}_1^{\infty}(x),\sigma_3\big]-\frac{\Theta_{\infty}}{2}\sigma_3\right)\lambda^{-1}\\
&\quad\quad{}+\left(-\mathbf{\Psi}_1^{\infty}(x)-\frac{\Theta_{\infty}}{2}\big[\mathbf{\Psi}_1^{\infty}(x),\sigma_3\big]+\frac{\ii x}{2}\Big\{\big[\mathbf{\Psi}_2^{\infty}(x),\sigma_3\big]-\big[\mathbf{\Psi}_1^{\infty}(x),\sigma_3\big]\mathbf{\Psi}_1^{\infty}(x)\Big\}\right)\lambda^{-2}+\mathcal{O}(\lambda^{-3}),\\
\mathbf{B}(\lambda,x)&=\frac{\partial \mathbf{\Psi}}{\partial x}\mathbf{\Psi}^{-1}=\frac{\ii}{2}\sigma_3\lambda +\frac{\ii}{2}[\mathbf{\Psi}_1^{\infty}(x),\sigma_3] + \left(\mathbf{\Psi}_{1}^{\infty\prime}(x)+\frac{\ii}{2}\big[\mathbf{\Psi}_2^{\infty}(x),\sigma_3\big]-\frac{\ii}{2}\big[\mathbf{\Psi}_1^{\infty}(x),\sigma_3\big]\mathbf{\Psi}_1^{\infty}(x)\right)\lambda^{-1}\\+\mathcal{O}(\lambda^{-2}).
\end{split}
\label{eq:AnBn-large}
\end{equation}
Similarly, in the limit $\lambda\to 0$ we get
\begin{equation}
\begin{split}
\mathbf{A}(\lambda,x)&=\frac{\partial \mathbf{\Psi}}{\partial \lambda}\mathbf{\Psi}^{-1}=\frac{\ii x}{2}\mathbf{\Psi}_0^{0}(x)\sigma_3\mathbf{\Psi}_0^{0}(x)^{-1}\lambda^{-2}\\
&\quad{}+\left(\frac{\Theta_0}{2}\mathbf{\Psi}_0^{0}(x)\sigma_3\mathbf{\Psi}_0^{0}(x)^{-1}+
\frac{\ii x}{2}\mathbf{\Psi}_1^{0}(x)\sigma_3\mathbf{\Psi}_0^{0}(x)^{-1}-\frac{\ii x}{2}\mathbf{\Psi}_0^{0}(x)\sigma_3
\mathbf{\Psi}_0^{0}(x)^{-1}\mathbf{\Psi}_1^{0}(x)\mathbf{\Psi}_0^{0}(x)^{-1}\right)\lambda^{-1}\\
&\quad{}+\mathcal{O}(1)\\
\mathbf{B}(\lambda,x)&=\frac{\partial \mathbf{\Psi}}{\partial x}\mathbf{\Psi}^{-1}=-\frac{\ii}{2}\mathbf{\Psi}_0^{0}(x)\sigma_3\mathbf{\Psi}_0^{0}(x)^{-1}\lambda^{-1}+\mathbf{\Psi}_{0}^{0\prime}(x)\mathbf{\Psi}_0^0(x)^{-1}+\frac{\ii}{2}\big[\mathbf{\Psi}_0^0(x)\sigma_3\mathbf{\Psi}_0^0(x)^{-1},\mathbf{\Psi}_1^0(x)\mathbf{\Psi}_0^0(x)^{-1}\big]\\+\mathcal{O}(\lambda).
\end{split}
\label{eq:AnBn-small}
\end{equation}
Functions $\mathbf{A}(\lambda,x)$ and $\mathbf{B}(\lambda,x)$ are analytic with respect to $\lambda$ with only possible singularities at $\lambda=0$ and $\lambda=\infty$. The Liouville's theorem shows that $\mathbf{A}(\lambda,x)$ and $\mathbf{B}(\lambda,x)$ are Laurent polynomials:
\begin{equation}
\mathbf{A}(\lambda,x)=\frac{\ii x}{2}\sigma_3 +
\left(\frac{\ii x}{2}\big[\mathbf{\Psi}_1^{\infty}(x),\sigma_3\big]-\frac{\Theta_{\infty}}{2}\sigma_3\right)\lambda^{-1}
+ \frac{\ii x}{2}\mathbf{\Psi}_0^{0}(x)\sigma_3\mathbf{\Psi}_0^{0}(x)^{-1}\lambda^{-2}
\label{eq:An}
\end{equation}
and
\begin{equation}
\mathbf{B}(\lambda,x)=
\frac{\ii}{2}\sigma_3\lambda + \frac{\ii}{2}[\mathbf{\Psi}_1^{\infty}(x),\sigma_3] -\frac{\ii}{2}\mathbf{\Psi}_0^{0}(x)\sigma_3
\mathbf{\Psi}_0^{0}(x)^{-1}\lambda^{-1}.
\label{eq:Bn}
\end{equation}
Denote
\begin{equation}\label{e:5}
y(x)=-\ii x\mathbf{\Psi}_{1,12}^{\infty}(x),\quad v(x)=\ii x\mathbf{\Psi}_{1,21}^{\infty}(x),\quad s(x)=-x\mathbf{\Psi}_{0,11}^0(x)\mathbf{\Psi}_{0,12}^{0}(x),\quad t(x)=\frac{\mathbf{\Psi}_{0,21}^{0}(x)}{\mathbf{\Psi}_{0,11}^{0}(x)}
\end{equation}	
Using the identity $1=\det(\Psi(\lambda,x))=\det(\mathbf{\Psi}_{0}^{0}+\mathbf{\Psi}_{1}^{0}\lambda+\ldots)=\det(\mathbf{\Psi}_{0}^{0})=1$ we eliminate $\mathbf{\Psi}_{0,22}^{0}(x)$ and arrive to formulas 
\begin{equation}
\mathbf{A}(\lambda,x)=\frac{\ii x}{2}\sigma_3+\frac{1}{\lambda}\begin{bmatrix}-\frac{1}{2}\Theta_{\infty} & y\\ v & \frac{1}{2}\Theta_{\infty}\end{bmatrix}+\frac{1}{\lambda^2}\begin{bmatrix}\tfrac{1}{2}\ii x-\ii st & \ii s\\ -\ii t(st-x) & -\tfrac{1}{2}\ii x+\ii st\end{bmatrix},
\label{eq:lambda-eqn}
\end{equation}
and
\begin{equation}
\quad\mathbf{B}(\lambda,x)=\frac{\ii\lambda}{2}\sigma_3+\frac{1}{x}\begin{bmatrix}0 & y\\ v& 0\end{bmatrix}-\frac{1}{\lambda x}\begin{bmatrix}\tfrac{1}{2}\ii x-\ii st & \ii s\\ -\ii t(st-x) & -\tfrac{1}{2}\ii x+\ii st\end{bmatrix}.
\label{eq:x-eqn}
\end{equation}

The compatibility condition 
$$\frac{\partial \mathbf{A}}{\partial x}-\frac{\partial \mathbf{B}}{\partial \lambda} + [\mathbf{A},\mathbf{B}]=\mathbf{0}$$ for the simultaneous equations
\begin{align}
    \frac{\partial \mathbf{\Psi}}{\partial \lambda}(\lambda,x)=\mathbf{A}(\lambda,x)\mathbf{\Psi}(\lambda,x),\quad  \frac{\partial \mathbf{\Psi}}{\partial x}(\lambda,x)=\mathbf{B}(\lambda,x)\mathbf{\Psi}(\lambda,x).
\end{align}
is the first-order system of nonlinear differential equations
\begin{equation}
\begin{gathered}
	x\frac{\dd y}{\dd x}=-2xs+\Theta_{\infty}y,\ \ \ \ \ \ \ x\frac{\dd v}{\dd x}=-2xt(st-x)-\Theta_{\infty}v,
\\
	x\frac{\dd s}{\dd x}=(1-\Theta_\infty)s-2xy+4yst,\ \ \ \ \ \ \ x\frac{\dd t}{\dd x}=\Theta_\infty t-2yt^2+2v.
\end{gathered}
\label{eq:PIII-system}
\end{equation}
One can notice that expression
\begin{equation}
I=\frac{2\Theta_\infty}{x}st-\Theta_\infty-\frac{2}{x}yt(st-x)+\frac{2}{x}vs
\label{eq:Integral}
\end{equation}
is a conserved quantity, i.e, \eqref{eq:PIII-system} implies that $\frac{\dd I}{\dd x}=0$ holds identically.
%
%
Using \eqref{eq:PIII-system} one can show that the combination 
\begin{equation}
w_n(x)=-\frac{y(x)}{s(x)}
\label{eq:u-recover}
\end{equation}
satisfies the differential equation
\begin{equation}
	x\frac{\dd w_n}{\dd x}=2x-(1-2\Theta_{\infty})w_n+4stw_n^2-2xw_n^2.
\label{eq:first-order-PIII}
\end{equation}
Taking another $x$-derivative and letting $\Theta_0$ denote the constant value of the integral $I$ one then obtains
the Painlev\'e-III equation \eqref{eq:v_n_painleve_equation} after substitution $\Theta_0=\frac{\alpha}{2}+n,\quad \Theta_\infty=n+2-\frac{\alpha}{2}$.
\end{proof}

We would like to establish the following result.
\begin{proposition}\label{prop:monodromy}
    The special function solution $w_n(x,\alpha)=-\ii u_n(-2\ii x,\alpha)$ of the Painlev\'e-III equation \eqref{eq:v_n_painleve_equation} with $u_n(x,\alpha)$ described by \eqref{def:qn} is given by 
    \begin{equation}
    w_n(x,\alpha)=\frac{-\ii \mathbf{\Psi}^\infty_{1,12}(x)}{\mathbf{\Psi}^0_{1,11}(x)\mathbf{\Psi}^0_{1,12}(x)}.
    \end{equation}
    where $\mathbf{\Psi}^\infty_1(x), \mathbf{\Psi}^0_1(x)$ are coefficients in the asymptotic expansions \eqref{eq:Psi-asymptotic-zero_new}, \eqref{eq:Psi-asymptotic-infinity_new} of the solution $\mathbf{\Psi}(\lambda,x)$ of the Riemann-Hilbert problem \ref{rhp:initial}.
\end{proposition}
To prove this Proposition we will need two Lemmas. One will address the case $n=0$, the other will extend to general $n \in \mathbb{N}$.
\begin{lemma}\label{lem:init}
 The statement of Proposition \ref{prop:monodromy} holds for $n=0$.
\end{lemma}
\begin{proof}
  We start by considering $\mathbf{Y}(\lambda,x)=\mathbf{\Psi}(\lambda,x)\lambda^{\Theta_\infty\sigma_3/2}\ee^{\ii x\lambda^{-1}\sigma_3/2-\ii x\lambda\sigma_3/2}$. It satisfies the following Riemann-Hilbert problem.
    \begin{rhp} \label{rhp:initial_Y}
Fix the parameters $b_1, b_2,\alpha\in\C $ and $\re(x)>0$. We seek a $2 \times 2$ matrix function $\lambda\mapsto\mathbf Y(\lambda, x)$ satisfying:
\begin{enumerate}
\item[]\textbf{Analyticity:} $\mathbf Y (\lambda, x)$ is analytic in $\C \setminus \Gamma$, where $L = \{ | \lambda| = 1\} \cup \ii \R_{-}$ is the jump contour shown in Figure \ref{fig:2}. 
\item[] \textbf{Jump condition:} $\mathbf Y (\lambda, x)$ has continuous boundary values on $L \setminus \{0\}$ from each component of $\mathbb{C}\setminus L$, which satisfy
$\mathbf Y_+ (\lambda, x) = \mathbf Y_- (\lambda, x)  \widehat{\mathbf{J}}(\lambda,x)$, where $\widehat{\mathbf{J}}(\lambda,x)$ is as shown in Figure \ref{fig:2} and where the $+$ (resp., $-$) subscript denotes a boundary value taken from the left (resp., right) of an arc of $L$. The parameters on Figure \ref{fig:2} have values
\begin{align}
&\widehat{\mathbf{C}}_{0\infty}=\begin{pmatrix}
        1&0\\2b_1\lambda_-^{2-\alpha/2}\ee^{\ii x(\lambda^{-1}-\lambda)}&1
    \end{pmatrix},\\
    & \widehat{\mathbf{S}}_2^{\infty}=\begin{pmatrix}
        1&0\\2(b_1-b_2\ee^{\ii\pi\alpha})\lambda_-^{2-\alpha/2}\ee^{\ii x(\lambda^{-1}-\lambda)}&1
    \end{pmatrix},\quad \widehat{\mathbf{S}}_2^{0}=\begin{pmatrix}
        1&0\\2\ee^{\ii\pi\alpha}(b_1-b_2)\lambda_-^{2-\alpha/2}\ee^{\ii x(\lambda^{-1}-\lambda)}&1
    \end{pmatrix}.
\end{align}
\item[] \textbf{Normalization:} $\mathbf Y(\lambda,x)$ satisfies the asymptotic conditions
\begin{equation}
\label{eq:y-asymptotic-infinity}
\boldsymbol Y(\lambda, x) = \left( \mathbb{1} +\mathbf{\Psi}^\infty_1(x)\lambda^{-1}+\ii x\lambda^{-1}\sigma_3/2+ \mathcal{O}(\lambda^{-2}) \right) \quad \mbox{as} \quad  \lambda \to \infty,
\end{equation}
and 
\begin{equation}
\label{eq:y-asymptotic-zero}
\boldsymbol Y(\lambda, x) = \left( \mathbf{\Psi}^0_1(x) + \mathcal{O}(\lambda) \right)  \lambda^{ \sigma_3} \quad \mbox{as} \quad  \lambda \to 0. 
\end{equation}

\end{enumerate}
\end{rhp}
\begin{figure}
\begin{center}
\includegraphics[scale = 1]{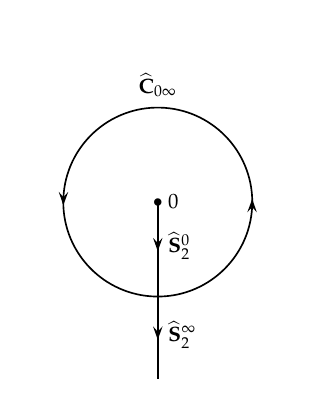}
\end{center}
\caption{The jump contour $L$ for $\mathbf{Y}(\lambda, x)$ and definition of $\mathbf{ \widehat{J}}(\lambda,x)$ when $x>0$. }
\label{fig:2}.
\end{figure}
We used above the fact that $\lambda_+^{\Theta_\infty\sigma_3/2}=\lambda_-^{\Theta_\infty\sigma_3/2}\ee^{{-\ii\pi\Theta_\infty\sigma_3}{}}$ along $\ii\mathbb{R}_-$. 

   We notice that the second column of $\mathbf{Y}(\lambda,x)$ has no jump, has a finite limit at infinity and has a simple pole at zero. Therefore, it has the form \begin{align}
        \mathbf{Y}_{12}(\lambda,x)=\frac{\mathrm{a}(x)}{\lambda},\quad \mathbf{Y}_{22}(\lambda,x)=1+\frac{\mathrm{b}(x)}{\lambda}.
    \end{align}
The jump condition on the first column of $\mathbf{Y}(\lambda,x)$ and the asymptotic condition at infinity imply that
\begin{align}
        \mathbf{Y}_{11}(\lambda,x)=1+\intop_{L}\frac{\mathbf{Y}_{12}(\mu,x)\widehat{\mathbf{J}}_{21}(\mu,x)}{\mu-\lambda}\dfrac{d \mu}{2\pi\ii},\quad \mathbf{Y}_{21}(\lambda,x)=\intop_{L}\frac{\mathbf{Y}_{22}(\mu,x)\widehat{\mathbf{J}}_{21}(\mu,x)}{\mu-\lambda}\dfrac{d \mu}{2\pi\ii}.
\end{align}

 Asymptotic condition at zero implies that $\mathbf{Y}_{11}(0,x)=0$, $\mathbf{Y}_{21}(0,x)=0$. That provides us with formulas for $\mathrm{a}(x)$ and $\mathrm{b}(x)$
   \begin{align}\label{def:a_b}
       \mathrm{a}(x)=-\left(\intop_L \widehat{\mathbf{J}}_{21}(\mu,x)\mu^{-2}\dfrac{d \mu}{2\pi\ii}\right)^{-1},\quad  \mathrm{b}(x)=-\frac{\intop_L \widehat{\mathbf{J}}_{21}(\mu,x)\mu^{-1}\dfrac{d \mu}{2\pi\ii}}{\intop_L \widehat{\mathbf{J}}_{21}(\mu,x)\mu^{-2}\dfrac{d \mu}{2\pi\ii}}.
   \end{align}
We would like to express $a(x)$ and $b(x)$ in terms of cylinder functions. We will use contour integral representations for $\mathcal{C}_\nu(-2\ii x)$ with $\re(x)>0$.  
\begin{align}\label{eq:rotated_contour_integral}
\mathcal{C}_\nu(-2\ii x)=\intop_{\Gamma_5\cup\Gamma_6}\frac{\ee^{-\ii x{}\left(t-\frac{1}{t}\right)}}{2 \pi \ii t^{1+\nu}}(2b_1\chi_{\Gamma_5}(t)-2b_2\chi_{ \Gamma_6}(t))dt.
\end{align}
 It can be obtained by rotation of contours $\Gamma_3$ and $\Gamma_4$ counterclockwise. Contours $\Gamma_5$ and $\Gamma_6$ can be found on Figure \ref{fig:contour_gamma56}.
\begin{figure*}[ht]
\begin{subfigure}[t]{0.4\textwidth}
\includegraphics{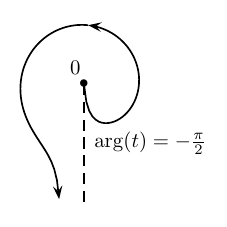}
\caption{Contour $\Gamma_5$}
        \end{subfigure}
        \begin{subfigure}[t]{0.4\textwidth}
\includegraphics{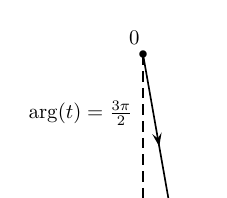}
\caption{Contour $\Gamma_6$}
    \end{subfigure}
    \caption{Contours of integration for $\Cy_\nu(-2\ii x)$.}
        \label{fig:contour_gamma56}
        \end{figure*}

Using the jump relation again $\lambda_+^{2-\frac{\alpha}{2}}=\lambda_-^{2-\frac{\alpha}{2}}\ee^{\ii\pi\alpha}$ along $\ii\mathbb{R}_-$ we notice that contour integral representation \eqref{eq:rotated_contour_integral} allows to rewrite
\begin{align}
    a(x)=-\frac{1}{\mathcal{C}_{\frac{\alpha}{2}-1}(-2\ii x)},\quad b(x)=-\frac{\mathcal{C}_{\frac{\alpha}{2}-2}(-2\ii x)}{\mathcal{C}_{\frac{\alpha}{2}-1}(-2\ii x)}.
\end{align}

On the next steps we would like to evaluate the fraction $\frac{-\ii \mathbf{\Psi}^\infty_{1,12}(x)}{\mathbf{\Psi}^0_{1,11}(x)\mathbf{\Psi}^0_{1,12}(x)}$. We have

 \begin{align}&\mathbf{\Psi}^0_{0,11}(x)=\left.\frac{\partial}{\partial \lambda}\mathbf{Y}_{11}(\lambda,x)\right|_{\lambda=0}=-\frac{\intop_L \widehat{\mathbf{J}}_{21}(\mu,x)\mu^{-3}\dfrac{d \mu}{2\pi\ii}}{\intop_L \widehat{\mathbf{J}}_{21}(\mu,x)\mu^{-2}\dfrac{d \mu}{2\pi\ii}}=-\frac{\mathcal{C}_{\frac{\alpha}{2}}(-2\ii x)}{\mathcal{C}_{\frac{\alpha}{2}-1}(-2\ii x)},\\&
\mathbf{\Psi}^0_{0,12}(x)=a(x),\quad \mathbf{\Psi}^\infty_{0,12}(x)=a(x).
\end{align}
Therefore
\begin{align}
\frac{-\ii \mathbf{\Psi}^\infty_{1,12}(x)}{\mathbf{\Psi}^0_{1,11}(x)\mathbf{\Psi}^0_{1,12}(x)}=\ii \frac{\intop_L \widehat{\mathbf{J}}_{21}(\mu,x)\mu^{-2}\dfrac{d \mu}{2\pi\ii}}{\intop_L \widehat{\mathbf{J}}_{21}(\mu,x)\mu^{-3}\dfrac{d \mu}{2\pi\ii}}=\ii\frac{\mathcal{C}_{\frac{\alpha}{2}-1}(-2\ii x)}{\mathcal{C}_{\frac{\alpha}{2}}(-2\ii x)}.
 \end{align}
According to Proposition \ref{prop:u_0}
\begin{align}
    w_0(x,\alpha)=-\ii u_0(-2\ii x,\alpha)=-\frac{\alpha}{4x}+\ii\frac{\mathcal{C}'_{\frac{\alpha}{2}}(-2\ii x)}{\mathcal{C}_{\frac{\alpha}{2}}(-2\ii x)}=\ii\frac{\mathcal{C}_{\frac{\alpha}{2}-1}(-2\ii x)}{\mathcal{C}_{\frac{\alpha}{2}}(-2\ii x)}
\end{align}
and we arrive at the desired result. 
\end{proof}
Define the Pauli-type matrices
\begin{equation*}
\widehat{\sigma}=\begin{bmatrix}1 & 0\\0 & 0\end{bmatrix}\quad\text{and}\quad
\widecheck{\sigma}=\begin{bmatrix}0 & 0 \\0 & 1\end{bmatrix}.
\end{equation*}
 We introduce the Schlesinger transformations by formulas 
\begin{equation}
{\mathbf{\Psi}}^{(1)}(\lambda,x)=(\widehat{\sigma}{\lambda}^{1/2}+{\mathbf{R}^{(1)}}(x){\lambda}^{-1/2})\mathbf{\Psi}(\lambda,x),
\label{eq:Schlesinger}
\end{equation}
\begin{equation}
{\mathbf{\Psi}}^{(3)}(\lambda,x)=(\widecheck{\sigma}{\lambda}^{1/2}+{\mathbf{R}^{(3)}}(x){\lambda}^{-1/2})\mathbf{\Psi}(\lambda,x),
\label{eq:Schlesinger3}
\end{equation}
where
\begin{equation}
\mathbf{R}^{(1)}(x)=\begin{bmatrix}\frac{\ii ty}{x}&-\frac{\ii y }{x}\\[0.1cm]
-t&1\end{bmatrix},
\quad\mathbf{R}^{(3)}(x)=\begin{bmatrix}1&-\frac{1}{t}\\[0.1cm]
\frac{\ii v}{x}&-\frac{\ii v}{t x}\end{bmatrix}.
\end{equation}
Let us rewrite equation \eqref{eq:v_n_painleve_equation} using parameters $\Theta_0$, $\Theta_\infty$.
\begin{align}
    \label{eq:v_n_painleve_equation_theta}
w''=\dfrac{\left( w'\right)^2}{w}-\dfrac{w'}{x}  + \dfrac{4\Theta_0w^2 + 4(1-\Theta_\infty)}{x}+4w^3-\frac{4}{w}.
\end{align}
We introduce the Bäcklund transformations 
\begin{equation}\label{eq:backlund1n}
            B_1: (w(x),\Theta_0,\Theta_\infty)\to \left(\frac{xw'(x)+2xw^2(x)+(-3+2\Theta_\infty)w(x)+2x}{w(x)(xw'(x)+2xw^2(x)+(1+2\Theta_0)w(x)+2x)},\Theta_0+1,\Theta_\infty-1\right)
        \end{equation}

\begin{equation}\label{eq:backlund3n}
            B_3: (w(x),\Theta_0,\Theta_\infty)\to \left(\frac{-xw'(x)-2xw^2(x)+(-1+2\Theta_\infty)w(x)+2x}{w(x)(xw'(x)+2xw^2(x)+(1+2\Theta_0)w(x)-2x)},\Theta_0+1,\Theta_\infty+1\right).
        \end{equation}

        Assume that $w(x)$ solves the Painlev\'e-III equation \eqref{eq:v_n_painleve_equation_theta} and denote \\$B_1(w(x),\Theta_0,\Theta_\infty)=(W(x),\Theta_0+1,\Theta_\infty-1)$. Then $W(x)$ solves Painlev\'e-III equation
\begin{equation*}
            W''(x)=\dfrac{\left( W'(x)\right)^2}{W(x)}-\dfrac{W'(x)}{x}  + \dfrac{4(\Theta_0+1) W^2(x) + 4(2-\Theta_\infty)}{x}+4W^3(x)-\frac{4}{W(x)}.
        \end{equation*}
Similarly, if we denote $B_3(w(x),\Theta_0,\Theta_\infty)=(W(x),\Theta_0+1,\Theta_\infty+1)$ then $W(x)$ solves Painlev\'e-III equation
\begin{equation*}
            W''(x)=\dfrac{\left( W'(x)\right)^2}{W(x)}-\dfrac{W'(x)}{x}  + \dfrac{4(\Theta_0+1) W^2(x) - 4\Theta_\infty}{x}+W^3(x)-\frac{1}{W(x)}.
        \end{equation*}
Using this notation we can see that 
\begin{align}
    B_3^n\left(w_0(x,\alpha),\frac{\alpha}{2},2-\frac{\alpha}{2}\right)= \left(w_n(x,\alpha),\frac{\alpha}{2}+n,2-\frac{\alpha}{2}+n\right).
\end{align}
\begin{lemma}\label{lem:schles}
Schlesinger transformations $\mathbf{\Psi}(\lambda,x)\to \mathbf{\Psi}^{(1)}(\lambda,x)$, $\mathbf{\Psi}(\lambda,x)\to \mathbf{\Psi}^{(3)}(\lambda,x)$
induce B\"acklund transformations $B_1$ and $B_3$ on solutions of the Painlev\'e-III equation \eqref{eq:v_n_painleve_equation_theta}. 
\end{lemma}
\begin{proof}
First, we notice that functions
$\mathbf{\Psi}^{(1)}(\lambda,x)$, $\mathbf{\Psi}^{(3)}(\lambda,x)$ have the same jumps as $\mathbf{\Psi}(\lambda,x)$ except for an additional sign change along $\ii \mathbb{R}_-$. Next, we look at the asymptotic behavior around infinity and get using \eqref{e:5} that 
\begin{equation}
\label{eq:Psi-asymptotic-infinity_new_back}
\boldsymbol \Psi^{(j)}(\lambda, x) = \left( \mathbb{1} + \mathbf{\Psi}^{\infty,(j)}_1(x)\lambda^{-1} + \mathcal{O}(\lambda^{-2}) \right) \ee^{\ii x \lambda \sigma_3 /2} \lambda^{-\Theta_\infty^{(j)}  \sigma_3/2} \quad \mbox{as} \quad  \lambda \to \infty,\quad j=1,3,
\end{equation}
\begin{equation}
\label{eq:Psi-asymptotic-zero_new_back}
\boldsymbol \Psi^{(j)}(\lambda, x) = \left( \mathbf{\Psi}^{0,(j)}_0(x) + \mathbf{\Psi}^{0,(j)}_1(x)\lambda + \right) \ee^{-\ii x \lambda^{-1} \sigma_3 /2} \lambda^{\Theta_0^{(j)} \sigma_3/2} \quad \mbox{as} \quad  \lambda \to 0, \quad j=1,3,
\end{equation}
where
    \begin{align}
    \Theta_\infty^{(1)}=\Theta_\infty-1,\quad  \Theta_0^{(1)}=\Theta_0+1,   \\
    \Theta_\infty^{(3)}=\Theta_\infty+1,\quad  \Theta_0^{(3)}=\Theta_0+1.
    \end{align}

  Identifying \eqref{eq:AnBn-large} with \eqref{eq:AnBn-small} we get
  \begin{align}
    &x\mathbf{\Psi}^{\infty}_{2,12} -{\ii y\mathbf{\Psi}^{\infty}_{1,22} }{}= -{s}{}+\frac{\Theta_\infty
   y}{x}-\frac{y}{x},\\
&x\mathbf{\Psi}^{\infty}_{2,21}+{\ii v \mathbf{\Psi}^{\infty}_{1,11} }{}=-{st^2}{}+xt-\frac{\Theta_\infty v}{x}-\frac{v}{x},\\
   &\frac{\mathbf{\Psi}^{0}_{1,21}-t\mathbf{\Psi}^{0}_{1,11}}{\mathbf{\Psi}^{0}_{0,11}}=\frac{\ii t y}{s}-\frac{\ii \Theta_0}{2 s}-\frac{\ii
   \Theta_\infty}{2 s},\\&
   \left({s\mathbf{\Psi}^{0}_{1,22}+x\mathbf{\Psi}^{0}_{1,12}-st\mathbf{\Psi}^{0}_{1,12}}\right){\mathbf{\Psi}^{0}_{0,11}}={\ii y}{}-\frac{\ii t y s}{x}-\frac{\ii \Theta_0 s}{2
   x}+\frac{\ii \Theta_\infty s}{2 x}.
  \end{align}
  Using this identities we can determine the following expressions appearing as result of Schlesinger transformation
  \begin{align}
    &  y^{(1)}=i s+\frac{\ii t y^2}{x}-\frac{\ii \Theta_\infty y}{x}+\frac{\ii
   y}{x},\quad v^{(1)}=-\ii x t,\\&
   s^{(1)}=\frac{\ii t y^3}{x s}-\frac{\ii \Theta_0 y^2}{2 x
   s}-\frac{\ii \Theta_\infty y^2}{2 x s}+\ii y,\\&
   t^{(1)}=\frac{\ii x (\Theta_0+\Theta_\infty-2 ty)}{y
   (\Theta_0+\Theta_\infty-2 t y)-2 x s},\\  &  v^{(3)}=-\ii s t^2-\frac{i v^2}{x t}+i x t-\frac{\ii \Theta_\infty
   v}{x}-\frac{\ii v}{x},\quad y^{(3)}=\frac{\ii x}{t},\\&
   s^{(3)}=\frac{\ii \Theta_0 x}{2 s t^2}+\frac{\ii \Theta_\infty x}{2
   s t^2}-\frac{\ii x y}{s t},\\&
   t^{(3)}=\frac{i v}{x}-\frac{2 \ii st^2}{\Theta_0+\Theta_\infty-2 ty}.
  \end{align}
  Remembering that $w=-\frac{y}{s}$ we get 
\begin{align}
  &  w^{(1)}=-\frac{2 s (x s+y (-\theta_\infty+t y+1))}{y (2 x
   s+y(-\Theta_0-\Theta_\infty+2 t y)},\\&
   w^{(3)}=-\frac{2 s t}{\Theta_0+\Theta_\infty-2 t y}.
\end{align}
Expressing $t$ in terms of $w'$ based on \eqref{eq:first-order-PIII} we arrive at \eqref{eq:backlund1n}, \eqref{eq:backlund3n}, as promised.

\end{proof}
\begin{proof}[Proof of Proposition \ref{prop:monodromy}]
Lemma \ref{lem:init} provides the monodromy data for the initial solution. Using Lemma \ref{lem:schles} we can see how it changes when we apply Schlesinger transformation or Bäcklund tranformation. The parameters $\Theta_0$ and $\Theta_\infty$ get appropriate shifts and the jump along the negative imaginary axis changes sign. That provides us with monodromy data for $w_n(x,\alpha)$.

\end{proof}
\section{Asymptotics at zero for \texorpdfstring{$\alpha\in 2\mathbb{Z}+\ii\mathbb{R}$}{alpha in 2Z+iR}}\label{app:special_asymptotics}
\subsection{Case of \texorpdfstring{$\alpha\in 2\mathbb{Z}+\ii\mathbb{R}$}{alpha in 2Z+iR}, \texorpdfstring{$\alpha\notin 2\mathbb{Z}$}{alpha not in 2Z}}\label{app:boundary_case}
We only need to consider the case of $\re(\alpha)=2n-2-4j$, $j=0,\ldots n-1$, the other values of $\alpha\in 2\mathbb{Z}+\ii\mathbb{R}$, $\alpha\notin 2\mathbb{Z}$ are covered by the Theorem \ref{thm:hankel_bessel_det_asymptotic}. The minimum of $\re(p(r,\alpha,n))$ is reached at $r_{min}=\frac{2n-\re(\alpha)}{4}=j+\frac{1}{2}$. Therefore there are two values of $r$ that produce the leading term of asymptotics: $r_c^{(1)}=j$ and $r_c^{(2)}=j+1$. Denote
\begin{align}\label{def:delta}
&\delta_{n,j}(\alpha,x)=(-1)^{\frac{n(n-1)}{2}+nj+n-j}\left( \frac{d_2}{\pi}\right)^n\left(\frac{d_1}{d_2}\sin\left(\frac{\pi\alpha}{2}\right)+\cos\left(\frac{\pi\alpha}{2}\right)\right)^j\\&\times\frac{G(n-j+1)G(-\frac{\alpha}{2}+n-j+1)G(j+1)G(j+\frac{\alpha}{2}+1)}{G(-\frac{\alpha}{2}+n-2j+1)G(\frac{\alpha}{2}-n+2j+1)}\left(\frac{x}{2}\right)^{(\alpha-2n+2j)j-\frac{n\alpha}{2}}
\end{align}
We arrive at the following result.
\begin{theorem}\label{thm:hankel_bessel_det_asymptotic_ext_imag_alpha}
 The Toeplitz determinant  \eqref{def:hankel_bessel_determinant} admits the following  $x \rightarrow 0$, $-\pi<\arg(x)<\pi$ asymptotics for fixed $d_1,d_2\in \mathbb{C}$, $n\in\mathbb{N}\cup\{0\}$, $\alpha\in\mathbb{C}\setminus (2\mathbb{Z})$, $d_2\neq 0$, $d_1\sin\left(\frac{\pi\alpha}{2}\right)+d_2\cos\left(\frac{\pi\alpha}{2}\right)\neq 0$,  and $Re(\alpha)=2n-4j-2\text{ for some }j=0, 1, 2 \ldots,n-1$, \begin{align*}&\Delta_n(x,\alpha)\sim \delta_{n,j}(\alpha,x)+\delta_{n,j+1}(\alpha,x), \text{ as }x\to 0,\quad -\pi<\arg(x)<\pi.\end{align*}
\normalsize
where $\delta_{n,j}(\alpha,x)$ is given by \eqref{def:delta}.\\
\end{theorem}

It implies the following asymptotics for the solution $u_n(x,\alpha)$.

\begin{theorem}\label{thm:q_n_asymptotic_imag_alpha} Solution \eqref{eq:q_n_formula} of the Painlev\'e-III equation \eqref{eq:q_n_painleve_equation}  admits the following  $x \rightarrow 0$, $-\pi<\arg(x)<\pi$ asymptotics for fixed $d_1,d_2\in \mathbb{C}$, $n\in\mathbb{N}\cup\{0\}$, $\alpha\in\mathbb{C}\setminus (2\mathbb{Z})$
\begin{enumerate}
    \item If $d_2\neq 0$, $d_1\sin\left(\frac{\pi\alpha}{2}\right)+d_2\cos\left(\frac{\pi\alpha}{2}\right)\neq 0$,   and $Re(\alpha)=2n-4j-2 \text{ for some }\\j=0, 1,\ldots,n-1$, then \begin{align*}&u_n(x,\alpha)\sim (-1)^{n} \left(\frac{d_1}{d_2}\sin\left(\frac{\pi\alpha}{2}\right)+\cos\left(\frac{\pi\alpha}{2}\right)\right)^{-1}\left(\frac{\Gamma(\frac{\alpha}{2}-n+2j+1)}{\Gamma(-\frac{\alpha}{2}+n-2j)}\right)^2\\&\times\frac{\Gamma(-\frac{\alpha}{2}+n-j+1)\Gamma(n-j)}{\Gamma(j+\frac{\alpha}{2}+1)\Gamma(j+1)}\left(\frac{x}{2}\right)^{-\alpha+2n-4j-1}+\\
    &(-1)^{n}\left(\frac{d_1}{d_2}\sin\left(\frac{\pi\alpha}{2}\right)+\cos\left(\frac{\pi\alpha}{2}\right)\right)\left(\frac{\Gamma(-\frac{\alpha}{2}+n-2j-1)}{\Gamma(\frac{\alpha}{2}-n+2j+2)}\right)^2\\& \times\frac{\Gamma(j+\frac{\alpha}{2}+1)\Gamma(j+2)}{\Gamma(-\frac{\alpha}{2}+n-j)\Gamma(n-j)}\left(\frac{x}{2}\right)^{\alpha-2n+4j+3}+\frac{x (-\alpha +2 n+2)}{(\alpha +4
   j-2 n+2)^2},\text{ as }x\to 0,\quad -\pi<\arg(x)<\pi.\end{align*}
   \item If $d_2\neq 0$, $d_1\sin\left(\frac{\pi\alpha}{2}\right)+d_2\cos\left(\frac{\pi\alpha}{2}\right)\neq 0$,   and $Re(\alpha)=2n+2$, then 
   \begin{align*}
&u_n(x,\alpha)\sim(-1)^n\left(\frac{d_1}{d_2}\sin\left(\frac{\pi\alpha}{2}\right)+\cos\left(\frac{\pi\alpha}{2}\right)\right)\frac{ \Gamma \left(\frac{\alpha }{2}\right) \Gamma
   \left(n-\frac{\alpha }{2}+1\right)  }{ \Gamma (n+1) \Gamma \left(\frac{\alpha
   }{2}-n\right)^2}\left(\frac{x}{2}\right)^{\alpha -2 n-1}+\frac{x}{2 n+2-\alpha},\\&\text{ as }x\to 0,\quad -\pi<\arg(x)<\pi.
   \end{align*}
    \item If $d_2\neq 0$, $d_1\sin\left(\frac{\pi\alpha}{2}\right)+d_2\cos\left(\frac{\pi\alpha}{2}\right)\neq 0$,   and $Re(\alpha)=2n-4j \text{ for some }\\j=0, 1,\ldots,n-1$, then \begin{align*}&\left(u_n(x,\alpha)\right)^{-1}\sim (-1)^{n} \left(\frac{d_1}{d_2}\sin\left(\frac{\pi\alpha}{2}\right)+\cos\left(\frac{\pi\alpha}{2}\right)\right)^{-1}\left(\frac{\Gamma(\frac{\alpha}{2}-n+2j)}{\Gamma(-\frac{\alpha}{2}+n-2j+1)}\right)^2\\&\times\frac{\Gamma(-\frac{\alpha}{2}+n-j+1)\Gamma(n-j+1)}{\Gamma(j+\frac{\alpha}{2})\Gamma(j+1)}\left(\frac{x}{2}\right)^{-\alpha+2n-4j+1}+\\
    &(-1)^{n}\left(\frac{d_1}{d_2}\sin\left(\frac{\pi\alpha}{2}\right)+\cos\left(\frac{\pi\alpha}{2}\right)\right)\left(\frac{\Gamma(-\frac{\alpha}{2}+n-2j)}{\Gamma(\frac{\alpha}{2}-n+2j+1)}\right)^2\\& \times\frac{\Gamma(j+\frac{\alpha}{2}+1)\Gamma(j+1)}{\Gamma(-\frac{\alpha}{2}+n-j+1)\Gamma(n-j)}\left(\frac{x}{2}\right)^{\alpha-2n+4j+1}-\frac{x (\alpha +2 n)}{(\alpha +4 j-2 n)^2},\text{ as }x\to 0,\quad -\pi<\arg(x)<\pi.\end{align*}
 \item If $d_2\neq 0$, $d_1\sin\left(\frac{\pi\alpha}{2}\right)+d_2\cos\left(\frac{\pi\alpha}{2}\right)\neq 0$,  and $Re(\alpha)=-2n$, then 
   \begin{align*}
&\left(u_n(x,\alpha)\right)^{-1}\sim(-1)^n\left(\frac{d_1}{d_2}\sin\left(\frac{\pi\alpha}{2}\right)+\cos\left(\frac{\pi\alpha}{2}\right)\right)^{-1}\frac{ \Gamma \left(1-\frac{\alpha }{2}\right) \Gamma
   \left(n+\frac{\alpha }{2}\right)  }{ \Gamma (n+1) \Gamma \left(1-n-\frac{\alpha
   }{2}\right)^2}\left(\frac{x}{2}\right)^{1 -2 n-\alpha}\\&-\frac{x}{\alpha+2n},\text{ as }x\to 0,\quad -\pi<\arg(x)<\pi.
   \end{align*}
   
\end{enumerate}
where $\Gamma(x)$ refers to the Gamma function.
\end{theorem}

\begin{proof}
    Using Theorem \ref{thm:hankel_bessel_det_asymptotic_ext_imag_alpha} we see that \begin{enumerate}
        \item For $\re(\alpha)=2n-4j-2,\quad j=0\ldots n-1$
        \begin{align*}
         u_n(x,\alpha)\sim   -\frac{(\delta_{n,j}(\alpha,x )+\delta_{n,j+1}(\alpha,x )) (\delta_{n+1,j+1}(\alpha -2,x)+\delta_{n+1,j+2}(\alpha -2,x))}{\delta_{n,j+1}(\alpha -2,x) \delta_{n+1,j+1}(\alpha,x)},\quad x\to 0.
        \end{align*}
          \item For $\re(\alpha)=2n+2$
        \begin{align*}
         u_n(x,\alpha)\sim  -\frac{\delta_{n,0}(\alpha,x ) (\delta_{n+1,0}(\alpha -2,x)+\delta_{n+1,1}(\alpha -2,x))}{\delta_{n,0}(\alpha -2,x) \delta_{n+1,0}(\alpha ,x)},\quad x\to 0.
        \end{align*}
                   \item For $\re(\alpha)=2n-4j,\quad j=0\ldots n-1$
        \begin{align*}
         u_n(x,\alpha)\sim   -\frac{\delta_{n,j}(\alpha,x ) \delta_{n+1,j+1}(\alpha -2,x)}{(\delta_{n,j}(\alpha -2,x)+\delta_{n,j+1}(\alpha -2,x)) (\delta_{n+1,j}(\alpha,x )+\delta_{n+1,j+1}(\alpha,x ))},\quad x\to 0.
        \end{align*}
           \item For $\re(\alpha)=-2n$
        \begin{align*}
         u_n(x,\alpha)\sim   -\frac{\delta_{n,n}(\alpha,x ) \delta_{n+1,n+1}(\alpha -2,x)}{\delta_{n,n}(\alpha -2,x) (\delta_{n+1,n}(\alpha,x )+\delta_{n+1,n+1}(\alpha,x ))},\quad x\to 0.
        \end{align*}
    \end{enumerate}
    After simplification we get the desired formulas.
\end{proof}
\subsection{Case of \texorpdfstring{$\alpha \in 2\mathbb{Z}$}{alpha in 2Z}, \texorpdfstring{$d_2\neq 0$}{nonzero d2}}\label{app:int_alpha_nonzero_d2}

One of the ways to obtain the desired asymptotics is to consider the limit $\im (\alpha)\to 0$ in the asymptotics from the previous section. Using \cite[\href{http://dlmf.nist.gov/5.17.E4}{(5.17.4)}]{DLMF} and the definition of digamma function $\psi(x)=\frac{\Gamma'(x)}{\Gamma(x)}$ we get from Theorem \ref{thm:hankel_bessel_det_asymptotic_ext_imag_alpha}.

\begin{theorem}\label{thm:hankel_bessel_det_asymptotic_ext_int_alpha_nonzero_d2}
 The Toeplitz determinant  \eqref{def:hankel_bessel_determinant} admits the following  $x \rightarrow 0$, $-\pi<\arg(x)<\pi$ asymptotics for fixed $d_1,d_2\in \mathbb{C}$, $n\in\mathbb{N}\cup\{0\}$, 
 $d_2\neq 0$,  and $\alpha=2n-4j-2 \text{ for some }j=0, 1, 2 \ldots,n-1$
\begin{align*}
   & \Delta_n(x,\alpha)\sim-\frac{d_2^{n}}{\pi^n}(-1)^{\frac{n(n-1)}{2}}\left(\frac{x}{2}\right)^{-2j-2j^2+n+2jn-n^2}G(j+1)G(j+2)G(n-j)G(n-j+1)\\&\times 
   \left(2 \ln(x)+\frac{d_1}{d_2}\pi +4\gamma+\psi(j+1)+\psi(n-j)-2\ln(2)\right ),\text{ as }x\to 0,\quad -\pi<\arg(x)<\pi.
\end{align*}
\normalsize
where $G(x)$ refers to the Barnes $G$-function, $\psi(x)$ is the digamma function, and $\gamma$ is the Euler-Mascheroni constant.\\
\end{theorem}
Similarly by taking limit $\im(\alpha)\to 0$ in Theorem \ref{thm:q_n_asymptotic_imag_alpha} we can get asymptotics for the solution.
\begin{theorem}\label{thm:q_n_asymptotic_int_alpha_nonzero_d2} Solution \eqref{eq:q_n_formula} of the Painlev\'e-III equation \eqref{eq:q_n_painleve_equation}  admits the following  $x \rightarrow 0$, $-\pi<\arg(x)<\pi$ asymptotics for fixed $d_1,d_2\in \mathbb{C}$, $n\in\mathbb{N}\cup\{0\}$, $d_2\neq 0$ 
\begin{enumerate}
    \item If $\alpha=2n-4j-2 \text{ for some }j=-1, 0,\ldots,n-1$, then \begin{align*}&u_n(x,\alpha)\sim  -\frac{x}{2}\left(2 \ln(x)+\frac{d_1}{d_2}\pi +4\gamma+\psi(j+2)+\psi(n-j)-2\ln(2)\right )\\&\times\left[(j+1)\left(2 \ln(x)+\frac{d_1}{d_2}\pi +4\gamma+\psi(j+2)+\psi(n-j)-2\ln(2)\right )-1\right],\text{ as }x\to 0,\quad -\pi<\arg(x)<\pi.\end{align*}
    \item If $\alpha=2n-4j \text{ for some }j=0, 1,\ldots,n$, then \begin{align*}&(u_n(x,\alpha))^{-1}\sim -\frac{x}{2}\left(2 \ln(x)+\frac{d_1}{d_2}\pi +4\gamma+\psi(j+1)+\psi(n-j+1)-2\ln(2)\right )\\&\times\left[(j-n)\left(2 \ln(x)+\frac{d_1}{d_2}\pi +4\gamma+\psi(j+1)+\psi(n-j+1)-2\ln(2)\right )+1\right],\text{ as }x\to 0,\quad -\pi<\arg(x)<\pi,
   \end{align*}
   where $\psi(x)$ is the digamma function and $\gamma$ is the Euler-Mascheroni constant.
\end{enumerate}
\end{theorem}

\subsection{Case of \texorpdfstring{$\alpha \in 2\mathbb{Z}$}{alpha in 2Z}, \texorpdfstring{$d_2= 0$}{zero d2}}\label{app:int_alpha_zero_d2}

For such values of parameters the contour integral representation \eqref{eq:tau-formula} takes form
 \begin{equation}
\Delta_n(x,\alpha)=\frac{d_1^n}{n!}(-1)^{\frac{n(n-1)}{2}}\intop_{\Gamma_1}\ldots \intop_{\Gamma_1}\prod_{1\leq j<k\leq n}(t_k-t_j)^2\prod_{k=1}^{n} \frac{\ee^{\frac{x}{2}\left(t_k-\frac{1}{t_k}\right)}}{2 \pi \ii t_k^{n+\frac{\alpha}{2}}}dt_k.\label{eq:tau-formula_zero_d2}.
        \end{equation}
    Let us assume $\alpha\geq0$. We make a change of variables $t_k=\frac{2}{x}s_k$ and deform the contour $\Gamma_1$ back to the original.    If we used different scaling change of variables, the coefficient next to the leading term would vanish. We get
         \begin{equation}
\Delta_n(x,\alpha)\sim\frac{d_1^n}{n!}(-1)^{\frac{n(n-1)}{2}}\left(\frac{x}{2}\right)^{\frac{n\alpha}{2}}\intop_{\Gamma_1}\ldots \intop_{\Gamma_1}\prod_{1\leq j<k\leq n}(s_k-s_j)^2\prod_{k=1}^{n} \frac{\ee^{s_k}}{2 \pi \ii s_k^{n+\frac{\alpha}{2}}}ds_k 
        \end{equation}
For $\alpha\notin 2\mathbb{Z}$ we already computed the integral above earlier, see \eqref{eq:H1_formula} with $r_c=n$. It is equal to the following expression
\begin{align*}
&\intop_{\Gamma_1}\ldots \intop_{\Gamma_1}\prod_{1\leq j<k\leq n}(s_k-s_j)^2\prod_{k=1}^{n} \frac{\ee^{s_k}}{2 \pi \ii s_k^{n+\frac{\alpha}{2}}}ds_k=\frac{(-1)^{n^2}}{\pi^n}\left(\sin\left(\frac{\pi\alpha}{2}\right)\right)^n G(n+2)\prod_{j=0}^{n-1}\Gamma\left(1-\frac{\alpha}{2}-n+j\right)\\&
=(-1)^{\frac{n(n-1)}{2}} \frac{G(n+2)}{\prod_{j=0}^{n-1}\Gamma\left(\frac{\alpha}{2}+n-j\right)}=(-1)^{\frac{n(n-1)}{2}} \frac{G(n+2)G(\frac{\alpha}{2}+1)}{G(\frac{\alpha}{2}+n+1)}
\end{align*}
This identity is valid for $\alpha\in 2\mathbb{Z}$, $\alpha\geq 0$
 as well by taking the limit of both sides of the equation.

For $\alpha\leq0$, $\alpha \in 2\mathbb{Z}$ we make the change of variables $t_k=\frac{x}{2}s_k$ and deform the contour to $\Gamma_2$. We get
 \begin{equation}
\Delta_n(x,\alpha)\sim\frac{d_1^n}{n!}(-1)^{\frac{n(n+1)}{2}}\left(\frac{x}{2}\right)^{-\frac{n\alpha}{2}}\intop_{\Gamma_2}\ldots \intop_{\Gamma_2}\prod_{1\leq j<k\leq n}(s_k-s_j)^2\prod_{k=1}^{n} \frac{\ee^{-\frac{1}{s_k}}}{2 \pi \ii s_k^{n+\frac{\alpha}{2}}}ds_k 
        \end{equation}
      Again the expression for the integral was computed earlier for $\alpha \notin 2\mathbb{Z}$, see \eqref{eq:H2_formula} with $r_c=0$.  As the result we get
        \begin{align*}
&\intop_{\Gamma_2}\ldots \intop_{\Gamma_2}\prod_{1\leq j<k\leq n}(s_k-s_j)^2\prod_{k=1}^{n} \frac{\ee^{-\frac{1}{s_k}}}{2 \pi \ii s_k^{n+\frac{\alpha}{2}}}ds_k=\frac{(-1)^{n}}{\pi^n}\ee^{-\ii\pi n\frac{\alpha}{2}}\left(\sin\left(\frac{\pi\alpha}{2}\right)\right)^n G(n+2)\\&\times\prod_{j=0}^{n-1}\Gamma\left(1+\frac{\alpha}{2}-n+j\right)
=(-1)^{\frac{n(n+1)}{2}} \ee^{-\ii\pi n\frac{\alpha}{2}}\frac{G(n+2)}{\prod_{j=0}^{n-1}\Gamma\left(-\frac{\alpha}{2}+n-j\right)}\\&=(-1)^{\frac{n(n+1)}{2}} \ee^{-\ii\pi n\frac{\alpha}{2}}\frac{G(n+2)G(-\frac{\alpha}{2}+1)}{G(-\frac{\alpha}{2}+n+1)}
\end{align*}
Again this identity is valid for $\alpha\in 2\mathbb{Z}$, $\alpha\leq 0$
 as well by taking the limit of both sides of the equation. Therefore we arrive at the following result.
\begin{theorem}\label{thm:hankel_bessel_det_asymptotic_ext_int_alpha_zero_d2}
 The Toeplitz determinant \eqref{def:hankel_bessel_determinant} admits the following $x \rightarrow 0$, $-\pi<\arg(x)<\pi$ asymptotics for fixed $d_1\in \mathbb{C}$, $n\in\mathbb{N}\cup\{0\}$, 
 $d_2= 0$,  and $\alpha\in 2\mathbb{Z} $.
 \begin{enumerate}
     \item If $\alpha\geq 0$, then
\begin{align*}
   & \Delta_n(x,\alpha)\sim{d_1^{n}} \frac{G(n+1)G(\frac{\alpha}{2}+1)}{G(\frac{\alpha}{2}+n+1)}{}\left(\frac{x}{2}\right)^{\frac{n\alpha}{2}},\text{ as }x\to 0,\quad -\pi<\arg(x)<\pi.
\end{align*}
    \item If $\alpha\leq 0$, then
\begin{align*}
   & \Delta_n(x,\alpha)\sim{d_1^{n}} \ee^{\frac{\ii\pi n\alpha}{2}}\frac{G(n+1)G(-\frac{\alpha}{2}+1)}{G(-\frac{\alpha}{2}+n+1)}{}\left(\frac{x}{2}\right)^{-\frac{n\alpha}{2}},\text{ as }x\to 0,\quad -\pi<\arg(x)<\pi.
   \end{align*}
 \end{enumerate}

\normalsize
Here $G(x)$ refers to the Barnes $G$-function.\\
\end{theorem}
This result is known in the literature, see \cite[(5.7)]{forrester_witte_boundary_conditions}.  This implies the following asymptotics for the solution.
\begin{theorem}\label{thm:q_n_asymptotic_int_alpha_zero_d2} Solution \eqref{eq:q_n_formula} of the Painlev\'e-III equation \eqref{eq:q_n_painleve_equation}  admits the following  $x \rightarrow 0$, $-\pi<\arg(x)<\pi$ asymptotics for fixed $d_1\in \mathbb{C}$, $n\in\mathbb{N}\cup\{0\}$, $d_2=0$ , $\alpha \in 2\mathbb{Z}$
\begin{enumerate}
    \item If $\alpha>0$, then \begin{align*}&u_n(x,\alpha)\sim \left({-\frac{\alpha}{2}-n}{}\right)\left(\frac{x}{2}\right)^{-1},\text{ as }x\to 0,\quad -\pi<\arg(x)<\pi.\end{align*}
   \item If  $\alpha\leq 0$, then 
   \begin{align*}
&u_n(x,\alpha)\sim \left(\frac{2}{2+2n-\alpha}\right) \frac{x}{2},\text{ as }x\to 0,\quad -\pi<\arg(x)<\pi.
   \end{align*}
   
\end{enumerate}
\end{theorem}
It is interesting to observe that we could have obtained Theorems \ref{thm:hankel_bessel_det_asymptotic_ext_int_alpha_zero_d2}, \ref{thm:q_n_asymptotic_int_alpha_zero_d2} from Theorems \ref{thm:hankel_bessel_det_asymptotic}, \ref{thm:q_n_asymptotic} by taking the limit $\im(\alpha)\to 0$ only for $\alpha\geq 0$, but not for $\alpha<0$.
\end{document}